\documentclass[twoside,11pt]{article}

\usepackage{jmlr2e}

\usepackage{textcomp}

\usepackage[utf8]{inputenc} 
\usepackage[T1]{fontenc}    
\usepackage{hyperref}       
\usepackage{url}            
\usepackage{booktabs}       
\usepackage{amsfonts}       
\usepackage{nicefrac}       
\usepackage{microtype}      
\usepackage{xcolor}         
\usepackage{algpseudocode}
\usepackage{algorithm}

\usepackage{ulem}
\usepackage{cancel}

\usepackage[english]{babel}
 \usepackage{tikz}
\usepackage{amsmath}
 \newcommand\numberthis{\addtocounter{equation}{1}\tag{\theequation}}
 \usepackage{stmaryrd}
 \usepackage{graphicx}
 \usepackage{enumitem}
 \usepackage{hyperref}
 \usepackage{cancel}

 \usepackage{setspace}
 \usepackage{amsfonts,amssymb,times}

\usepackage{lastpage}
\jmlrheading{25}{2024}{1-\pageref{LastPage}}{4/23; Revised 4/24}{8/24}{23-0493}{El Mehdi Achour, François Malgouyres, and Sébastien Gerchinovitz}

\newcommand{\RCLRP}{rank-constrained linear regression problem}
\newcommand{\Ss}{\mathcal{S}}
\newcommand{\Wbf}{\textbf{W}}

\newtheorem{assumptions}{Assumption}
\DeclareMathOperator{\tr}{tr}
\DeclareMathOperator{\rk}{rk}
\DeclareMathOperator{\col}{col}
\DeclareMathOperator{\Ker}{Ker}
\DeclareMathOperator*{\argmin}{arg\,min}

\begin{document}
\title{The Loss Landscape of Deep Linear Neural Networks: a Second-order Analysis}

 \author{\name El Mehdi Achour 
\email achour@mathc.rwth-aachen.de \\
       \addr Department of Mathematics \\ RWTH Aachen University \\ Aachen, Germany
       \AND
       \name François Malgouyres \email francois.malgouyres@math.univ-toulouse.fr \\
       \addr Institut de Mathématiques de Toulouse ; UMR 5219\\
          Université de Toulouse ; CNRS \\
          UPS IMT F-31062 Toulouse Cedex 9, France\\
       \AND
       \name Sébastien Gerchinovitz \email sebastien.gerchinovitz@irt-saintexupery.fr \\
       \addr Institut de Recherche Technologique Saint Exupéry, Toulouse, France
       }

\editor{Martin Jaggi}
\ShortHeadings{Loss Landscape of Deep Linear Networks}{Achour, Malgouyres, and Gerchinovitz}

\maketitle

\begin{abstract}%
We study the optimization landscape of deep linear neural networks with square loss. It is known that, under weak assumptions, there are no spurious local minima and no local maxima. However, the existence and diversity of non-strict saddle points, which can play a role in first-order algorithms' dynamics, have only been lightly studied. We go a step further with a complete analysis of the optimization landscape at order $2$. Among all critical points, we characterize global minimizers, strict saddle points, and non-strict saddle points. We enumerate all the associated critical values. The characterization is simple, involves conditions on the ranks of partial matrix products, and sheds some light on global convergence or implicit regularization that has been proved or observed when optimizing linear neural networks. In passing, we provide an explicit parameterization of the set of all global minimizers and exhibit large sets of strict and non-strict saddle points.

\end{abstract}


\sloppy
\section{Introduction}

Deep learning has been widely used recently due to its good empirical performances in image recognition, natural language processing, and speech recognition, among other fields. 
However, there is still a gap between theory and practice.
One of the aspects that are partially missing in the picture is why gradient-based algorithms can achieve low training error despite a non-convex objective.
Another partially open question is why they generalize well to unseen data despite many more parameters than the number of points in the training set, and how implicit regularization can help. One important research direction analyses the landscape of the empirical risk.
In this paper, we characterize  the local structures around critical points of the empirical risk, for deep linear neural networks with the square loss.

Before summarizing the related literature and our main contributions, we first recall definitions that will be key throughout the paper.

\subsection{Reminder: Minimizers, Critical Points of Order 1 or 2, Strict and Non-strict Saddle Points}

Let us recall the definitions of local structures of the landscape of the empirical risk, which are important from the statistical and optimization points of view.

For $\pmb{w} \in \mathbb{R}^n$, denote by $\pmb{w} \longmapsto L ({\pmb{w}})$ the function we want to minimize.
Assume that $\pmb{w} \longmapsto L ({\pmb{w}})$ is $C^2$, and denote by $\nabla L $ and $\nabla^2 L $ its gradient and its Hessian.\footnote{When the input parameter is not a vector, but, e.g., a sequence of matrices, the same definitions hold, where the gradient and the Hessian are computed with respect to the vectorized version of the input parameters.}
We also write $A \succeq 0$ to say that a matrix $A \in \mathbb{R}^{n \times n}$ is positive semi-definite.
Recall the following four definitions, which are nested:
\begin{itemize}
\item[$\bullet$] { $\pmb{w^*}$ is a \bf global minimizer} if and only if  $\forall \pmb{w} \in \mathbb{R}^n$, $\ L (\pmb{w}^{*}) \leq L ({\pmb{w}})$.
	
\item[$\bullet$] { $\pmb{w^*}$ is a \bf local minimizer} if and only if there exists a neighbourhood $\mathcal{O} \subset \mathbb{R}^n$ of $\pmb{w^*}$ such that 
	$\forall \pmb{w}\in\mathcal{O} , \ L ({\pmb{w}}^{*}) \leq L ({\pmb{w}})$.
	
\item[$\bullet$] { $\pmb{w^*}$ is a \bf second-order critical point} if and only if	$\nabla L ({\pmb{w}}^{*})=0\ \mbox{and}\ \nabla^2 L ({\pmb{w}}^{*})\succeq 0$. If, on the contrary, the Hessian has a negative eigenvalue, we say that the point has a negative curvature.

\item[$\bullet$] { $\pmb{w^*}$ is a \bf first-order critical point} if and only if $\nabla L ({\pmb{w}}^{*})=0$.\\

We can also distinguish a specific type of first-order critical point: saddle points.
As discussed below, they can be second-order critical points or not.\footnote{\textcolor{black}{Defining the \emph{index} of a critical point as the number of negative eigenvalues of its Hessian, we can equivalently define strict saddle points as saddle points of index greater than or equal to $1$. Similarly, non-strict saddle points are saddle points of index $0$. Note that the latter are \emph{degenerate}, i.e., their Hessian is singular.}}

\item[$\bullet$] { $\pmb{w^*}$ is a \bf saddle point} if and only if it is a first-order critical point which is neither a local minimizer nor a local maximizer.
\begin{itemize}
    \item { A saddle point $\pmb{w^*}$ is \bf strict} if and only if it is not a second-order critical point (i.e., the Hessian $\nabla^2 L ({\pmb{w}}^{*})$ has a negative eigenvalue). Figure \ref{fig:strict saddle} gives an example.
    \item { A saddle point $\pmb{w^*}$ is \bf non-strict} if and only if it is a second-order critical point. In that case, the Hessian $\nabla^2 L ({\pmb{w}}^{*})$ is positive semi-definite and has at least one eigenvalue equal to zero.
    Typically, in the direction of the corresponding eigenvectors, a higher-order term makes it a saddle point (e.g., $L(\pmb{w}) = \sum_{i=1}^n w_i^3$ at $\pmb{w}^{*}=0)$.
    Figure \ref{fig:plateau} gives an example. 
\end{itemize}
\end{itemize}

\begin{figure}
    \begin{minipage}[b]{0.45\textwidth}
    \centering
    \includegraphics[width=4cm]{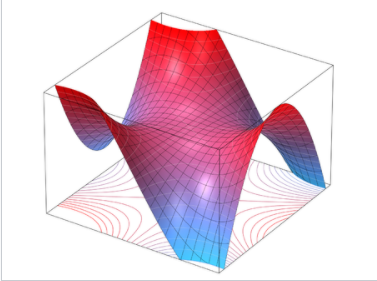}
    \caption{Example of a landscape with a plateau (non-strict saddle point).}
    \label{fig:plateau}
    \end{minipage}
    \hfill
    \begin{minipage}[b]{0.45\textwidth}
    \centering
    \includegraphics[width=4cm]{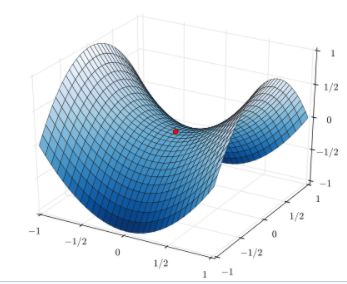}
    \caption{Example of a landscape with a strict saddle point at (0,0).}
    \label{fig:strict saddle}
    \end{minipage}
\end{figure}

\subsection{On the Importance of a Landscape Analysis at Order 2}
\label{sec:intro-importanceordre2}

When the function we are trying to minimize is 
smooth,
convex, and has a global minimizer, the gradient descent algorithm with a well-chosen learning rate converges to a first-order critical point, which is a global minimizer \citep{nesterov1998introductory}. However, in general, finding a global optimum of a non-convex function is an NP-complete problem \citep{murty1987some}; this is, in particular, the case for a simple 3-node neural network \citep{blum1989training}. Despite that, when optimizing neural networks, the current practice is still to use gradient-based algorithms. 

It has been known for decades that, even in the non-convex setting, for large classes of functions, gradient-based algorithms converge to a first-order critical point, in the sense that the iterates produced by the algorithm reach an arbitrary small gradient after a finite (polynomial) number of iterations \citep{nesterov1998introductory}. Recent works have shown that classical first-order algorithms
escape strict saddle points \citep{DBLP:conf/colt/LeeSJR16,lee2019first}. Well-chosen algorithms can be stopped  at an output with arbitrarily small gradient and nearly-positive semi-definite Hessian in polynomial time \citep{jin2017escape,jin2018accelerated,daneshmand2018escaping,jin2021nonconvex,gadat2022asymptotic}. Higher order algorithms, designed to escape strict saddle points, have been constructed and have a faster convergence  \citep[e.g.,][]{adolphs2019local,JMLR:v24:20-608}. However, nothing prevents these algorithms to spend many epochs in the vicinity of non-strict saddle points.  This results in a long plateau during training.

To see that this behavior actually occurs in practice, consider the simple experiment whose results are shown in Figures~\ref{fig:eg loss near saddle point} and~\ref{fig:histo escape points} (more details in Appendix~\ref{sec experiments}). For each run of this experiment, the parameters of a linear neural network of depth $5$ are optimized to fit random input/output pairs. The discrepancy is measured with the square loss and we use the ADAM optimizer. Depending on the run, the algorithm is initialized in the vicinity either of a strict saddle point (in red) or a non-strict saddle point (in blue). The distance between the random initial iterate and the saddle point is purposely not negligible: it is fixed to around $10\%$ of the norm of the saddle point. Figure~\ref{fig:eg loss near saddle point} shows the typical loss evolution for both cases. We can see that ADAM rapidly escapes from the strict saddle point but needs many epochs to escape the plateau in the vicinity of the non-strict saddle point. Figure \ref{fig:histo escape points} shows that this observation generalizes to most runs. We compare the empirical distributions of a random time (called \emph{escape epoch}) defined as the epoch at which the loss has significantly decreased from its initial value. When initialized in the vicinity of non-strict saddle points, the algorithm suffers from an often large escape epoch and might be stopped there, without the possibility to distinguish this non-strict saddle point from a global minimum. Improving the analysis beyond local minimizers and characterizing strict and non-strict saddle points are therefore key to understanding gradient descent dynamics and implicit regularization.

\begin{figure}
    \begin{minipage}[b]{0.45\textwidth}
    \centering
    \includegraphics[width=6cm]{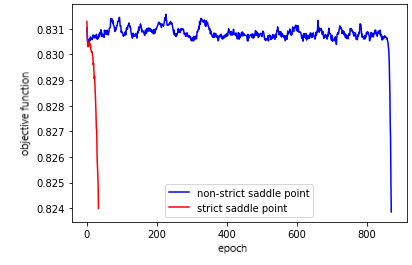}
    \caption{The loss function during the iterative process,  when initialized around a strict saddle point (in red) or a non-strict  saddle point (in blue).}
    \label{fig:eg loss near saddle point}
    \end{minipage}
    \hfill
    \begin{minipage}[b]{0.45\textwidth}
    \centering
    \includegraphics[width=7cm]{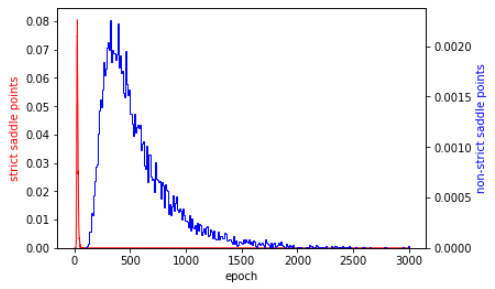}
    \caption{Histogram of escape epochs,  when initialized around a strict (in red) or a non-strict saddle point (in blue). For clarity, the $y$-axis is endowed with two scales. The right axis corresponds to the blue curve and the left to the red one.}
    \label{fig:histo escape points}
    \end{minipage}
   
\end{figure}

\subsection{Related Works on Linear Networks}
\label{sec:relatedwork}

Despite the fact that they are rarely used to solve real-world applications\footnote{They indeed compute a linear map between the input and output spaces.}, many recent works have focused on linear neural networks.
These studies are motivated by the fact that the empirical risk of linear networks is nonconvex and shares similar properties with practical nonlinear neural networks.
\textcolor{black}{Indeed, as shown in \cite{saxe2014exact}, linear networks exhibit nonlinear learning phenomena similar to those seen during the optimization of nonlinear networks, including long plateaus followed by rapid transitions to lower error solutions. Also, the implicit regularization phenomena observed for nonlinear networks \citep{safran2022effective,timor2023implicit,jacot2022implicit,marion2024implicit,belkin2021fit,bartlett2021deep} occurs also for linear networks (see the paragraph on this topic below). Studying these phenomena for linear networks is a good starting point for rigorous work.}

The study of linear neural networks can be divided into two categories. The first line of research studies the geometric landscape of the empirical risk. The second line studies the trajectory of gradient descent dynamics in linear networks. Our work falls into the first category. \\

 \textbf{Geometric landscape for linear networks:}
This first started with \cite{Baldi:1989:NNP:70359.70362}.
They proved that for a $1$-hidden layer linear network, under some conditions on the data matrices, and for the square loss, every local minimizer is a global minimizer.
\cite{kawaguchi2016deep} later generalized and extended this result to deep linear neural networks under mild conditions and again proved that every local minimizer is a global minimizer (this part has been proved later by \cite{lu2017depth} with weaker assumptions on the data and simpler proofs).
This author also proved that every other critical point is a saddle point, that for a $1$-hidden layer linear network all saddle points are strict, while for deeper networks, there exist non-strict saddle points (\cite{kawaguchi2016deep} exhibits a space of non-strict saddle points where all but one weight matrix are equal to zero).
\cite{yun2018global} gave a condition for a critical point to be either a global minimizer or a saddle point.
\cite{zhou2018critical} removed all assumptions on the data and gave analytical forms for the critical points of the empirical risk.
In the characterization, the weight matrices are defined recursively and can be found by solving equations; in particular, they gave a characterization of global minimizers.
\cite{DBLP:conf/iclr/NouiehedR18} showed  using assumptions only on the 
width of the layers
that every local minimizer is a global minimizer.
They prove that this assumption on the architecture is sharp in the sense that without it, and if we do not make assumptions on the data matrices as in previous works, then there exists a poor local minimizer.
\cite{zhu2020global} used assumptions only on the input data matrix, to prove that for a $1$-hidden layer linear network, every local minimizer is a global minimizer and every other critical point has a negative curvature.
\cite{LaurentB18} proved for different general convex losses that, under assumptions on the architecture, all local minima are global.
Finally, \cite{Trager2020Pure} and \cite{mehta2021loss} used results from algebraic geometry to give other properties about critical points of linear networks.

Most of the previous works focus on local minimizers.
None of these works provide simple necessary and sufficient conditions for a saddle point to be strict or not.\footnote{By ``simple'', we mean an easier-to-exploit condition than just looking at the smallest eigenvalue of the Hessian.} In particular, in the case of more than two hidden layers, only very specific examples of non-strict saddle points were described. Furthermore, global minimizers were characterized but not explicitly parameterized. See Section~\ref{Comparison with previous works} for more details. \\

 \textbf{Gradient dynamics and implicit regularization for linear networks:}
In this line of research,
authors study the dynamics of first-order algorithms for linear networks, which they sometimes combine with results about the loss landscape.
\cite{arora2018a} proved that gradient descent converges to a global minimum at a linear rate, under assumptions on the width of the layers, the initial iterate, and the loss at initialization.
Other works also proved similar results with different assumptions \citep{Eftekhari20,bartlett2018gradient,wu2019global}. 
However, as noted by \cite{shamir2019exponential}, these works consider strong assumptions on the loss at initialization.
Indeed, \cite{shamir2019exponential} gave a negative result on a
deep linear network of width $1$, by proving that for standard initializations, gradient descent can take exponential time to converge to the global minimizer.
The author also provided empirical examples of the same phenomenon happening
for larger widths.
On the other hand, \cite{du2019width} proved that if the layers are wide enough, convergence to a global minimimizer
can be achieved in polynomial time using a classical data-independent random Gaussian initialization (known as Xavier initialization). The required minimum width of the network depends on the norm of a global minimizer of the linear regression problem. As we will see in Section~\ref{Comparison with previous works} this global convergence result can be re-interpreted in terms of the loss landscape at order~$2$.

On a similar line of research, \cite{chitour2018geometric} proved using assumptions on the architecture of the network and the data matrices that gradient flow almost surely converges to a global minimizer for a $1$-hidden layer linear network.
Later, \cite{Bah_2021} proved the same result under weaker assumptions.
They also proved that, in deep linear networks, the gradient flow almost surely converges to global minimizers of the \RCLRP. This has been extended to gradient descent in \cite{nguegnang2021convergence}.
In \cite{jacot2021saddle}, the authors conjecture that, for deep linear networks, the gradient flow initialized randomly in the vicinity of the origin, asymptotically exhibits a saddle-to-saddle dynamics, where the rank of the linear map increases at each new saddle. 

This is related to another consequence of the landscape properties: implicit regularization.
\cite{arora2019implicit} showed that, for matrix recovery, deep linear networks converge to low-rank solutions  even when all the hidden layers are of size larger than or equal to the input and output sizes.
\cite{razin2020nipsimplicit} proved that, in deep matrix factorization, implicit regularization may not be explainable by norms, as all norms may go to infinity.
They rather suggest seeing implicit regularization as a minimization of the rank.
\cite{saxe2019mathematical} and \cite{gidel2019implicit} proved with different assumptions on the data and a vanishing
initialization that both gradient flow and discrete gradient dynamics sequentially learn solutions of a \RCLRP~with a gradually increasing rank.
Finally, \cite{gissin2019implicit} proved for a toy model that this incremental learning happens more often (with larger initialization), when the depth of the network increases.
As we will see in Section \ref{Comparison with previous works}, these results can be re-interpreted in the light of the landscape at order 2.

\subsection{Summary of our Contributions}
\label{sec:summarycontributions}
Our contributions on the optimization landscape of deep linear networks can be summarized as follows.
\begin{itemize}
    \item[$\bullet$] We characterize the square loss landscape of deep linear networks at order $2$ (see Theorem \ref{main theorem} and Figure \ref{fig:Classification of p.c}). That is, under some classical and weak assumptions on the data, we characterize, among all first-order critical points, which are global minimizers, strict saddle points, and non-strict saddle points. The characterization is simple and involves conditions on the ranks of partial matrix products. To the best of our knowledge, this is the first simple, necessary and sufficient condition that differentiates strict saddle points from non-strict saddle points. 
    
    \item[$\bullet$] Several results follow from the characterization: under the same assumptions,
    \begin{itemize}
        \item we first immediately recover the fact that all saddle points are strict for one-hidden layer linear networks;
        \item more importantly, for deeper networks, when proving that all cases considered in the characterization can indeed occur, we exhibit large sets of strict and non-strict saddle points (see Proposition \ref{existence of tightened and non tightened critical points} and its proof in Appendix  \ref{proof of existence of tightened and non tightened critical points});
        \item we show that the non-strict saddle points are associated with $r_{max}$ plateau values of the empirical risk, where $r_{max}$ is the size of the thinnest layer of the network (see Theorem \ref{main theorem}). Typically these are values of the empirical risk that first-order algorithms can take for some time, as in Figure~\ref{fig:eg loss near saddle point}, and which might be confused with a global minimum.
    \end{itemize}

    \item[$\bullet$] As a by-product of our analysis, we obtain explicit parameterizations of sets containing or included in the set of all first-order critical points (see Propositions~\ref{simplif mat} and~\ref{reciproque param}). We also derive an explicit parameterization  of the set of all global minimizers (see Proposition \ref{prop parameterization of global minimizers}).

    \end{itemize}
    
    The above results are compared in details with previous works in Section \ref{Comparison with previous works}. In particular, our second-order characterization sheds some light on two phenomena:
    \begin{itemize}
        \item[$\bullet$] Implicit regularization: we recover the fact that every non-strict saddle point corresponds to a global minimizer of the \RCLRP, as shown in \cite[Proposition~35]{Bah_2021}.  Our characterization additionally shows that only a fraction of the  critical points corresponding to rank-constrained solutions are non-strict saddle points. The others are strict saddle points. Given the differences in the behavior of first-order algorithms in the vicinity of strict and non-strict saddle points as illustrated on Figures~\ref{fig:eg loss near saddle point} and~\ref{fig:histo escape points}, our results open new research directions related to the very nature of implicit regularization and its stability.
        \item[$\bullet$] Our characterization can also be useful to understand recent global convergence results in terms of the loss landscape at order $2$. In particular, we show how to re-interpret a proof of \cite{du2019width} to see that gradient descent with Xavier initialization on wide enough deep linear networks meets no non-strict saddle points on its trajectory.
    \end{itemize}

\subsection{Outline of the Paper}
The paper is organized as follows. We define the setting in Section~\ref{Settings} and state our results in Section~\ref{main contributions}. We prove our main result (Theorem~\ref{main theorem}) in Section~\ref{sketch of proof}. More precisely, we detail the proof structure and main arguments but defer all technical derivations to the appendix. We finally conclude our work in Section~\ref{sec:conclusion}.

Most technical details can be found in the appendix, which is organized as follows.
Section \ref{Appendix Notation and general useful lemmas and properties} contains additional notation and lemmas that will be useful in all subsequent sections. In Section~\ref{Appendix lemmas for first-order critical points} we provide proofs of propositions and lemmas related to first-order critical points, while Section~\ref{Appendix Parameterization of first-order critical points and global minimizers} gathers the proofs for the parameterization of first-order critical points and global minimizers.
Sections~\ref{proof main prop partie baldi}, \ref{Appendix subtle strict saddles}, and~\ref{Appendix non strict saddles} contain proofs corresponding to each subsection of Section~\ref{sketch of proof}. Finally, in Section~\ref{sec experiments}, we describe in more details the illustrative experiment underlying Figures~\ref{fig:eg loss near saddle point} and~\ref{fig:histo escape points}.

\section{Setting}\label{Settings}

In this section we formally define our setting (deep linear networks with square loss), set some notation, and describe our assumptions on the data. \\

 \textbf{Model and notation:} We consider a fully-connected linear neural network of depth $H \geq 2$.
The neural network consists of $H$ layers and maps any input $x \in \mathbb{R}^{d_x}$ to an output $W_H \cdots W_1 x \in \mathbb{R}^{d_y}$, where $W_{H} \in \mathbb{R}^{d_{y} \times d_{H-1}},  \ldots  ,W_{h} \in \mathbb{R}^{d_{h} \times d_{h-1}}, \ldots ,W_{1} \in \mathbb{R}^{d_{1} \times d_{x}}$, are the matrices associated with the $H$ layers ($d_h$ is the width of layer $h$).
We set $d_H = d_y$ and $d_0 = d_x$.
The input layer is of size $d_x$ and the output layer is of size $d_y$.
We also define the smallest width of the layers as $r_{max} = \min(d_H, \ldots ,d_0)$.\footnote{The notation $r_{max}$ comes from the fact that it is the maximum possible rank of the product $W_H \cdots W_1$.}
We denote the parameters of the model by $\Wbf = (W_H,\ldots,W_1)$. \\

Let $\left(x_{i}, y_{i}\right)_{i=1 . . m}$ with $x_{i} \in \mathbb{R}^{d_{x}}$ and $y_{i} \in \mathbb{R}^{d_{y}}$, be the training set that we gather column-wise in matrices $X \in \mathbb{R}^{d_{x} \times m}$ and $Y \in \mathbb{R}^{d_y \times m}$.
We consider the empirical risk $L$ defined by:
$$L(\Wbf)= \sum_{i=1}^{m}\left\|W_{H} W_{H-1} \cdots W_{2} W_{1}x_i - y_i \right\|_{2}^{2}=\|W_H \cdots W_1 X-Y\|^{2} \;,$$
where $\|.\|_2$ is the Euclidean norm and $\|.\| $ denotes the Frobenius norm of a matrix.\\
We set:
$$\Sigma_{X X}=\sum_{i=1}^{m} x_{i} x_{i}^{T}=X X^T \in \mathbb{R}^{d_{x} \times d_{x}} \quad , \quad \Sigma_{Y Y}=\sum_{i=1}^{m} y_{i} y_{i}^{T} = Y Y^T \in \mathbb{R}^{d_{y} \times d_{y}},$$
$$\Sigma_{X Y}=\sum_{i=1}^{m} x_{i} y_{i}^{T} = X Y^T \in \mathbb{R}^{d_{x} \times d_{y}} \quad , \quad \Sigma_{Y X}=\sum_{i=1}^{m} y_{i} x_{i}^{T}=Y X^T \in \mathbb{R}^{d_{y} \times d_{x}},$$
where, $A^T$ denotes the transpose of $A$.\\

\begin{assumptions} \label{Assump H} 
Throughout the article, we assume that $d_y \leq d_x \leq m$, that $\Sigma_{XX}$ is invertible, and that $\Sigma_{XY}$ is of full rank $d_y$.
We define $\Sigma^{1/2} = \Sigma_{YX} \Sigma_{X X}^{-1} X \in \mathbb{R}^{d_{y} \times m}$ and $\Sigma=\Sigma^{1/2}(\Sigma^{1/2})^T = \Sigma_{Y X} \Sigma_{X X}^{-1} \Sigma_{X Y} \in \mathbb{R}^{d_{y} \times d_{y}}$.
We assume that the singular values of $\Sigma^{1/2}$ are all distinct (i.e., that $\Sigma$ has $d_y$ distinct eigenvalues). \\
\end{assumptions}
These assumptions are exactly the ones considered in \cite{kawaguchi2016deep}.
Note that we do not make any assumption on the width of the hidden layers.
As noted by \cite{Baldi:1989:NNP:70359.70362}, full-rank
matrices are dense, and deficient-rank matrices are of measure 0.
In general, $m \geq d_x \geq d_y$, which is the classical learning regime, is essentially sufficient to have the other assumptions verified, due to the randomness of the data.  \\
Let 
\begin{align} \label{svd de sigma 1/2}
    \Sigma^{1/2} = U \Delta V^T
\end{align}
be a singular value decomposition of $\Sigma^{1/2}$, where $U \in \mathbb{R}^{d_y \times d_y}$ and $V \in \mathbb{R}^{m \times m}$ are orthogonal, and the diagonal elements of $\Delta \in \mathbb{R}^{d_y \times m}$ are in decreasing order.\\
Since $\Sigma=\Sigma^{1/2}(\Sigma^{1/2})^T$,
$\Sigma$ can be diagonalized as $\Sigma = U \Lambda U^T$ where $\Lambda = \text{diag}(\lambda_1,\ldots,\lambda_{d_y})$, with $\lambda_1 > \cdots > \lambda_{d_y} \geq 0$.
Moreover, a consequence of Assumption \ref{Assump H} is that $\Sigma$ is positive definite (see Lemma \ref{sigma invertible}); therefore, we have $\lambda_{d_y}>0$.\\

 \textbf{Additional notation:} We list below some notation and conventions that will be used throughout the paper.\\
For all integers $a\leq b$, we denote by  $ \llbracket a,b \rrbracket$ the set of integers between $a$ and $b$ (including $a$ and $b$).
If $a>b$, $ \llbracket a,b \rrbracket$ is the empty set (e.g. $\llbracket1,0 \rrbracket = \emptyset$).\\
If $\Ss = \emptyset$, then $\sum_{i \in \Ss} \lambda_i = 0$. \\
Given a matrix $A \in \mathbb{R}^{p \times q}$, $\col(A)$, $\Ker(A)$ and $\rk(A)$, denote respectively the column space, the null space and the rank of $A$. \\
For a matrix $A \in \mathbb{R}^{p \times q}$, we write $A_i \in \mathbb{R}^{p}$ for the $i$-th column of $A$ and $A_{\mathcal{J}} \in \mathbb{R}^{p \times |\mathcal{J}|}$ for the sub-matrix obtained by concatenating the column vectors $A_i$, for $i \in \mathcal{J}$.
The identity matrix of size $p$ will be denoted by $I_p$.
\\
When we write $W_h \cdots W_{h'}$ for $h > h'$, the expression denotes the product of all $W_j$ from $j=h$ to $j=h'$.
To simplify later developments, we allow two additional cases:
when $h = h'$, the expression simply denotes $W_h$,
and when $h' = h + 1$, it stands for the identity matrix $I_{d_h} \in \mathbb{R}^{d_h \times d_h}$.

Considering submatrices of compatible sizes, we define a block matrix by one of the three following ways:
\begin{itemize}
    \item $[A , B]$ is the horizontal concatenation of the matrices $A$ and $B$;
    \item $\begin{bmatrix}
        G \\ 
        H
        \end{bmatrix}$ is the vertical concatenation of $G$ and $H$;
    \item $\left[
        \begin{array}{c  c}
        C & D \\
        E & F
        \end{array}
        \right]$ is a 2 $\times$ 2 block matrix.
\end{itemize}

By convention, in block matrices, some blocks can have 0 lines or 0 columns; this means that such blocks do not exist.
However if we have a product between two matrices that have 0 as the common size (the number of columns for the first matrix, of the lines for the second matrix), then their product equals a zero matrix, of the right size.
More formally, if $A \in \mathbb{R}^{n \times 0 }$ and $B \in \mathbb{R}^{0 \times p}$, then, by convention, $AB = 0_{n \times p}$.
Note that the product of block matrices is compatible with this convention (e.g.,
$ [A \ , \ B] \begin{bmatrix}
C \\ D
\end{bmatrix} = AC + BD $ is still true if $B \in \mathbb{R}^{n \times 0}$ and $D \in \mathbb{R}^{0 \times p}$).\\
Further notation that are used in the appendix can be found at the beginning of Appendix \ref{Appendix Notation and general useful lemmas and properties}.

\section{Main Results}\label{main contributions}

In this section, we state the main results of this paper.
We start with a necessary condition for being a first-order critical point of $L$ (Proposition \ref{global map and critical values}), to which we give a light reciprocal (Proposition \ref{reciproque S vers W}).
We then move to our main result (Theorem \ref{main theorem}), which is a second-order classification of all first-order critical points. It distinguishes between global minimizers, strict saddle points and non-strict saddle points.
Finally, the third result is a necessary parameterization for critical points (Proposition \ref{simplif mat}) and an explicit parameterization of all global minimizers (Proposition \ref{prop parameterization of global minimizers}).
These results are compared with previous works in Section~\ref{Comparison with previous works}.
All the proofs can be found in Section \ref{sketch of proof} or in the appendix, where most technical derivations are deferred.

\subsection{First-order Critical Points: Preliminary Results}

In the next proposition, we restate in our framework a necessary condition for being a first-order critical point, which was already present in \cite{Baldi:1989:NNP:70359.70362} and most of the papers in this line of research.
This proposition will serve later to distinguish between different types of critical points.

\begin{proposition}[Global map and critical values]\label{global map and critical values}
    Suppose Assumption \ref{Assump H} in Section \ref{Settings} holds true.
    Let $\Wbf = (W_H, \ldots , W_1)$ be a first-order critical point of $L$ and set $r = \text{rk}(W_H \cdots W_1) \in \llbracket 0,r_{max} \rrbracket$.\\
    There exists a unique subset $\Ss \subset\llbracket 1,d_y \rrbracket$ of size $r$ such that:
 $$W_H \cdots W_1 = U_{\Ss} U_{\Ss}^T \Sigma_{Y X} \Sigma_{X X}^{-1},$$
 where $U$ was defined in \eqref{svd de sigma 1/2}.
 We say that the critical point $\Wbf$ is \emph{associated with $\Ss$}.
 The associated critical value is 
 $$\qquad L(\Wbf) = \tr(\Sigma_{YY}) - \sum_{i \in \Ss} \lambda_i.$$
\end{proposition}

The proof can be found in Appendix \ref{proof of global map and critical values}.
The result is true even for $r=0$, using the conventions from Section~\ref{Settings} (in this case, $\Ss = \emptyset$). \\
Note that $\Sigma_{Y X} \Sigma_{X X}^{-1}$ corresponds to the solution of the classical linear regression problem.
Therefore, we can see that for every critical point $\Wbf$ of $L$, the product $W_H \cdots W_1$ is the projection of this least-squares estimator onto a subspace generated by a subset of the eigenvectors of $\Sigma$.
Note that $\tr(\Sigma_{YY}) = \|Y\|^2$. \\

The following proposition is a light reciprocal to Proposition \ref{global map and critical values}, by showing that all subsets $\Ss$ and the corresponding critical values $\tr(\Sigma_{YY}) - \sum_{i \in \Ss} \lambda_i$ are associated to an existing critical point.
In particular, the largest critical value is reached for $\Ss=\emptyset$ and the smallest critical value for $\Ss=\llbracket 1,r_{max} \rrbracket$.
\begin{proposition}\label{reciproque S vers W}
    Suppose Assumption \ref{Assump H} in Section \ref{Settings} holds true.
    For any $\Ss \subset \llbracket 1,d_y \rrbracket$ of size $r \in \llbracket 0,r_{max} \rrbracket$, there exists a first-order critical point $\Wbf$ associated with~$\Ss$. 
\end{proposition}
The proof of Proposition \ref{reciproque S vers W} is deferred to Appendix \ref{proof of reciproque S vers W}. \textcolor{black}{The proof uses Proposition \ref{reciproque param}, which is proved in Appendix \ref{proof reciproque param}, before Appendix \ref{proof of reciproque S vers W}.}

\subsection{Second-order Classification of the Critical Points of \texorpdfstring{$L$}{L}}

The main result of this section is Theorem~\ref{main theorem} below, where we classify all first-order critical points into global minimizers, strict saddle points and non-strict saddle points.
To state Theorem \ref{main theorem} we first need to introduce some definitions. \\

Let $\Wbf = (W_H, \ldots , W_1)$ be a first-order critical point of $L$.
Below, we introduce the notions of \emph{complementary block}, \emph{tightened pivot} and \emph{tightened critical point} that are key to the main results.
Consider the sequence of $H$ matrices $W_H, \ldots, W_2,W_1$ and connect them by plugging $\Sigma_{XY} $ between $W_1$ and $W_H$ so as to form a cycle as on Figure \ref{fig:cycle with 2 complementary blocks}.
Note that the dimensions of these matrices allow us to consider any product of consecutive matrices on this cycle, e.g., $W_H W_{H-1} W_{H-2}$ or $W_2 W_1 \Sigma_{XY} W_H$ (the matrix $\Sigma_{XY}$ between $W_1$ and $W_H$ is key here). Such products of consecutive matrices in the cycle are what we call "\textbf{blocks}".
In the sequel, we call "\textbf{pivot}" any pair of indices $(i,j) \in \llbracket 1 , H \rrbracket$, with $i > j$, and we consider blocks around a pivot $(i,j)$, as defined formally below.
\begin{definition}[Complementary blocks]
    Let $\Wbf = (W_H, \ldots , W_1)$ be a first-order critical point of $L$. \\
    For any pivot $(i,j) \in \llbracket 1 , H \rrbracket$, ($i>j$), we define the two complementary blocks to $(i,j)$ as:
   $$W_{j-1} \cdots W_1 \Sigma_{XY} W_H \cdots W_{i+1} \qquad \text{and} \qquad W_{i-1} \cdots W_{j+1}.$$    
\end{definition}
The general case is represented on Figure \ref{fig:cycle with 2 complementary blocks}. \\
Note that, when $i=j+1$, the second complementary block is $W_j W_{j+1}$, which using the convention in Section~\ref{Settings} is $I_{d_j}$.
Similarly, if $i=H$ and $j=1$, the first complementary block is $\Sigma_{XY}$.
First we state a proposition about the ranks of the complementary blocks which is key to our analysis.  
\begin{proposition}\label{all blocks ranks geq than r}
    Suppose Assumption \ref{Assump H} in Section \ref{Settings} holds true.
    Let $\Wbf = (W_H, \ldots , W_1)$ be a first-order critical point of $L$ and $r= \rk(W_H \cdots W_1)$.
    For any pivot $(i,j)$, the rank of each of the two complementary blocks is larger than or equal to $r$.
\end{proposition}
The proof is in Appendix \ref{proof of all blocks ranks geq than r}.
The boundary case when at least one of the two ranks is equal to $r$ plays a special role in the loss landscape at order 2.

\begin{definition}[Tightened pivot]\label{def pivot tightned}
    Let $\Wbf = (W_H, \ldots , W_1)$ be a first-order critical point of $L$ and let $r= \rk(W_H \cdots W_1)$. \\
    We say that a pivot $(i,j)$ is \textbf{tightened} if and only if at least one of the two complementary blocks to $(i,j)$ is of rank $r$.
\end{definition}

\begin{definition} [Tightened critical point] \label{def W tightned}
    Let $\Wbf = (W_H, \ldots, W_1)$ be a first-order critical point of $L$.
    We say that $\Wbf$ is tightened if and only if every pivot $(i,j)$ is tightened.
\end{definition}
\textcolor{black}{When $H \geq 3$}, note that a sufficient condition for a first-order critical point $\Wbf$ to be tightened is the existence of three weight matrices $W_{h_1}$, $W_{h_2}$ and $W_{h_3}$ of rank $r=\rk(W_H \cdots W_1)$. This is a simple intuition on tightened critical points, that the reader can keep in mind when reading the article. \textcolor{black}{A special case of this is when $\Wbf$ is $0$-balanced (Definition 1 in \cite{arora2018a}), that is, when $W_{j+1}^T W_{j+1} = W_j W_j^T$ for all $j \in \llbracket 1,H-1 \rrbracket$. Indeed, in that case, the weight matrices $W_j$ have equal ranks and $(W_H \cdots W_1)(W_H \cdots W_1)^T = W_H \cdots W_2 (W_1 W_1^T) W_2^T \cdots W_H^T = W_H \cdots W_2 (W_2^T W_2) W_2^T \cdots W_H^T = W_H \cdots W_3 (W_2 W_2^T)^2 W_3^T \cdots W_H^T = \ldots = (W_H W_H^T)^H$, so that $\rk(W_j) = \rk(W_H)=\rk(W_H \cdots W_1)=r$ for all $j \in \llbracket 1,H \rrbracket$. Therefore, when $H \geq 3$, first-order critical points that are $0$-balanced are tightened.}

Note also that when $H=2$, there is no tightened critical point with $r<r_{max}$, because the pivot $(2,1)$ is not tightened (both complementary blocks $\Sigma_{XY}$ and $I_{d_1}$ are of full rank, which is larger than or equal to $r_{max} = \min\{d_y,d_1,d_x\}$).

\begin{figure}
    \centering
    \begin{tikzpicture}
\draw[dotted] (1.5*0.87,1.5*0.5) arc(30:50:1.5);
\draw[dotted] (1.5*-0.57,1.5*0.82) arc(125:155:1.5);
\draw[dotted] (1.5*-0.87,1.5*-0.5) arc(210:330:1.5);
\draw (1.5,0) circle(0.3);
\draw (-1.5,0.05) circle(0.3);
\draw (0,1.5*1) node{$ \Sigma_{XY}$};
\draw (1.5*0.42,1.5*0.90) node{$W_H$};
\draw (1.5*-0.38,1.5*0.90) node{$W_1$};
\draw (1.5*1,0) node[red]{$W_i$};
\draw (-1*1.5,0) node[red]{$W_j$};
\draw (1.5*0.94,1.5*0.34) node{$W_{i+1}$};
\draw (-1.3*0.94,-1.3*0.34) node{$W_{j+1}$};
\draw (1.5*0.94,-1.5*0.34) node{$W_{i-1}$};
\draw (-1.3*0.94,1.3*0.34) node{$W_{j-1}$};
\draw [<-] (2*0.98,2*0.17) arc(10:170:2);
\draw [<-] (-2*0.98,2*-0.17) arc(190:350:2);
\draw (0,2) node[above]{First complementary block:  $W_{j-1} \cdots W_1 \Sigma_{XY} W_H \cdots W_{i+1}$};
\draw (0,-2) node[below]{Second complementary block: $W_{i-1} \cdots W_{j+1} $};
\end{tikzpicture}
    \caption{Complementary blocks to the pivot $(i,j)$ \;.}
    
    \label{fig:cycle with 2 complementary blocks}
\end{figure}
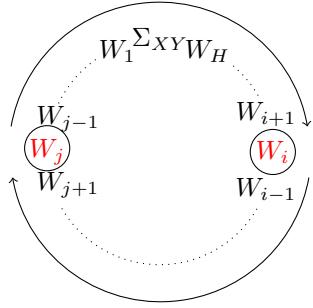

We can now state our main theorem, which characterizes the nature of any first-order critical point $\Wbf$ depending on the associated index set $\Ss$ and the tightening condition. The corresponding classification is illustrated on Figure \ref{fig:Classification of p.c}.
 Note that Theorem~\ref{main theorem} precisely differentiates between first-order critical points that are second-order critical points and those that are not. Combined with the fact that every first-order critical point is either a global minimizer or a saddle point \citep{kawaguchi2016deep}, we can distinguish global minimizers, strict saddle points and non-strict saddle points. The main and most technical contribution is, in the case $\Ss = \llbracket 1,r \rrbracket$, to distinguish between strict and non-strict saddle points.

We recall that $r_{max} = \min(d_H, \ldots, d_0)$ is the width of the thinnest layer, and that $U$ corresponds to the eigenvectors of $\Sigma$ (see \eqref{svd de sigma 1/2}).


\begin{theorem}[Classification of the critical points of $L$] \label{main theorem}
Suppose Assumption \ref{Assump H} in Section \ref{Settings} holds true.

Let $\Wbf = (W_H,\ldots,W_1)$ be a first-order critical point of $L$ and set $r=\rk(W_H \cdots W_1) \leq r_{max}$.
Following Proposition \ref{global map and critical values}, we consider the index set $\Ss$ associated with $\Wbf$.
\begin{itemize}
    \item When $r = r_{max}$:
    \begin{itemize}
        \item if $\Ss = \llbracket 1,r_{max} \rrbracket$, then $\Wbf$ is a global minimizer.
        \item if  $\Ss \neq \llbracket 1,r_{max} \rrbracket$, then $\Wbf$ is not a second-order critical point ($\Wbf$ is a strict saddle point).
    \end{itemize}
    \item When $ r < r_{max}$: $\Wbf$ is a saddle point.
    \begin{itemize}
        \item if $\Ss \neq \llbracket 1,r \rrbracket$, then $\Wbf$ is not a second-order critical point ($\Wbf$ is a strict saddle point).
        \item if $\Ss = \llbracket 1,r \rrbracket$: we have $W_H \cdots W_1 = U_{\Ss}U_{\Ss}^T \Sigma_{YX}\Sigma_{XX}^{-1} \in \argmin_{R \in \mathbb{R}^{d_y \times d_x}, \rk(R) \leq r} \|RX - Y\|^2.$
        \begin{itemize}
            \item if $ \Wbf$ is not tightened, then $\Wbf$ is not a second-order critical point ($\Wbf$ is a strict saddle point).
            \item if $ \Wbf$ is tightened, then $\Wbf$ is a second-order critical point ($\Wbf$ is a non-strict saddle point).
        \end{itemize}
    \end{itemize}
\end{itemize}
\end{theorem}

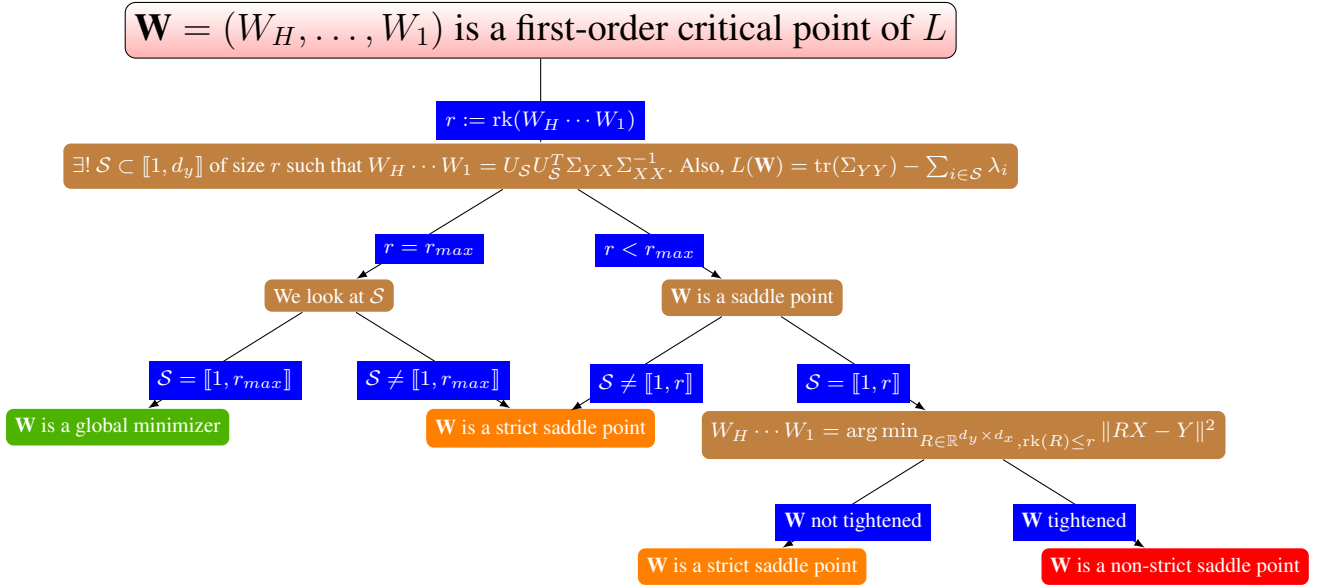
\begin{figure}
    
\scalebox{0.8}{
\tikzset{
  treenode/.style = {shape=rectangle, rounded corners,
                     draw, align=center,
                     top color=white, bottom color=blue!20},
  root/.style     = {treenode, font=\Large, bottom color=red!30},
  env/.style      = {treenode, font=\ttfamily\normalsize},
  dummy/.style    = {circle,draw}
}
\begin{tikzpicture}[sibling distance=16em,
    glob/.style={rectangle, draw=none, rounded corners=1mm, fill=green!70!red,
        text centered, anchor=north, text=white, },
    strict/.style={rectangle, draw=none, rounded corners=1mm, fill=orange,
        text centered, anchor=north, text=white},
    bad/.style={rectangle, draw=none, rounded corners=1mm, fill=red,
        text centered, anchor=north, text=white},
    middle/.style={rectangle, draw=none, rounded corners=1mm, fill=brown,
        text centered, anchor=north, text=white},
    rule/.style={rectangle, draw=none, fill=blue,
        text centered, anchor=north, text=white},
    every node/.style       = {font=\footnotesize},
     edge from parent/.style = {draw, -latex}]
    \node [root] {$\Wbf = (W_H,\ldots,W_1)$ is a first-order critical point of $L$}
    child {node [middle] {$\exists!$ $\Ss \subset \llbracket 1,d_y \rrbracket$ of size $r$ such that $W_H \cdots W_1 = U_{\Ss} U_{\Ss}^T \Sigma_{Y X} \Sigma_{X X}^{-1}$. Also,
     $L(\Wbf) = \tr(\Sigma_{YY}) - \sum_{i \in \Ss} \lambda_i$}
        child { node [middle] {We look at $\Ss$} 
            child {node[glob] {$\Wbf$ is a global minimizer}
            edge from parent node[rule] {$\Ss = \llbracket 1,r_{max} \rrbracket$} 
            }
            child {node[strict] {$\Wbf$ is a strict saddle point}
            edge from parent node[rule] {$\Ss \neq \llbracket 1,r_{max} \rrbracket$} 
            }
            edge from parent node[rule] {$r=r_{max}$}
        }
        child { node[middle] {$\Wbf$ is a saddle point}
            child { node[strict] {$\Wbf$ is a strict saddle point}
            edge from parent node[rule] {$\Ss \neq \llbracket 1,r \rrbracket$} 
            }
            child { node[middle] {$W_H \cdots W_1 = \argmin_{R \in \mathbb{R}^{d_y \times d_x},\rk(R) \leq r} \|RX - Y\|^2$}
                child { node[strict] {$\Wbf$ is a strict saddle point}
                edge from parent node[rule] {$\Wbf$ not tightened}
                }
                child {node[bad] {$\Wbf$ is a non-strict saddle point}
                edge from parent node[rule] {$\Wbf$ tightened}
                }
            edge from parent node[rule] {$\Ss = \llbracket 1,r \rrbracket$}
            }
            edge from parent node[rule] {$r < r_{max}$}
        }
        edge from parent node[rule] {$r:=\rk(W_H \cdots W_1)$}
    };
\end{tikzpicture}
} 
\caption{Second-order classification of the critical points of $L$.}
    
    \label{fig:Classification of p.c}
\end{figure}

The proof of Theorem \ref{main theorem} is given in Section \ref{sketch of proof}, with most technical derivations deferred to the appendix. We now make several remarks. Note from the above that every non-strict saddle point corresponds to a global minimizer of the \RCLRP, as already shown by \cite[Proposition~35]{Bah_2021}.

The next proposition shows the existence of both tightened and non-tightened critical points for $H \geq 3$ (there are no tightened critical points when $H=2$ and $r<r_{max}$). Combining this result with Proposition~\ref{reciproque S vers W} indicates that all conclusions of Theorem \ref{main theorem} can be observed. \textcolor{black}{In particular, as already established in \cite{kawaguchi2016deep}, $L$ is not a Morse function when $H\geq 3$.}
\begin{proposition}\label{existence of tightened and non tightened critical points}
    Suppose Assumption \ref{Assump H} in Section \ref{Settings} holds true.
    For $H \geq 3$, for every $\Ss = \llbracket 1,r \rrbracket$ with $0 \leq r<r_{max}$, there exist both a tightened critical point and a non-tightened critical point associated with $\Ss$.
\end{proposition}

The proof is postponed to Appendix \ref{proof of existence of tightened and non tightened critical points}.
It is constructive: we exhibit in the proof large sets of tightened and non-tightened critical points.\\
We can draw additional consequences from Theorem \ref{main theorem} and Propositions~\ref{reciproque S vers W} and~\ref{existence of tightened and non tightened critical points}:
\begin{itemize}
    \item[$\bullet$]
    For $H=2$, for any $r<r_{max}$, there exist strict saddle points satisfying $W_H \cdots W_1 \in \argmin_{R \in \mathbb{R}^{d_y \times d_x}, \rk(R) \leq r} \|RX - Y\|^2$.
    \item[$\bullet$]
    For $H \geq 3$, for any $r<r_{max}$, there exist both strict and non-strict saddle points satisfying $W_H \cdots W_1 \in \argmin_{R \in \mathbb{R}^{d_y \times d_x}, \rk(R) \leq r} \|RX - Y\|^2$.  
    \item[$\bullet$] In the special case $r=0$, we have $\Ss = \emptyset$ and $\emptyset = \llbracket 1,r \rrbracket$ by convention (see Section~\ref{Settings}), so that $\Ss = \llbracket 1,r \rrbracket$. In this case, Theorem \ref{main theorem} and Proposition \ref{existence of tightened and non tightened critical points} together imply that their exist both strict and non-strict saddle points $\Wbf$ such that $W_H \cdots W_1 = 0$ when $H \geq 3$.
\end{itemize}

\textcolor{black}{Finally, recall from a previous remark that, when $H \geq 3$, all first-order critical points that are $0$-balanced are tightened. We know from earlier works (e.g., \cite{arora2018a,arora2018optimization}) that the quantities $W_{j+1}^T W_{j+1} - W_j W_j^T$ are invariant under Gradient Flow. In particular, when we initialize the weight matrices such that these quantities are equal to zero (the so-called 0-balanced initialization), we have $W_{j+1}^T W_{j+1} - W_j W_j^T=0$ for all $j \in \llbracket 1,H-1 \rrbracket$ along the whole trajectory of Gradient Flow. In that case, all eventually visited saddle points associated to some $\Ss= \llbracket1,r\rrbracket$ 
are 0-balanced, hence tightened (by the remark after Definition \ref{def W tightned}) and therefore non-strict (by Theorem~\ref{main theorem}).}

\subsection{Parameterization of First-order Critical Points and Global Minimizers}\label{Parameterization of first-order critical points and global minimizers}

We now turn back to first-order critical points, and state all new related results.
In our analysis, these results precede the proof of Theorem \ref{main theorem}.
The presentation has been reversed in Section \ref{main contributions} to highlight the main contribution of the article.\\
The next proposition provides an explicit parameterization of first-order critical points.
Note that this is only a necessary condition.

\begin{proposition}\label{simplif mat}
    Suppose Assumption \ref{Assump H} in Section \ref{Settings} holds true.
    Let $\Wbf = (W_H, \ldots ,W_1)$ be a first-order critical point of $L$ associated with $\Ss$ (cf Proposition \ref{global map and critical values}), and let $Q = \llbracket1,d_y \rrbracket \setminus \Ss$.
    Then, there exist invertible matrices $D_{H-1} \in \mathbb{R}^{d_{H-1} \times d_{H-1}}, \ldots ,D_1 \in \mathbb{R}^{d_1 \times d_1}$ and matrices $Z_H \in \mathbb{R}^{(d_y-r) \times (d_{H-1}-r)}$, $Z_1 \in \mathbb{R}^{(d_1-r) \times d_x}$ and $Z_h \in \mathbb{R}^{(d_h-r) \times (d_{h-1}-r)}$ for $h \in \llbracket 2 , H-1 \rrbracket$ such that if we denote $\widetilde{W}_H = W_H D_{H-1}$ , $\widetilde{W}_1 = D_1^{-1} W_1$ and  $\widetilde{W}_h = D_{h}^{-1} W_h D_{h-1}$, for all $h \in \llbracket 2 , H-1  \rrbracket$, then we have
    \begin{align}
        \widetilde{W}_H &= [U_{\Ss} , U_Q Z_H] \label{simplif W_H} \\
        \widetilde{W}_1 &= \begin{bmatrix}
        U_{\Ss}^T\Sigma_{YX}\Sigma_{XX}^{-1} \\ 
        Z_1
        \end{bmatrix} \label{simplif W_1} \\
         \widetilde{W}_h &= \left[
        \begin{array}{c  c}
        I_r & 0 \\
        0 & Z_h
        \end{array}
        \right] \quad \forall h \in \llbracket 2 , H-1 \rrbracket \label{simplif W_h} \\
        \widetilde{W}_H \cdots \widetilde{W}_2 &= \left[U_{\Ss} , 0 \right]. \label{simplif W_H...W_2}
    \end{align}
    
\end{proposition}
The proposition is proved in Appendix \ref{proof simplif mat}, and will be key to prove the last statement of Theorem \ref{main theorem}. \\
Next, we give a sufficient condition for any $\Wbf$ satisfying \eqref{simplif W_H}, \eqref{simplif W_1} and \eqref{simplif W_h}, to be a first-order critical point of $L$.
\begin{proposition}\label{reciproque param}
     Suppose Assumption \ref{Assump H} in Section \ref{Settings} holds true.
    Let $\Ss \subset \llbracket 1,d_y \rrbracket$ of size $r \in \llbracket 0,r_{max} \rrbracket$ and $Q = \llbracket 1,d_y \rrbracket \setminus \Ss$.
    Let $D_{H-1} \in \mathbb{R}^{d_{H-1} \times d_{H-1}}, \ldots ,D_1 \in \mathbb{R}^{d_1 \times d_1}$ be invertible matrices and
    let $Z_H \in \mathbb{R}^{(d_y-r) \times (d_{H-1}-r)}$, $Z_1 \in \mathbb{R}^{(d_1-r) \times d_x}$ and $Z_h \in \mathbb{R}^{(d_h-r) \times (d_{h-1}-r)}$ for $h \in \llbracket 2 , H-1 \rrbracket$.
    Let the parameter of the network $\Wbf = (W_H,\ldots,W_1)$ be defined as follows:
\begin{align*}
    {W}_H &= [U_{\Ss} , U_Q Z_H]D_{H-1}^{-1} \\
    {W}_1 &= D_1 \begin{bmatrix}
    U_{\Ss}^T\Sigma_{YX}\Sigma_{XX}^{-1} \\ 
    Z_1
    \end{bmatrix}  \\
    {W}_h &= D_h\left[
    \begin{array}{c  c}
    I_r & 0 \\
    0 & Z_h
    \end{array}
    \right]D_{h-1}^{-1} \quad \forall h \in \llbracket 2 , H-1 \rrbracket \;.
\end{align*}
If $r=r_{max}$ or if there exist $h_1 \neq h_2$ such that $Z_{h_1} = 0$ and $Z_{h_2} = 0$,
then, $\Wbf$ is a first-order critical point of $L$ associated with $\Ss$.
\end{proposition}
The proof of Proposition \ref{reciproque param} is in Appendix \ref{proof reciproque param}.\\
Note that, combining Propositions \ref{simplif mat} and \ref{reciproque param}, we obtain an explicit parameterization of all critical points $\Wbf$ with a global map $W_H \cdots W_1$ of maximum rank $r_{max}$.
In particular, it yields the next proposition, which provides an explicit parameterization of all the global minimizers of $L$.

\begin{proposition}[Parameterization of all global minimizers]\label{prop parameterization of global minimizers}
    Suppose Assumption \ref{Assump H} in Section \ref{Settings} holds true.
    Set $\Ss_{max} = \llbracket 1,r_{max} \rrbracket$ and $Q_{max} = \llbracket 1,d_y \rrbracket \setminus \Ss_{max} = \llbracket r_{max}+1 , d_y \rrbracket$.\\
    Then, $\Wbf = (W_H, \ldots , W_1)$ is a global minimizer of $L$ if and only if there exist invertible matrices $D_{H-1} \in \mathbb{R}^{d_{H-1} \times d_{H-1}}, \ldots ,D_1 \in \mathbb{R}^{d_{1} \times d_{1}}$, and matrices $Z_H \in \mathbb{R}^{(d_y-r_{max}) \times (d_{H-1}-r_{max})}$, $Z_h \in \mathbb{R}^{(d_h-r_{max}) \times (d_{h-1}-r_{max})}$ for $h \in \llbracket 2 , H-1 \rrbracket$, and $Z_1 \in \mathbb{R}^{(d_1-r_{max}) \times d_x}$ such that:
    \begin{align*}
        {W}_H &= [U_{\Ss_{max}} , U_{Q_{max}} Z_H] D_{H-1}^{-1} \\
        {W}_1 &= D_1 \begin{bmatrix}
        U_{\Ss_{max}}^T\Sigma_{YX}\Sigma_{XX}^{-1} \\ 
        Z_1
        \end{bmatrix} \\
        {W}_h &= D_h \left[
        \begin{array}{c  c}
        I_{r_{max}} & 0 \\
        0 & Z_h
        \end{array}
        \right] D_{h-1}^{-1} \qquad \forall h \in \llbracket 2 , H-1 \rrbracket \;.
    \end{align*}
\end{proposition}
The proof is in Appendix \ref{proof of prop parameterization of global minimizers}. See in particular
a remark in the same appendix on how to interpret the above formulas precisely (some blocks $Z_h$ have $0$ lines or columns).

\subsection{Comparison with the State-of-the-art}\label{Comparison with previous works}

Next we further detail our contributions in light of earlier works.\\

\textbf{Parameterization of global minimizers.} To the best of our knowledge, Proposition~\ref{prop parameterization of global minimizers} is the first explicit parameterization of the set of all global minimizers for deep linear networks and the square loss. For $H \geq 2$, it had been previously noted by \cite{yun2018global} that a critical point $\Wbf$ is a global minimizer if and only if $\rk(W_H \cdots W_1) = r_{max}$ and $\col(W_H \cdots W_{d_{p+1}}) = \col(U_{\Ss_{max}})$, where $\Ss_{max} = \llbracket 1,r_{max} \rrbracket$ and where $p$ is any layer with the smallest width $r_{max}$.
This is an implicit characterization.

Another previous work that characterized global minimizers is \cite{zhou2018critical}, but their characterization is not explicit: the weight matrices are defined recursively and should satisfy some equations, while in Proposition~\ref{prop parameterization of global minimizers} the weight matrices are given explicitly.
The same remark holds for their characterization of first-order critical points. \\

\textbf{Saddle points.} 
Among saddle points, we give a characterization of those that are strict and those that are not. \\
Previously, for $H \geq 3$, it had been noted by \cite{kawaguchi2016deep} that $(0,\ldots,0)$ is a non-strict saddle point. This result also follows from Theorem 1 since any critical point is tightened whenever at least 3 weight matrices are of rank $r=\rk(W_H \cdots W_1)$ (which is the case for $(0,\ldots,0)$ with $r=0$).

Also, Theorem \ref{main theorem} generalizes two results from \cite{kawaguchi2016deep} and \cite{chitour2018geometric} about sufficient conditions for strict saddle points. Indeed, it is proved in
\cite{kawaguchi2016deep}  that, if $\Wbf$ is a saddle point such that $\rk(W_{H-1} \cdots W_2) = r_{max}$, then $\Wbf$ is a strict saddle point.
\cite{chitour2018geometric} proved under further assumptions on the data and the architecture that a sufficient condition for a saddle point to be strict is that $\rk(W_{H-1} \cdots W_2)> r = \rk(W_H \cdots W_1)$.
Note that both results are special cases of Theorem~\ref{main theorem}, with the pivot $(H,1)$.
More precisely, assume that $\Wbf$ is a saddle point such that either $\rk(W_{H-1} \cdots W_2) = r_{max} = r = \rk(W_H \cdots W_1)$ or $\rk(W_{H-1} \cdots W_2)> r = \rk(W_H \cdots W_1)$ (which includes both conditions above). Then, if $\Ss \neq \llbracket 1,r \rrbracket$ (whether $r=r_{max}$ or not), by Theorem~\ref{main theorem}, $\Wbf$ is a strict saddle point without any condition on $\Wbf$. But if $\Ss = \llbracket 1,r \rrbracket$ with $r<r_{max}$, 
our assumption above implies that
the pivot $(H,1)$, and therefore $\Wbf$, is not tightened (recall that $\rk(\Sigma_{XY})=d_y \geq r_{max} > r$). In any case, $\Wbf$ is a strict saddle point.

Finally, Theorem \ref{main theorem} generalizes another result of \cite{kawaguchi2016deep} stating that all saddle points are strict for one-hidden layer linear networks. Indeed, let $H=2$ and assume that we have a saddle point associated with $\Ss = \llbracket 1,r \rrbracket$ for $r<r_{max}$ (the only case where we can expect to see non-strict saddle points, by Theorem~\ref{main theorem}). Since $H=2$, there is only one pivot which is $(2,1)$; this pivot is not tightened because the complementary blocks are $I_{d_1}$ and $\Sigma_{XY}$ and both are of rank larger than or equal to $r_{max}$. Therefore, by Theorem~\ref{main theorem}, when $H=2$ (and under Assumption \ref{Assump H}), all saddle points are strict.\\

\textbf{Convergence to global minimizer: an example where gradient descent meets no non-strict saddle points.}
Some recent works on deep linear networks proved under assumptions on the data, the initialization, or the minimum width of the network, that gradient descent or variants converge to a global minimum in polynomial time \citep[e.g.,][]{arora2018a,bartlett2018gradient,Eftekhari20,du2019width}. Since for general non-convex functions, gradient descent may get stuck at a non-strict saddle point, and since non-strict saddle points exist for any linear neural network of depth $H \geq 3$, it seemed impossible to deduce convergence to a global minimum using landscape results only. Instead, papers such as \cite{du2019width} chose to ``directly analyze the trajectory generated by [...] gradient descent''.

It turns out that our characterization of strict saddle points can help re-interpret such global convergence results. Consider for instance the work of \cite{du2019width}, who proved that with high probability gradient descent with Xavier initialization converges to a global minimum for any deep linear network which is wide enough. They analyze a network where all hidden layers have a width $d_{\textrm{hidden}}$ at least proportional to the number $H$ of layers and to other quantities depending on the data $X,Y$, the output dimension $d_y$, and the desired probability level. In their analysis, \cite[Section~7]{du2019width} prove that with high probability, a condition $\mathcal{B}(t)$ holds at every iteration $t$. Importantly, this condition implies that the point $\Wbf$ output by gradient descent at iteration $t$ cannot be a non-strict saddle point. Indeed, using our notation, the condition $\mathcal{B}(t)$ yields the lower-bound\footnote{\textcolor{black}{$\sigma_{\min}(W_H \cdots W_2)$ denotes the minimum singular value of $W_H \cdots W_2 \in \mathbb{R}^{d_y \times d_{\textrm{hidden}}}$, among $\min\{d_y,d_{\textrm{hidden}}\}=d_y$ singular values in total (\citealt{{du2019width}} assume that $d_{\textrm{hidden}} \geq d_y$).}} $\sigma_{\min}(W_H \cdots W_2) \geq \frac{3}{4} d_{\textrm{hidden}}^{(H-1)/2} > 0$, which in particular entails that the matrix product $W_H \cdots W_2$ is of full rank $\min\{d_{\textrm{hidden}},d_y\} \geq r_{max}$. 
Let us check that if $\Wbf$ is a saddle point, then it is necessarily strict. By Theorem~\ref{main theorem}, either $r=\rk(W_H \cdots W_1)$ is equal to $r_{max}$, in which case the saddle point $\Wbf$ is indeed strict, or $r<r_{max}$, in which case the pivot $(H,1)$ is not tightened (since the two blocks $\Sigma_{XY}$ and $W_{H-1} \cdots W_2$ are of rank at least $r_{max}$), so that the saddle point $\Wbf$ is strict, as previously claimed.

As a consequence, our characterization of strict saddle points in Theorem~\ref{main theorem} helps re-interpret the analysis of \cite[Section~7]{du2019width}: under Assumption \ref{Assump H}, and for wide enough deep linear networks, gradient descent with Xavier initialization meets no non-strict saddle points on its trajectory.\\

\textbf{Implicit regularization.} 
Implicit regularization, in the context of linear networks, refers to statements showing  that the iterates trajectory passes in the vicinity of critical points $\Wbf$ such that $W_H\cdots W_1 = \argmin_{R \in \mathbb{R}^{d_y \times d_x}, \rk(R) \leq r} \|RX - Y\|^2$, for increasing $r\in\llbracket 0,r_{max} \rrbracket$. In such settings, the gradient dynamics sequentially finds the best linear regression predictor in
\[\mathcal{D}_r = \{R \in \mathbb{R}^{d_y \times d_x}~,~\rk(R) \leq r\},
\]
for increasing $r$. The subset $\mathcal{D}_r\subset\mathbb{R}^{d_y\times d_x}$ is independent of $X$, $Y$ and the network architecture, and plays the role of a regularization constraint in the function space.

In the parameter space however, as indicated in Theorem~\ref{main theorem} and  Proposition~\ref{existence of tightened and non tightened critical points}, there exists both non-strict and strict saddle points. As illustrated in Section \ref{sec:intro-importanceordre2}, Figures~\ref{fig:eg loss near saddle point} and~\ref{fig:histo escape points}, it takes more time to a first order algorithm to escape non-strict saddle points than strict ones. When $H \geq 3$, there exist two phenomenon: a 'light' implicit regularization, in the vicinity of strict saddle points, and  a 'strong' implicit regularization in the vicinity of non-strict saddle points.

In \cite{Bah_2021}, the authors proved that gradient flow converges almost surely to a global minimizer or non-strict saddle points of $L$. The limit point corresponds to a global minimizer of the \RCLRP.
In Theorem \ref{main theorem} and Proposition \ref{existence of tightened and non tightened critical points}, we prove the existence and characterize such points, and in addition to non-strict saddle points we prove that some $\Wbf$ leading to the solution of the \RCLRP~are strict saddle points.
Doing so, we characterize and drastically reduce the strong implicit regularization set.

In \cite{gidel2019implicit}, the authors proved that for $H=2$, for a vanishing initialization and a sufficiently small learning-rate, the gradient algorithm sequentially learns solutions of the \RCLRP~with a gradually increasing rank.
More precisely, the algorithm avoids all critical points associated with $\Ss \neq \llbracket 1,r \rrbracket$, but comes close to a critical point associated with $\Ss=\llbracket 1,r \rrbracket$, spends some time around it and decreases again. We know that for $H=2$ all saddle points are strict and that the phenomenon described by the authors corresponds to a 'light implicit regularization'.

In \cite{gissin2019implicit}, the authors proved for a toy  linear network, that, for $H=2$, the algorithms need an exponentially vanishing initialization for this incremental learning to occur, while for $H \geq 3$, a polynomially vanishing initialization is enough. This indicates that this incremental learning arises more frequently in deep networks. The difference might be explained by the 'strong' implicit regularization due to the existence of non-strict saddle points when $H\geq 3$.

Authors have put to evidence the rank related implicit regularization depicted in Theorem \ref{main theorem} and Proposition \ref{existence of tightened and non tightened critical points} for similar problems.
In \cite{arora2019implicit}, the authors exhibit that for small initializations and learning-rate, for matrix recovery, deep matrix factorization favors solutions of low-rank. In the same context, the authors of \cite{razin2020nipsimplicit} state that, implicit regularization in deep matrix completion should be seen as a minimization of rank rather than norms.

\subsection{Perspectives}

\textbf{Implicit regularization.} From Theorem \ref{main theorem}, we know that the critical points such that $W_H\cdots W_1 = \argmin_{R \in \mathbb{R}^{d_y \times d_x}, \rk(R) \leq r} \|RX - Y\|^2$ can be either strict saddle points or non-strict saddle points. From Proposition \ref{existence of tightened and non tightened critical points} we know that both cases exist. We know from the experiment described in Figures~\ref{fig:eg loss near saddle point} and~\ref{fig:histo escape points}, that first order algorithms need more time to escape from the vicinity of non-strict saddle points than strict saddle points. There are two phenomena: 'light' and 'strong' implicit regularization. To the best of our knowledge, whether the saddle points approached by the iterates trajectory are strict or non-strict and the impact of this property on the implicit regularization phenomenon have not been studied.

Though this study goes beyond the scope of this paper, 
let us sketch the main trends that we can anticipate from our results. On one side, as explained above, we anticipate the number of iterations spent by a first-order algorithm in the vicinity of a  non-strict saddle point to be larger than in the vicinity of a strict saddle point. Said differently, the 'size' of the flat region surrounding non-strict saddle points is larger than the one surrounding strict saddle points. 
On the other side, looking at the rank constraint in Definition \ref{def pivot tightned} (which corresponds to the very last item of Theorem~\ref{main theorem}), we anticipate that there are much fewer non-strict saddle points than strict saddle points. 'Strong' implicit regularization therefore occurs at fewer locations. \textcolor{black}{The influence of these two factors on the trajectory of the iterates depends on the initialization and the chosen algorithm.}

\textbf{Extent of 'flat regions'.}
Beyond the behavior of the objective function captured by the derivatives, it would be interesting to study the extent of the 'flat regions'. The goal would typically be to provide estimates of the time spent by a (stochastic) first order algorithm to escape the flat region. We observed in Figures~\ref{fig:eg loss near saddle point} and~\ref{fig:histo escape points}, that the flat regions associated to non-strict saddle points are larger but it would be interesting to extend this empirical study and to study formal estimates of the 'size' of the flat regions.

\textbf{Basins of attraction.}
Second order critical points can be limit points of gradient descent algorithms. Even worse, the basin of attraction of such points can be of positive Lebesgue measure. It would be interesting to exploit the tightness condition and the manifold of non-strict saddle points to prove that, as conjectured in \cite{chitour2018geometric} and \cite{Bah_2021}, the gradient descent algorithm almost surely converges to a global minimizer. 

\textbf{Generalizing the tightness condition.}
The tightness condition in the definitions \ref{def pivot tightned} and \ref{def W tightned} is for instance satisfied as soon as three factors are of rank $r$. It is adapted to linear networks. It would be interesting to generalize it to other problems such as matrix factorization, structured linear networks or tensor problems, sharing the same 'compositional structure'.

\section{Proof of Theorem \ref{main theorem}}\label{sketch of proof}

The proof of Theorem~\ref{main theorem} proceeds in several steps. In the end (see page~\pageref{pf:thm1}), it will directly follow from Propositions~\ref{main prop partie baldi}, \ref{main prop subtle strict saddles}, \ref{main prop non strict saddles} below and from Lemma~\ref{lemma yun} in Appendix~\ref{Appendix Notation and general useful lemmas and properties}. In this section, we outline the overall proof structure and state the main intermediate results. We also provide proof sketches for these intermediate results, but defer many technical details to the appendix. \\

In our proofs, we will not compute the Hessian $\nabla^2 L(\Wbf)$ explicitly since this might be quite tedious.
To show that a point $\Wbf$ is (or is not) a second-order critical point of $L$, we will instead Taylor-expand $L(\Wbf + t\Wbf')$ along any direction $\Wbf'$ and use the following lemma. 
Its proof follows directly from Taylor's theorem. 

\begin{lemma}[Characterization of first-order and second-order critical points]\label{Characterization of critical points in our settings}
    Let $\Wbf = (W_H, \ldots, W_1)$.
    Assume that, for all $\Wbf' = (W'_H, \ldots, W'_1)$, the loss $L(\Wbf + t \Wbf') $ admits the following asymptotic expansion when $t \to 0$:
    \begin{align} \label{Taylor expansion of loss}
        L(\Wbf + t \Wbf') &= L(\Wbf) + c_1(\Wbf,\Wbf') t + c_2(\Wbf , \Wbf') t^2 + o(t^2).
    \end{align}
    Then:
    \begin{itemize}
        \item $\Wbf$ is a first-order critical point of $L$ iff $c_1(\Wbf,\Wbf') = 0$ for all $\Wbf'$. 
        \item $\Wbf$ is a second-order critical point of $L$ iff $c_1(\Wbf,\Wbf') = 0$ and $c_2(\Wbf,\Wbf') \geq 0$ for all $\Wbf'$.\\
        Therefore if for a first-order critical point $\Wbf$, we can exhibit a direction $\Wbf'$ such that $c_2(\Wbf,\Wbf') < 0$, then $\Wbf$ is not a second-order critical point.
    \end{itemize}
\end{lemma}

We divide the proof of Theorem \ref{main theorem} into three parts.
Recall that from \cite{kawaguchi2016deep}, we know that  all first-order critical points are either global minimizers or saddle points (that is, there is no local extrema apart from global minimizers).
We refine this classification.

\subsection{Global Minimizers and 'Simple' Strict Saddle Points}

In this section, we start by identifying simple sufficient conditions on the support $\Ss$ associated to a first-order critical point $\Wbf$ which guarantee that $\Wbf$ is either a global minimizer or a strict saddle point.
More subtle strict saddle points and non-strict saddle points will be addressed in Sections \ref{subtle strict saddles} and \ref{non-strict saddles}.
\begin{proposition}\label{main prop partie baldi}
    Suppose Assumption \ref{Assump H} in Section \ref{Settings} holds true.
    Let $\Wbf = (W_H, \ldots , W_1)$ be a first-order critical point of $L$ associated with $\Ss$ and set $r=\rk(W_H \cdots W_1) \leq r_{max}$.
\begin{itemize}
    \item When $r = r_{max}$:
    \begin{itemize}
        \item if $\Ss = \llbracket 1,r_{max} \rrbracket$, then $\Wbf$ is a global minimizer.
        \item if  $\Ss \neq \llbracket 1,r_{max} \rrbracket$, then $\Wbf$ is not a second-order critical point ($\Wbf$ is a strict saddle point).
    \end{itemize}
    \item When $ r < r_{max}$: $\Wbf$ is a saddle point.
    \begin{itemize}
        \item if $\Ss \neq \llbracket 1,r \rrbracket$, then $\Wbf$ is not a second-order critical point ($\Wbf$ is a strict saddle point).
        \end{itemize}
    \end{itemize}
\end{proposition}

The proof is postponed to Appendix \ref{proof main prop partie baldi}.
To prove that $\Wbf$ associated with $\Ss \neq \llbracket 1,r \rrbracket$ , $r\leq r_{max}$ is not a second-order critical point,
we explicitly exhibit a direction $\Wbf'$ such that the second-order coefficient $c_2(\Wbf,\Wbf')$ in the Taylor expansion of $L(\Wbf + t \Wbf') $ around $t=0$, in \eqref{Taylor expansion of loss}, is negative. 
Using Lemma \ref{Characterization of critical points in our settings}, we conclude that $\Wbf$ is not a second-order critical point.\\
Recall from Proposition \ref{global map and critical values} that the loss at any first-order critical point is given by $tr(\Sigma_{YY}) - \sum_{i \in \Ss} \lambda_i$.
The spirit of the proof is that critical points associated with $\Ss \neq \llbracket 1,r \rrbracket$ capture a smaller singular value $\lambda_j$ instead of a larger one $\lambda_i$ with $i < j$.
Thus, to see that the loss can be further decreased at order 2 (and is therefore not a second-order critical point by Lemma \ref{Characterization of critical points in our settings}), a natural proof strategy is to perturb the singular vector corresponding to $\lambda_j$ along the direction of the singular vector corresponding to $\lambda_i$.
This part of the proof is an adaption of the proof of \cite{Baldi:1989:NNP:70359.70362}.

\subsection{Strict Saddle Points Associated with \texorpdfstring{$\Ss = \llbracket 1,r \rrbracket, \  r< r_{max}$}{}} \label{subtle strict saddles}
We now address situations that to our knowledge, have never been addressed, in the literature.
We prove the following.
\begin{proposition}\label{main prop subtle strict saddles}
Suppose Assumption \ref{Assump H} in Section \ref{Settings} holds true.
Let $\Wbf = (W_H, \ldots, W_1)$ be a first-order critical point of $L$ associated with $\Ss = \llbracket 1,r \rrbracket$, with $0 \leq r < r_{max}$. \\
If $\Wbf$ is not tightened, then $\Wbf$ is not a second-order critical point ($\Wbf$ is a strict saddle point).
\end{proposition}
We sketch the main arguments below.
We will again construct a direction $\Wbf'$ such that the second-order coefficient $c_2(\Wbf,\Wbf')$ in the asymptotic expansion of $L(\Wbf + t \Wbf') $ around $t=0$, in \eqref{Taylor expansion of loss}, is negative. \\
More precisely, for a first-order critical point $\Wbf$, for any $\beta \in \mathbb{R}$, we will consider a well-chosen $\Wbf'_{\beta}$  such that $c_2(\Wbf , \Wbf'_{\beta}) = a \beta^2 + c \beta$ for some constants $a,c$ (possibly depending on $\Wbf$) such that $a \geq 0$ and $c \neq 0$. Taking 
\begin{equation} \label{choosing beta}
    \beta =
    \begin{cases}
        -c \quad \text{if} \quad a=0 \\
        -\frac{c}{2a} \quad \text{if} \quad a>0 \\
    \end{cases}
\end{equation}
we obtain
$$c_2(\Wbf , \Wbf'_{\beta}) = \begin{cases}
-c^2 \quad \text{if} \quad a=0 \\
-\frac{c^2}{4a} \quad \text{if} \quad a>0 \\
\end{cases} $$
and therefore $$ c_2(\Wbf , \Wbf'_{\beta}) < 0.$$
Using Lemma \ref{Characterization of critical points in our settings}, we can conclude that $\Wbf$ is not a second-order critical point. \\
We now provide intuitions on how to choose $\Wbf'$.
Since $\Wbf$ is not tightened, there exists a pivot $(i,j)$, with $i>j$, which is not tightened. 
Depending on the values of $i$ and $j$ we will construct $\Wbf'$ differently.
However, the strategy for constructing $\Wbf'$ is the same in all cases. \\
Recall again that from Proposition \ref{global map and critical values}, at any first-order critical point $\Wbf$, the value of the loss is given by $tr(\Sigma_{YY}) - \sum_{i \in \Ss} \lambda_i$.
Contrary to the previous section, since $\Ss=\llbracket 1,r \rrbracket$ there is no immediate way to decrease the loss (at order 2) without increasing the rank of the product of the weight matrices.
Indeed, we have $W_H \cdots W_1 = U_{\Ss}U_{\Ss}^T \Sigma_{YX}\Sigma_{XX}^{-1} \in \argmin_{\rk(R) \leq r} \|RX-Y\|^2 $. \\
Therefore, to be able to decrease the value of the loss, we need to perturb $\Wbf$ in a way that the product of the perturbed parameter weight matrices becomes of rank strictly larger than $r$. 
Also, to prove that $\Wbf$ is not a second-order critical point, we need to decrease the loss at order 2.
This is possible when $\Wbf$ is not tightened.
For the non-tightened pivot $(i,j)$, we choose a perturbation $\Wbf'$ with all $W'_h=0$ except for $W'_i$ and $W'_j$.
Furthermore, our construction of $W'_i$ and $W'_j$ depends on whether $i$ and/or $j$ are on the boundary $\{1,H\}$.
This is due to the fact that $H$ and $1$ play a special role in the product of the perturbed weights $(W_H + tW'_H) \cdots (W_1 + tW'_1)$.
This is why we distinguish the four cases below:

\begin{itemize}
    \item 1st case: $i \in \llbracket 2,H-1 \rrbracket$ and $j=1$. This case is treated in Appendix \ref{1st case}.
    \item 2nd case: $i=H$ and $j=1$. This case is treated in Appendix \ref{2nd case}.
    \item 3rd case: $i=H$ and $j \in \llbracket 2,H-1 \rrbracket$. This case is treated in Appendix \ref{3rd case}.
    \item 4th case: $i,j \in \llbracket 2,H-1 \rrbracket$ with $i>j$. This case is treated in Appendix \ref{4th case}.
\end{itemize}

\subsection{Non-strict Saddle Points}\label{non-strict saddles}
We now provide a sketch of the proof for the converse of Proposition \ref{main prop subtle strict saddles}, as stated in Proposition \ref{main prop non strict saddles} below.
All the proofs related to this section are deferred to Appendix \ref{Appendix non strict saddles}.

\begin{proposition} \label{main prop non strict saddles}
Suppose Assumption \ref{Assump H} in Section \ref{Settings} holds true.
Let $\Wbf = (W_H, \ldots, W_1)$ be a first-order critical point of $L$ associated with $\Ss = \llbracket 1,r \rrbracket$, $0 \leq r<r_{max}$. \\
If $\Wbf$ is tightened, then $\Wbf$ is a second-order critical point ($\Wbf$ is a non-strict saddle point).
\end{proposition}

To prove Proposition \ref{main prop non strict saddles}, we first state a proposition which indicates that multiplications by invertible matrices do not change the nature of the critical point.
\begin{lemma}\label{same nature}
    For all $h \in \llbracket 1 , H-1 \rrbracket$, let $D_h \in \mathbb{R}^{d_{h} \times d_{h}}$ be an invertible matrix.
    We define  $\widetilde{W}_H = W_H D_{H-1}$ , $\widetilde{W}_1 = D_1^{-1} W_1$ and  $\widetilde{W}_h = D_{h}^{-1} W_h D_{h-1}$, for all $h \in \llbracket 2 , H-1 \rrbracket$. Then
    \begin{itemize}
        \item $\Wbf = (W_H , \ldots , W_1)$ is a first-order critical point of $L$ if and only if $\widetilde{\Wbf} = (\widetilde{W}_H , \ldots , \widetilde{W}_1)$ is a first-order critical point of $L$.
        \item $\Wbf = (W_H , \ldots , W_1)$ is a second-order critical point of $L$ if and only if $\widetilde{\Wbf} = (\widetilde{W}_H , \ldots , \widetilde{W}_1)$ is a second-order critical point of $L$.
    \end{itemize}
\end{lemma}
The lemma is proved in Appendix \ref{Proof same nature}. \\

Proposition \ref{main prop non strict saddles} is then obtained using Proposition \ref{simplif mat} (note that when $\Wbf$ is tightened, $\widetilde{\Wbf}$ is also tightened since
the rank of a matrix does not change when multiplied by invertible matrices), by showing
that $\widetilde{\Wbf} = (\widetilde{W}_H, \ldots ,\widetilde{W}_1)$ as given by Proposition \ref{simplif mat} is a second-order critical point of $L$ and using Lemma \ref{same nature} to conclude that $\Wbf$ is a second-order critical point.
This is easier since $\widetilde{\Wbf}$ has a simpler form.

More precisely, we have the following result, from which Proposition \ref{main prop non strict saddles} follows (see Appendix \ref{Proof of main prop non strict saddles} for details).

\begin{proposition}\label{prop 4}
Suppose Assumption \ref{Assump H} in Section \ref{Settings} holds true.
Let $\Wbf = (W_H,\ldots,W_1)$ be a first-order critical point of $L$ associated with $\Ss = \llbracket 1,r \rrbracket$ with $0 \leq r< r_{max}$ such that there exist matrices $Z_H \in \mathbb{R}^{(d_y-r) \times (d_{H-1}-r)}$, $Z_1 \in \mathbb{R}^{(d_1-r) \times d_x}$ and $Z_h \in \mathbb{R}^{(d_h-r) \times (d_{h-1}-r)}$ for $h \in \llbracket 2 , H-1 \rrbracket$ with
\begin{align}
    {W}_H &= [U_{\Ss} , U_Q Z_H] \label{W_H simplified} \\
    {W}_1 &= \begin{bmatrix}
    U_{\Ss}^T\Sigma_{YX}\Sigma_{XX}^{-1} \\ 
    Z_1
    \end{bmatrix} \label{W_1 simplified} \\
    {W}_h &= \left[
    \begin{array}{c  c}
    I_r & 0 \\
    0 & Z_h
    \end{array}
    \right] \quad \forall h \in \llbracket 2 , H-1 \rrbracket \label{W_h simplified} \\
    {W}_H \cdots {W}_2 &= \left[U_{\Ss} , 0 \right], \label{W_H...W_2 simplified}
\end{align} 
where $Q = \llbracket1,d_y \rrbracket \setminus \Ss$. \\
If $\Wbf$ is tightened,
then $\Wbf$ is a second-order critical point of $L$.
\end{proposition}
Proposition \ref{prop 4} is proved in details in Section \ref{proof prop 4}.
We provide a proof sketch below. \\
We denote, for $t$ in the neighborhood of 0, and $h \in \llbracket 1,H \rrbracket$,  $W_h(t) = W_h + t W'_h$ where $W'_h \in \mathbb{R}^{d_h \times d_{h-1}}$ is arbitrary.\\
We define $\Wbf(t) := (W_H(t), \ldots ,W_1(t))$ and $W(t) := W_H(t) \cdots W_1(t)$.
As in the previous two sections, we use Lemma \ref{Characterization of critical points in our settings}.
However, this time, we show that the second-order coefficient $c_2(\Wbf,\Wbf')$ is non-negative for all directions $\Wbf'$. \\
To compute the loss $\|W(t)X-Y\|^2$, we expand
\begin{align*}
    W(t) &= W_H(t) \cdots W_1(t) \\
    &= (W_H + t W'_H) \cdots (W_1 + t W'_1) \\
    &= W_H \cdots W_1 + t\sum_{i=1}^{H} W_H \cdots W_{i+1} W'_i W_{i-1} \cdots W_1\\ &\quad + t^2 \sum_{H \geq i>j \geq 1} W_H \cdots W_{i+1} W'_i W_{i-1} \cdots W_{j+1} W'_j W_{j-1} \cdots W_1 +o(t^2) \;.
\end{align*}
Therefore,
\begin{align*}
    L(\Wbf(t)) &= \left\| W_H \cdots W_1 X - Y + t\sum_{i=1}^{H} W_H \cdots W_{i+1} W'_i W_{i-1} \cdots W_1 X \right. \\ 
    & \quad  \left. + t^2 \sum_{H \geq i>j \geq 1} W_H \cdots W_{i+1} W'_i W_{i-1} \cdots W_{j+1} W'_j W_{j-1} \cdots W_1 X +o(t^2) \right\|^2 \;.
\end{align*}
We can now easily calculate the second-order coefficient $c_2(\Wbf,\Wbf')$ in the Taylor expansion of $L(\Wbf(t))$ around $t=0$ (in \eqref{Taylor expansion of loss}). \\
Recalling that $c_2(\Wbf, \Wbf')$ is such that 
$L(\Wbf(t)) = L(\Wbf) + c_2(\Wbf , \Wbf') t^2 + o(t^2) $ (since $\Wbf$ is a first-order critical point), we have
\begin{align*}
    c_2(\Wbf , \Wbf') &= \left\|\sum_{i=1}^{H} W_H \cdots W_{i+1} W'_i W_{i-1} \cdots W_1 X \right\|^2 \\
    &\quad + 2 \left\langle  \sum_{H \geq i>j \geq 1} W_H \cdots W_{i+1} W'_i W_{i-1} \cdots W_{j+1} W'_j W_{j-1} \cdots W_1 X \ ,\ W_H \cdots W_1 X - Y \right\rangle \;,
\end{align*}
where $\left\langle  A,B \right\rangle = \tr(AB^T)$.
In order to simplify the notation and equations, we define, for all $i \in \llbracket 1 , H \rrbracket$,  
\begin{align}
    T_i = W_H \cdots W_{i+1} W'_i W_{i-1} \cdots W_1 X \;, \label{T_i}
\end{align}
and for all $i,j \in \llbracket 1 , H \rrbracket$ with $i > j$: 
\begin{align}
    T_{i,j} = \left\langle  W_H \cdots W_{i+1} W'_i W_{i-1} \cdots W_{j+1} W'_j W_{j-1} \cdots W_1 X  \ ,\ W_H \cdots W_1 X - Y \right\rangle \;. \label{T_ij}  
\end{align}
Then we set
\begin{align}
    FT = \left\|\sum_{i=1}^{H} T_i \right\|^2 \;, \label{FT}
\end{align}
and 
\begin{align}
    ST = 2 \sum_{H \geq i>j \geq 1} T_{i,j} \;. \label{ST}
\end{align}
The coefficient becomes 
\begin{align*}
    c_2(\Wbf , \Wbf') &= \left\|\sum_{i=1}^{H} T_i \right\|^2 + 2 \sum_{H \geq i>j \geq 1} T_{i,j} = FT + ST \;.
\end{align*}
Using the fact that $\Wbf$ is tightened, some weight products become simple (see Lemma \ref{E}) and  we can simplify $T_i$ and $T_{i,j}$ (see Lemmas \ref{simplif all T_i} and \ref{simplif all T_ij} in Appendix \ref{Appendix non strict saddles}). \\
This allows us to establish that, for any $\Wbf'$, there exist matrices $A_2$, $A_3$, $A_4$ and a non-negative scalar $a_1$ such that $FT = a_1 + \|A_2\|^2 + \|A_3\|^2 + \|A_4\|^2$ (see Appendix \ref{ section simplifying FT}) and $ST = -2 \left\langle  A_3 , A_4 \right\rangle$ (see Appendix \ref{section simplifying ST}).
Therefore 
$$c_2(\Wbf , \Wbf' ) = FT + ST = a_1 + \|A_2\|^2 + \|A_3-A_4\|^2 \geq 0 \;,$$
and using Lemma \ref{Characterization of critical points in our settings}
we conclude that $\Wbf$ is a second-order critical point. \\

We are now in a position to prove Theorem~\ref{main theorem} as a direct corollary from the above results.

\begin{proof}[Proof of Theorem~\ref{main theorem}]
\label{pf:thm1}
The classification into global minimizers, strict saddle points, and non-strict saddle points follows directly from Propositions~\ref{main prop partie baldi}, \ref{main prop subtle strict saddles}, and \ref{main prop non strict saddles} above. As for the fact that
\[
W_H \cdots W_1 = U_{\Ss}U_{\Ss}^T \Sigma_{YX}\Sigma_{XX}^{-1} \in \argmin_{R \in \mathbb{R}^{d_y \times d_x}, \rk(R) \leq r} \|RX - Y\|^2
\]
when $\Ss = \llbracket 1,r \rrbracket$, it follows from Proposition~\ref{global map and critical values} above and from Lemma~\ref{lemma yun} in Appendix~\ref{Appendix Notation and general useful lemmas and properties}.
\end{proof}

\section{Conclusion}\label{sec:conclusion}
We studied the optimization landscape of linear neural networks of arbitrary depth with the square loss.
We first derived a necessary condition for being a first-order critical point by associating any of them with a set of eigenvectors of a data-dependent matrix.
We then provided a complete characterization of the landscape at order $2$ by distinguishing between global minimizers, strict saddle points, and non-strict saddle points. As a by-product of this analysis, we exhibited large sets of strict and non-strict saddle points and derived an explicit parameterization of all global minimizers. Our second-order characterization also sheds some light on the implicit regularization that may be induced by first-order algorithms, by proving that non-strict saddle points and some strict saddle points are among the global minimizers of the \RCLRP. It also helps re-interpret a recent convergence result, stating that gradient descent with Xavier initialization converges to a global minimum for any wide enough deep linear network.

\acks{
Our work has benefited from the AI Interdisciplinary Institute ANITI.
ANITI is funded by the French "Investing for the Future – PIA3" program under the Grant agreement n°ANR-19-PI3A-0004.
The authors gratefully acknowledge the support of the DEEL project.\footnote{\url{https://www.deel.ai/}} 
El Mehdi Achour acknowledges as well  funding by the Deutsche Forschungsgemeinschaft (DFG, German Research Foundation) – SFB 1481 – 442047500.
\textcolor{black}{We would like to thank the anonymous reviewer who made the connection between 0-balancedness and tightenedness.}
}

\appendix

\section{Notation and Useful Properties}\label{Appendix Notation and general useful lemmas and properties}

In this section, we define some additional notation and terminology that will be used through all subsequent appendices. We also state simple linear algebra facts (Section~\ref{sec:simpleLAfacts}), together with some properties about the Moore-Penrose inverse (Section~\ref{annexe pseudo inverse}).
Since most of the proofs rely on linear algebra, we recommend the unfamiliar reader to check classical textbooks. \\

 \textbf{Additional notation:} If a matrix $A$ has already a subscript like $W_H$ for example, we denote  by $(W_H)_{.,i}$ the $i$-th column and by $(W_H)_{.,J}$ the sub-matrix obtained by concatenating the column vectors $(W_H)_{.,i}$, for all $i \in \mathcal{J}$. Also $(W_H)_{i,.}$ denotes the $i$-th row of $W_H$ and $(W_H)_{\mathcal{I},.}$ the sub-matrix obtained by concatenating the line vectors $(W_H)_{i,.}$, for all $i \in \mathcal{I}$.
More generally $(W_H)_{\mathcal{I}, \mathcal{J}}$ denotes the matrix $W_H$ restricted to the index set $\mathcal{I} \times \mathcal{J}$.
For instance, $(W_H)_{1:r,r+1:d_{H-1}} \in \mathbb{R}^{r \times (d_{H-1}-r)}$ is the matrix formed from $W_H$ by keeping the rows from $1$ to $r$ and the columns from $r+1$ to $d_{H-1}$.
The symbol $\delta_{i,j}$ denotes the Kronecker index which equal to 0 if $i \neq j$ and 1 if $i=j$. \\

Also, we define the partial gradients with respect to each weight matrix as follows.

\subsection{Partial Gradients}
\begin{definition}[gradient and partial gradients of $L$]
Since the input $\Wbf = (W_H, \ldots, W_1)$ of $L(\Wbf)$ is not a vector but a sequence of matrices, we define the gradient $\nabla L(\Wbf)$ of $L$ at $\Wbf$ with a similar format :
$$ \nabla L(\Wbf) = \left( \nabla_{W_H} L(\Wbf), \ldots , \nabla_{W_1} L(\Wbf)\right)\;, $$
where each partial gradient $\nabla_{W_h} L(\Wbf) \in \mathbb{R}^{d_h \times d_{h-1}}$ is the matrix whose entries are the partial derivatives $\frac{\partial L}{\partial (W_h)_{i,j}}$ for $i=1,..,d_h$ and $j=1,..,d_{h-1}$
\end{definition}

The next lemma provides explicit formulas for the partial gradients of $L$.
A proof can be found at the end of \cite{yun2018global}.
\begin{lemma} \label{partial derivatives of L}
    Let $h \in \llbracket 2,H-1 \rrbracket$.
    The partial gradient of $L$ with respect to $W_h$ is:
    $$\nabla_{W_h} L(\Wbf)  = 2 (W_H \cdots W_{h+1})^T(W_H \cdots W_1 \Sigma_{X X} - \Sigma_{Y X})(W_{h-1} \cdots W_1)^T \;.$$
    We also have the partial gradient with respect to $W_H$:
    $$\nabla_{W_H} L(\Wbf)  = 2 (W_H \cdots W_1 \Sigma_{X X} - \Sigma_{Y X})(W_{H-1} \cdots W_1)^T \;.$$
    Finally, the partial gradient with respect to $W_1$ is:
    $$\nabla_{W_1} L(\Wbf)  = 2 (W_H \cdots W_{2})^T(W_H \cdots W_1 \Sigma_{X X} - \Sigma_{Y X}) \;.$$
\end{lemma}

\subsection{Simple Linear Algebra Facts}
\label{sec:simpleLAfacts}
Recall that $\Sigma^{1/2} = \Sigma_{Y X} \Sigma_{X X}^{-1} X$ and $\Sigma = \Sigma^{1/2} (\Sigma^{1/2})^T = \Sigma_{Y X} \Sigma_{X X}^{-1} \Sigma_{X Y}.$
Recall also from \eqref{svd de sigma 1/2} that $\Sigma^{1/2} = U \Delta V^T$ is a Singular Value Decomposition, where $U \in \mathbb{R}^{d_y \times d_y}$ and $V \in \mathbb{R}^{m \times m}$ are orthogonal matrices.
\begin{lemma}\label{sigma invertible}
    Suppose Assumption \ref{Assump H} in Section \ref{Settings} holds true.
    Then $\Sigma$ is invertible.
\end{lemma}
\begin{proof}
Given the definition of $\Sigma^{1/2}$ 
,
it is a standard fact of linear algebra that $\rk(\Sigma^{1/2}) = \rk(\Sigma_{Y X} \Sigma_{X X}^{-1} X) \leq \rk(\Sigma_{YX})$.
On the other hand, $\rk(\Sigma^{1/2}) = \rk(\Sigma_{Y X} \Sigma_{X X}^{-1} X) \geq \rk(\Sigma_{Y X} \Sigma_{X X}^{-1} XX^T) = \rk(\Sigma_{Y X})$ since $\Sigma_{XX} = XX^T$.
Therefore $\rk(\Sigma^{1/2}) = \rk(\Sigma_{YX}) = d_y$ by Assumption \ref{Assump H}.
Finally, using another fact of linear algebra we have $\rk(\Sigma) = \rk(\Sigma^{1/2} (\Sigma^{1/2})^T) = \rk(\Sigma^{1/2})$, and therefore $\rk(\Sigma) = d_y$.
Hence, $\Sigma$ is invertible.
\end{proof}
The next lemma is about global minimizers of the \RCLRP.
\begin{lemma}\label{lemma yun}
    Suppose Assumption \ref{Assump H} in Section \ref{Settings} holds true.
    Let $\Ss = \llbracket 1,r \rrbracket$. We have
    $$U_{\Ss}U_{\Ss}^T \Sigma_{YX}\Sigma_{XX}^{-1} \in \argmin_{R \in \mathbb{R}^{d_y \times d_x},\rk(R) \leq r} \|RX - Y\|^2.$$
\end{lemma}
\begin{proof}
A proof can be found in \cite{yun2018global}.
\end{proof}

We now present a lemma with elementary properties that we will use frequently and that are related to the orthogonality of $U$. The proof is straightforward.

\begin{lemma} \label{properties of orhtogonality}
    We have the following properties related to the orthogonality of the matrix $U$:
    \begin{itemize}
        \item We have $I_{d_y} = U U^T = U^T U$.
        \item For any $i, j \in \llbracket 1 , d_y \rrbracket$, we have $U_i^T U_j = \delta_{i,j}$.
        \item For any $I,J \subset \llbracket 1,d_y \rrbracket$ such that $I \cap J = \emptyset$, we have $U_I^T U_J = 0_{|I| \times |J|}$.
        \item For any $I,J \subset \llbracket 1,d_y \rrbracket$ such that $I \cap J = \emptyset$ and $I \cup J = \llbracket 1,d_y \rrbracket$, we have $ I_{d_y} = U_{I} U_{I}^T + U_{J} U_{J}^T$.
        \item For any $J \subset \llbracket 1,d_y \rrbracket$, we have $U_J^T U_J = I_{|J|}$ and $\rk(U_J U_J^T) = |J|$.
    \end{itemize}
\end{lemma}
Note that the same applies also to the other orthogonal matrix $V \in \mathbb{R}^{m \times m}$ appearing in the Singular Value Decomposition of $\Sigma^{1/2}$ (we only replace $d_y$ by $m$).

Another useful lemma is the following:

\begin{lemma} \label{U_S sigma U_Q = 0}
    For any $I,J \subset \llbracket 1,d_y \rrbracket$ such that $I \cap J = \emptyset$, we have
    $$U_I^T \Sigma U_J = 0_{|I| \times |J|}.$$
    In particular, for any $\Ss \subset \llbracket1,d_y\rrbracket$ and $Q= \llbracket1,d_y\rrbracket \setminus \Ss$, we have $ U_{\Ss}^T \Sigma U_Q = 0 \;.$
\end{lemma}

\begin{proof}
    We have, for any $k \in \llbracket 1,d_y \rrbracket$, $\Sigma U_k = \lambda_k U_k$. Hence for $j \neq k$ we have   $U_j^T \Sigma U_k = \lambda_k U_j^T U_k = 0$ since $U$ is orthogonal.
    Therefore, if we take two disjoint sets $J=\{j_1, \ldots ,j_p\}, K=\{k_1, \ldots ,k_n\} \subset \llbracket1 , d_y \rrbracket$, the coefficient in the position $(l,m)$ of the matrix $U_J^T \Sigma U_K$ is equal to $U_{j_l} \Sigma U_{k_m}$ which is zero, since $j_l \neq k_m$.
    Therefore, $U_J^T \Sigma U_K = 0$.
    In particular, $U_{\Ss}^T \Sigma U_Q = 0$.
\end{proof}

\subsection{The Moore-Penrose Inverse and its Properties}\label{annexe pseudo inverse}
The Moore-Penrose inverse is the most known and used generalized inverse\footnote{\url{en.wikipedia.org/wiki/Moore-Penrose_inverse}}.
It is defined as follows: \\
For $A \in \mathbb{R}^{m \times n}$, the pseudo-inverse of $A$ is defined as the matrix $A^{+} \in \mathbb{R}^{n \times m}$ which satisfies the 4 following criteria known as the Moore-Penrose conditions:
\begin{enumerate}
    \item $A A^{+} A=A$.
    \item $A^{+} A A^{+}=A^{+}$.
    \item $\left(A A^{+}\right)^{T}=A A^{+}$.
    \item $\left(A^{+} A\right)^{T}=A^{+} A$.
\end{enumerate}
$A^{+}$ exists for any matrix $A$ and is unique.
We also have the following properties:
\begin{enumerate}[label=(\roman*)]
    \item $A^{+}=\left(A^{T} A\right)^{+} A^{T}.$
    \item $\rk(A) = \rk(A^{+}) = \rk(AA^{+}) = \rk(A^{+}A).$
    \item If the linear system $Ax=b$ has any solutions, they are all given by $$ x = A^{+} b + (I - A^{+}A) w $$ for arbitrary vector $w$. This is equivalent to $$ x = A^{+} b + u $$ for arbitrary $u \in \text{Ker}(A)$.
    \item $P_A := AA^{+}$ is the orthogonal projection onto the range of $A$, and is therefore symmetric ($P_A^T = P_A$) (follows from 3) and idempotent ($P_A^2 = P_A$) (follows from 1). 
    \item $I_n - A^{+}A$ is the orthogonal projector onto the kernel of $A$.
\end{enumerate}

\section{Propositions and Lemmas for First-order Critical Points}\label{Appendix lemmas for first-order critical points}

In this section, we prove all lemmas about first-order critical points.
We start by stating some preliminary results.
\subsection{Preliminaries}\label{Appendix B.1}

The following lemma gives a necessary condition for $\Wbf$ to be a first-order critical point.
It also provides the global map of the network, defined by $W_H \cdots W_1$.
Finally, it states that the projection matrix $P_K$ and $\Sigma$ commute, where $K=W_H \cdots W_2$.
This is key in the rest of the analysis.

\begin{lemma} \label{lemma W_1 + global map and commutation}
Suppose Assumption \ref{Assump H} in Section \ref{Settings} holds true.
Let $\Wbf = (W_H,\ldots,W_1)$ be a first-order critical point of $L$.
We define $K = W_H \cdots W_2$ and $W = W_H W_{H-1} \cdots  W_1 = K W_1$.
Then, we have $$ W_1 = K^{+} \Sigma_{Y X} \Sigma_{X X}^{-1} + M \;,$$ where  $M \in \mathbb{R}^{d_1 \times d_x}$ is such that $KM=0$ and $K^{+}$ is the Moore-Penrose inverse of $K$ (see Appendix \ref{annexe pseudo inverse}).
As a consequence,
$$
\begin{cases}
    W = P_K \Sigma_{Y X} \Sigma_{X X}^{-1} \\
    \rk(W) = \rk(P_K) = \rk(K)
\end{cases}
$$
where we recall that $P_K = KK^{+} \in \mathbb{R}^{d_y \times d_y}$ is the matrix of the orthogonal projection onto the range of $K$.
Finally,
$$\Sigma P_K = P_K \Sigma \;.$$
\end{lemma}
Note that $\Sigma_{Y X} \Sigma_{X X}^{-1}$ is the global minimizer of the problem with one layer (i.e the classical linear regression problem).
Therefore, the global map $W_H \cdots W_1$ of any first-order critical point of $L$ is equal to the
global minimizer of the linear regression projected onto the column space of $K$.
\begin{proof}
Let $\Wbf = (W_H,\ldots,W_1)$ be a first-order critical point of $L$.
In particular, the partial gradients of $L$ with respect to $W_1$ and $W_H$ are equal to zero at $\Wbf$.
Using Lemma \ref{partial derivatives of L}, this implies

\[
\left\{
\begin{array}{r@{\hspace{1mm}}c@{\hspace{1mm}}l}
(W_H \cdots W_2)^T W_{H} \cdots W_{1} \Sigma_{X X} &=& (W_H \cdots W_2)^T \Sigma_{Y X}  \\
W_{H} \cdots W_{1}\Sigma_{X X} (W_{H-1} \cdots W_1)^T &=&  \Sigma_{Y X} (W_{H-1} \cdots W_1)^T \;.
\end{array}
\right.
\]
We substitute in these equations $K = W_H W_{H-1} \cdots W_2$ and $W = W_H W_{H-1} \cdots  W_1 = K W_1$. Using that $\Sigma_{XX} $ is invertible, and multiplying the second equation on the right by $W_H^{T}$, we obtain that any critical point of $L$ satisfies
\begin{equation} \label{eq:2}
\begin{cases}
K^{T}K W_1 = K^{T} \Sigma_{Y X} \Sigma_{X X}^{-1}  \\
W \Sigma_{XX} W^{T} = \Sigma_{Y X} W^{T} \;.
\end{cases}
\end{equation}
The first equation implies $W_1 = (K^{T}K)^{+}K^T \Sigma_{Y X} \Sigma_{X X}^{-1} + M$, where  $M \in \mathbb{R}^{d_1 \times d_x}$ is such that $K^{T}KM=0$ (see Property (iii) in the reminder on Moore-Penrose inverse in Appendix \ref{annexe pseudo inverse}). \\
We have $(K^{T}K)^{+}K^{T}$ = $K^{+}$(see Property (i) in Appendix \ref{annexe pseudo inverse}) and a standard fact of linear algebra is that $\text{Ker}(K^T K) = \text{Ker}(K)$. \\
Therefore, using these properties, we obtain $W_1 = K^{+}\Sigma_{Y X} \Sigma_{X X}^{-1} + M$, where $KM=0$.
This proves the first statement of the lemma.
We then have,
\begin{align}
W &= K W_1 =K K^{+} \Sigma_{Y X} \Sigma_{X X}^{-1} + KM = P_K \Sigma_{Y X} \Sigma_{X X}^{-1} \;. \label{eq:3}
\end{align}
where $P_K=KK^{+}$ is the orthogonal projection matrix onto the column space of $K$ (see Appendix \ref{annexe pseudo inverse}).
Using Assumption \ref{Assump H}, we have that $\Sigma_{Y X} \Sigma_{X X}^{-1}$ is of full row rank, hence 
\begin{align}\label{rk(W) = rk(K)}
    \rk(W) = \rk(P_K \Sigma_{Y X} \Sigma_{X X}^{-1}) = \rk(P_K) = \rk(K) \;,
\end{align}
where the last equality comes from the property (ii) in Section \ref{annexe pseudo inverse}.
Therefore, \eqref{eq:2} and \eqref{rk(W) = rk(K)} prove the second statement of the lemma. \\
To prove that $\Sigma P_K = P_K \Sigma$, we remark that, using the second equation in \eqref{eq:2}, $\Sigma_{YX} W^T = W\Sigma_{X X}W^T $ and since $W\Sigma_{X X}W^T$ is symmetric and $(\Sigma_{Y X})^T = \Sigma_{X Y}$, we have
$$  \Sigma_{YX} W^T = W \Sigma_{X Y} \;.$$
Substituting the expression of $W$ from \eqref{eq:3}, and since $P_K$ and $\Sigma_{X X}^{-1}$ are symmetric, we have
\begin{equation*}
\Sigma_{Y X} \Sigma_{X X}^{-1} \Sigma_{X Y} P_K = P_K \Sigma_{Y X} \Sigma_{X X}^{-1} \Sigma_{X Y} \;.
\end{equation*}
Using the definition of $\Sigma$, this can be rewritten as
$$\Sigma P_K = P_K \Sigma \;,$$
which concludes the proof.
\end{proof}

\begin{lemma}\label{P_K = U_S U_S^T}
    Suppose Assumption \ref{Assump H} in Section \ref{Settings} holds true.
    Let $\Wbf = (W_H, \ldots, W_1)$ be a first-order critical point of $L$.
    We set $K = W_H \cdots W_2$ and $r = \rk(W_H \cdots W_1)$.\\
    There exists a unique subset $\Ss \subset\llbracket 1,d_y \rrbracket$ of size $r$ such that: $$P_K =U \mathcal{I}^{\Ss} U^T = U_{\Ss} U_{\Ss}^T,$$
    where $\mathcal{I}^{\Ss} \in \mathbb{R}^{d_y \times d_y}$ is the diagonal matrix such that, for all $i \in \llbracket 1,d_y \rrbracket$, $(\mathcal{I}^{\Ss})_{i,i} = 1$ if $i \in \Ss$ and $0$ otherwise. \\
\end{lemma}
\begin{proof}
Let $\Wbf = (W_H, \ldots, W_1)$ be a first-order critical point of $L$.
Using Lemma \ref{lemma W_1 + global map and commutation}, we have  $ \Sigma P_K = P_K \Sigma $.
Substituting the diagonalization of $\Sigma$ from Section \ref{Settings}, this becomes $ U \Lambda U^T P_K = P_K U \Lambda U^T$.
Since $U$ is orthogonal, multiplying by $U^T$ on the left and by $U$ on the right we obtain $\Lambda U^T P_K U = U^T P_K U \Lambda$. Hence, $U^T P_K U$ commutes with a diagonal matrix whose diagonal elements are all distinct.
Therefore, $\Gamma :=U^T P_K U$ is diagonal, and $P_K = U \Gamma U^T$ is a diagonalization of $P_K$.
From Lemma \ref{lemma W_1 + global map and commutation}, we also have $r = \rk(P_K)$.
But, we know that $P_K = KK^{+} \in \mathbb{R}^{d_y \times d_y}$ is the matrix of an orthogonal projection.
Therefore, its eigenvalues are 1 with multiplicity $r$ and 0 with multiplicity $d_y - r$. \\

Therefore, there exists an index set
$\Ss \subset \llbracket 1,d_y \rrbracket$ of size $r$
such that $\Gamma = \mathcal{I}^{\Ss}$ where $\mathcal{I}^{\Ss} \in \mathbb{R}^{d_y \times d_y}$ is the diagonal matrix such that, for all $i \in \llbracket 1,d_y \rrbracket$, $(\mathcal{I}^{\Ss})_{i,i} = 1$ if $i \in \Ss$ and $0$ otherwise. \\
Therefore,$$P_K=U \mathcal{I}^{\Ss}U^T=U \mathcal{I}^{\Ss} \mathcal{I}^{\Ss}U^T=U_{\Ss} U_{\Ss}^T.$$
If there exist $\Ss'$ such that $\Gamma = \mathcal{I}^{\Ss'}$, we get $P_K = U\mathcal{I}^{\Ss}U^T = U \mathcal{I}^{\Ss'} U^T$ which implies $\mathcal{I}^{\Ss} = \mathcal{I}^{\Ss'}$, hence $\Ss = \Ss'$.
Therefore, $\Ss$ is unique.
\end{proof}

\subsection{Proof of Proposition \ref{global map and critical values} } \label{proof of global map and critical values}

In this proof, we use Lemmas \ref{lemma W_1 + global map and commutation} and \ref{P_K = U_S U_S^T} stated and proved in the previous section. \\
Recall that $\lambda_1 > \cdots > \lambda_{d_y}$ are the eigenvalues of $\Sigma = \Sigma_{YX} \Sigma_{XX}^{-1} \Sigma_{XY} \in \mathbb{R}^{d_y \times d_y}$.\\
Let $\Wbf = (W_H, \ldots, W_1)$ be a first-order critical point of $L$.
We set $K= W_H \cdots W_2$, $r =\rk(W_H \cdots W_1)$. 
Using Lemma \ref{P_K = U_S U_S^T}, there exists a unique subset $\Ss \subset\llbracket 1,d_y \rrbracket$ of size $r$ such that: $$P_K = U_{\Ss} U_{\Ss}^T.$$
Therefore, using Lemma \ref{lemma W_1 + global map and commutation},
$$ W_H \cdots W_1 = P_K \Sigma_{Y X} \Sigma_{X X}^{-1} = U_{\Ss} U_{\Ss}^T \Sigma_{Y X} \Sigma_{X X}^{-1}.$$
This proves the first statement of Proposition \ref{global map and critical values}.\\
To prove the second statement, notice that we have
\begin{align*}
    L(\Wbf) &= \|WX-Y\|^2 \\
    &= \|WX\|^2 - 2 \left\langle  WX \ , \ Y  \right\rangle + \|Y\|^2 \\
    &= \tr(W \Sigma_{XX}W^T) - 2 \tr(W \Sigma_{XY}) + \tr(\Sigma_{YY}) \\
    &= \tr( U_{\Ss} U_{\Ss}^T \Sigma_{YX} \Sigma_{XX}^{-1} \Sigma_{XX}  \Sigma_{XX}^{-1} \Sigma_{XY} U_{\Ss} U_{\Ss}^T)  - 2 \tr( U_{\Ss} U_{\Ss}^T \Sigma_{YX} \Sigma_{XX}^{-1} \Sigma_{XY}) + \tr(\Sigma_{YY}) \\
    &= \tr( U_{\Ss} U_{\Ss}^T U_{\Ss} U_{\Ss}^T \Sigma) - 2 \tr(U_{\Ss} U_{\Ss}^T \Sigma) + \tr(\Sigma_{YY}) \\
\end{align*}
Since $U_{\Ss}^T U_{\Ss} = I_r$ (see Lemma \ref{properties of orhtogonality}), using Lemma \ref{P_K = U_S U_S^T} and the fact that $U$ diagonalizes $\Sigma$, this becomes
\begin{align*}
    L(\Wbf) &= \tr(\Sigma_{YY}) - \tr(U_{\Ss} U_{\Ss}^T \Sigma)  \\
    &= \tr(\Sigma_{YY}) - \tr( U \mathcal{I}^{\Ss} U^T U \Lambda U^T ) \\
    &= \tr(\Sigma_{YY}) - \tr( \mathcal{I}^{\Ss} U^T U \Lambda U^T U) \\
    &= \tr(\Sigma_{YY}) - \tr( \mathcal{I}^{\Ss} \Lambda) \\
    &= \tr(\Sigma_{YY}) - \sum_{i \in \Ss} \lambda_i \;.
\end{align*}
This proves the second and last statement of Proposition \ref{global map and critical values}.

\subsection{Lemma \ref{clem}}\label{proof of clem}
In this section we state and prove a lemma about first-order critical points which will be useful in various proofs.
This lemma gives a simpler form for $K = W_H \cdots W_2$ and $W_1$.

\begin{lemma}\label{clem}
Suppose Assumption \ref{Assump H} in Section \ref{Settings} holds true.
Let $\Wbf = (W_H, \ldots , W_1)$ be a first-order critical point of $L$ associated with $\Ss$.
We set $r = \rk(W_H \cdots W_1)$. \\
Then there exists an invertible matrix $D \in \mathbb{R}^{d_{1} \times d_{1}}$, a matrix $M \in \mathbb{R}^{d_1 \times d_x}$ satisfying $W_H \cdots W_2 M=0$, 
such that:
$$K=W_H \cdots W_2 = \biggl[ U_{\Ss} \quad 0_{d_y \times (d_1-r)} \biggr] D$$ and $$W_1 = D^{-1} \left[
\begin{array}{c}
U_{\Ss}^T \Sigma_{Y X} \Sigma_{X X}^{-1} \\
0_{(d_1-r) \times d_x}
\end{array}
\right] + M \;.$$
\end{lemma}

Note that the result is still true when $r=0$, provided that $U_{\emptyset} \in \mathbb{R}^{d_y \times 0}$.

To prove Lemma \ref{clem}, we use Lemmas \ref{lemma W_1 + global map and commutation} and \ref{P_K = U_S U_S^T} stated and proved in the preliminaries of Appendix \ref{Appendix B.1}.
We will also need the following lemma
\begin{lemma} \label{existence of D invertible}
    Let $n$ be a positive integer and $\emptyset \neq \Ss \subset\llbracket 1,d_y \rrbracket $ such that $n \geq r:=|\Ss|$.
    Let $A \in \mathbb{R}^{d_y \times n}$ such that $AA^{+} = U_{\Ss} U_{\Ss}^T$. Then there exists an invertible matrix $D \in \mathbb{R}^{n \times n}$ such that $$A = [ U_{\Ss} \quad 0_{d_y \times (n-r)} ] D$$ and $$A^{+} = D^{-1}\left[
    \begin{array}{c}
    U_{\Ss}^T \\
    0_{(n-r) \times d_y}
    \end{array}
    \right] \;.$$
\end{lemma}
\begin{proof}[Proof of Lemma \ref{existence of D invertible}]

The matrix $I_n - A^{+}A$ is the orthogonal projection onto $\text{Ker(A)}$ (see Appendix \ref{annexe pseudo inverse}), hence $$\rk(I_n - A^{+}A) = \text{dim Ker}(A) = n-\rk(A) $$
But we have (see Property (ii) in Appendix \ref{annexe pseudo inverse}) $\rk(A^{+}A) = \rk(A) = \rk(AA^{+})$ and, using Lemma \ref{properties of orhtogonality},  $\rk(A^{+}A) = \rk(U_{\Ss} U_{\Ss}^T) = r$.
Therefore, $\rk(A)=r$ and
$$\rk(I_n - A^{+}A) = n-r.$$

Let $B \in \mathbb{R}^{n \times (n-r)}$ and $C \in \mathbb{R}^{(n-r) \times n}$ be such that $ I_n - A^{+}A = BC $ (such matrices can be obtained by considering the Singular Value Decomposition of $I_n - A^{+}A$). \\
Denoting $D = \left[
    \begin{array}{c}
    U_{\Ss}^T A \\
    C
    \end{array}
    \right] \in \mathbb{R}^{n \times n}$, we have
\begin{align*}
    [A^{+} U_{\Ss} \ , \ B ] D &= [A^{+} U_{\Ss} \ , \ B ] \left[
    \begin{array}{c}
    U_{\Ss}^T A \\
    C
    \end{array}
    \right] = A^{+} U_{\Ss} U_{\Ss}^T A + BC = A^{+}AA^{+}A + I_n - A^{+}A \;.
\end{align*} 
Using Criteria 1 in Appendix \ref{annexe pseudo inverse} we obtain
$$ [A^{+} U_{\Ss} \ , \ B ] D = A^{+}A + I_n - A^{+}A = I_n \;.$$
Therefore, $D$ is invertible and $D^{-1} = [A^{+} U_{\Ss} \ , \ B ]$. We have 
\begin{align*}
    [U_{\Ss} \ , \ 0_{d_y \times (n-r)} ] D = [ U_{\Ss} \ , \ 0_{d_y \times (n-r)} ] \left[
    \begin{array}{c}
    U_{\Ss}^T A \\
    C
    \end{array}
    \right] = U_{\Ss} U_{\Ss}^T A = AA^{+}A = A \;,
\end{align*}
where the last equality follows from Criteria 1 in Appendix \ref{annexe pseudo inverse}.
This proves the first equality of Lemma \ref{existence of D invertible}
Finally,
\begin{align*}
    D^{-1} \left[
    \begin{array}{c}
    U_{\Ss}^T \\
    0_{(n-r) \times d_y}
    \end{array}
    \right] = [A^{+} U_{\Ss} \ , \ B ] \left[
    \begin{array}{c}
    U_{\Ss}^T \\
    0_{(n-r) \times d_y}
    \end{array}
    \right] 
    = A^{+} U_{\Ss} U_{\Ss}^T 
    = A^{+}AA^{+} 
    = A^{+} \;,
\end{align*}
where the last equality follows again from Criteria 2 in Appendix \ref{annexe pseudo inverse}.
This concludes the proof of Lemma \ref{existence of D invertible}.
\end{proof}

Now we prove Lemma \ref{clem}.

\begin{proof}[Proof of Lemma \ref{clem}]

Let $\Wbf = (W_H, \ldots , W_1)$ be a first-order critical point of $L$ associated with $\Ss$ and $r=\rk(W_H \cdots W_1)$.\\
Using Lemma \ref{lemma W_1 + global map and commutation}, we have $r=\rk(W_H \cdots W_2).$ \\
If $r=0$, the conclusion of Lemma \ref{clem} is trivial because of the convention $U_{\emptyset} \in \mathbb{R}^{d_y \times 0}$.\\
When $r \geq 1$,
using Lemma \ref{lemma W_1 + global map and commutation} and Proposition \ref{global map and critical values}, we have $W_H \cdots W_1 = P_K \Sigma_{Y X} \Sigma_{X X}^{-1} = U_{\Ss} U_{\Ss}^T \Sigma_{Y X} \Sigma_{X X}^{-1} $.
Since $\Sigma_{Y X}$ is of full row rank this implies
$P_K = KK^{+} = U_{\Ss} U_{\Ss}^T $.
Therefore, we can apply Lemma \ref{existence of D invertible} with $n=d_1$ and $A=K$ to conclude that there exists an invertible matrix $D \in \mathbb{R}^{d_1 \times d_1}$ such that 
$$ K = [U_{\Ss} \ , \ 0_{d_y \times (d_1-r)} ] D$$
which is the form of $K$ in Lemma \ref{clem}.
Moreover, Lemma \ref{existence of D invertible} also guarantees that
$$ K^{+} = D^{-1}\left[
    \begin{array}{c}
    U_{\Ss}^T \\
    0_{(d_1-r) \times d_y}
    \end{array}
    \right] \;. $$
Using Lemma \ref{lemma W_1 + global map and commutation}, we have $W_1 = K^{+}\Sigma_{Y X} \Sigma_{X X}^{-1} + M$ with $KM=0$.
Therefore, $W_1 = D^{-1} \left[
\begin{array}{c}
U_{\Ss}^T \Sigma_{Y X} \Sigma_{X X}^{-1} \\
0_{(d_1-r) \times d_x}
\end{array}
\right] + M$, with $KM=0$.
This concludes the proof of Lemma \ref{clem}.
\end{proof}

\subsection{Proof of Lemma \ref{same nature}}\label{Proof same nature}
For any $h \in \llbracket 1 , H-1 \rrbracket$ let $D_h \in \mathbb{R}^{d_{h} \times d_{h}}$ be an invertible matrix.
We define $ \widetilde{\Wbf} = (\widetilde{W}_H, \ldots, \widetilde{W}_1) $ by
$\widetilde{W}_H = W_H D_{H-1},$ $\widetilde{W}_1 = D_1^{-1} W_1$ and  $\widetilde{W}_h = D_{h}^{-1} W_h D_{h-1}$ for all $h \in \llbracket 2 , H-1 \rrbracket$. \\
Assume that $\Wbf = (W_H, \ldots ,W_1)$ is a first-order critical point.
Then using Lemma \ref{partial derivatives of L} this is equivalent to 

\begin{align} \label{systeme gradient nul}
\begin{cases}
\nabla_{W_h} L(\Wbf)  = 2 (W_H \cdots W_{h+1})^T(W_H \cdots W_1 \Sigma_{X X} - \Sigma_{Y X})(W_{h-1} \cdots W_1)^T = 0 \quad  \forall h \in \llbracket 2,H-1 \rrbracket \\
\nabla_{W_H} L(\Wbf)  = 2 (W_H \cdots W_1 \Sigma_{X X} - \Sigma_{Y X})(W_{H-1} \cdots W_1)^T = 0 \\
\nabla_{W_1} L(\Wbf)  = 2 (W_H \cdots W_{2})^T(W_H \cdots W_1 \Sigma_{X X} - \Sigma_{Y X}) = 0 \;.
\end{cases}
\end{align}
Using the definition of $\widetilde{\Wbf}$ above, we have
$$ 
\begin{cases}
W_H \cdots W_1 = \widetilde{W}_H \cdots \widetilde{W}_1 \\
W_H \cdots W_{h+1} = \widetilde{W}_H \cdots \widetilde{W}_{h+1} D_h^{-1} \quad  \forall h \in \llbracket 1,H-1 \rrbracket \\
W_{h-1} \cdots W_1 = D_{h-1} \widetilde{W}_{h-1} \cdots \widetilde{W}_{1} \quad  \forall h \in \llbracket 2,H \rrbracket  \;. \\
\end{cases}
$$
Therefore \eqref{systeme gradient nul} is equivalent to
\begin{align*}
\begin{cases}
    (D_h^{-1})^T(\widetilde{W}_H \cdots \widetilde{W}_{h+1})^T(\widetilde{W}_H \cdots \widetilde{W}_1\Sigma_{X X} - \Sigma_{Y X})(\widetilde{W}_{h-1} \cdots \widetilde{W}_1)^T D_{h-1}^T = 0 \quad  \forall h \in \llbracket 2,H-1 \rrbracket  \\
    (\widetilde{W}_H \cdots \widetilde{W}_1\Sigma_{X X} - \Sigma_{Y X})(\widetilde{W}_{H-1} \cdots \widetilde{W}_1)^T D_{H-1}^T = 0 \\
    (D_1^{-1})^T(\widetilde{W}_H \cdots \widetilde{W}_{2})^T(\widetilde{W}_H \cdots \widetilde{W}_1 \Sigma_{X X} - \Sigma_{Y X}) = 0 \;.
\end{cases}
\end{align*}
This is equivalent to
\begin{align*} 
\begin{cases}
\nabla_{W_h} L(\widetilde{\Wbf})  = 2 (\widetilde{W}_H \cdots \widetilde{W}_{h+1})^T(\widetilde{W}_H \cdots \widetilde{W}_1 \Sigma_{X X} - \Sigma_{Y X})(\widetilde{W}_{h-1} \cdots \widetilde{W}_1)^T = 0 \quad  \forall h \in \llbracket 2,H-1 \rrbracket \\
\nabla_{W_H} L(\widetilde{\Wbf})  = 2 (\widetilde{W}_H \cdots \widetilde{W}_1 \Sigma_{X X} - \Sigma_{Y X})(\widetilde{W}_{H-1} \cdots \widetilde{W}_1)^T = 0  \\
\nabla_{W_1} L(\widetilde{\Wbf})  = 2 (\widetilde{W}_H \cdots \widetilde{W}_{2})^T(\widetilde{W}_H \cdots \widetilde{W}_1 \Sigma_{X X} - \Sigma_{Y X}) = 0 \;.
\end{cases}
\end{align*}
which is equivalent to $\nabla_{W_h} L(\widetilde{\Wbf})= 0$, for all $h \in \llbracket 1 , H \rrbracket \;.$
Therefore, $\Wbf$ is a first-order critical point if and only if $\widetilde{\Wbf}$ is a first-order critical point.
This proves the first part of the proposition. \\
Now assume that $\Wbf = (W_H, \ldots ,W_1)$ is a first-order critical point such that it is not a second-order critical point.
Note that from the first part of the proof $\widetilde{\Wbf} = (\widetilde{W}_H, \ldots ,\widetilde{W}_1)$ is also a first-order critical point.
Let us prove that $\widetilde{\Wbf}$ is not a second-order critical point.
Using Lemma \ref{Characterization of critical points in our settings}, since $\Wbf$ is not a second-order critical point, there exist $\Wbf' = (W'_H, \ldots ,W'_1)$ such that, if we denote $\Wbf(t) = \Wbf + t \Wbf'$, the second-order term of $L(\Wbf(t))$ is strictly negative i.e $c_2(\Wbf,\Wbf')<0$.
We will prove that there exist $\widetilde{\Wbf}'$ such that $c_2(\widetilde{\Wbf}, \widetilde{\Wbf}') < 0$ and, using again Lemma \ref{Characterization of critical points in our settings}, we conclude. \\
As already said, we set $W_h(t) = W_h + t W'_h$, for all $h \in \llbracket 1, H \rrbracket$.
We denote 
$$ 
\begin{cases}
    \widetilde{W}_H(t) = \widetilde{W}_H + t \widetilde{W}'_H  = \widetilde{W}_H + t W'_H D_{H-1} \\
    \widetilde{W}_1(t) = \widetilde{W}_1 + t \widetilde{W}'_1  = \widetilde{W}_1 + t D_1^{-1}W'_1 \\
    \widetilde{W}_h(t) = \widetilde{W}_h + t \widetilde{W}'_h  = \widetilde{W}_h + t D_h^{-1}W'_h D_{h-1} \quad \forall h \in \llbracket 2,H-1 \rrbracket \\
    \widetilde{\Wbf}' = (\widetilde{W}'_H , \ldots , \widetilde{W}'_1) \;.
\end{cases}
$$
Hence, we have (where $\prod_{h=H-1}^{2} A_h$ should read as $A_{H-1} \cdots A_2$)
\begin{align*}
    & \widetilde{W}_H(t) \cdots \widetilde{W}_1(t) \\ &= (W_H D_{H-1} + t W'_H D_{H-1}) \left(\prod_{h=H-1}^{2}(D_{h}^{-1} W_{h} D_{h-1} + t D_{h}^{-1} W'_{h} D_{h-1})\right)(D_1^{-1}W_1 + t D_1^{-1}W'_1) \\
    &= (W_H + t W'_H) \cdots (W_1 + t W'_1) \\
    &= W_H(t) \cdots W_1(t) \;.
\end{align*}
Therefore, $L(\widetilde{\Wbf}(t)) = L(\Wbf(t))$ and
\begin{align*}
    c_2(\widetilde{\Wbf}, \widetilde{\Wbf}') = c_2(\Wbf,\Wbf') \;.
\end{align*}
Since by hypothesis $c_2(\Wbf, \Wbf') < 0$, we conclude that $c_2(\widetilde{\Wbf}, \widetilde{\Wbf}') < 0$.
Hence  $(\widetilde{W}_H, \ldots ,\widetilde{W}_1)$ is not a second-order critical point. \\
We prove that if $\widetilde{\Wbf}$ is not a second-order critical point then $\Wbf$ is not a second-order critical point
in the same way, by changing $D_h$ with $D_h^{-1}$ for all $h \in \llbracket 1, H \rrbracket$. This proves the second part of the proposition and concludes the proof.

\subsection{Proof of Proposition \ref{reciproque param}} \label{proof reciproque param}
Let $\Ss \subset \llbracket 1,d_y \rrbracket$ of size $r \in \llbracket 0,r_{max} \rrbracket$ and $Q = \llbracket 1,d_y \rrbracket \setminus \Ss$.
    Let $Z_H \in \mathbb{R}^{(d_y-r) \times (d_{H-1}-r)}$, $Z_1 \in \mathbb{R}^{(d_1-r) \times d_x}$ and $Z_h \in \mathbb{R}^{(d_h-r) \times (d_{h-1}-r)}$ for $h \in \llbracket 2 , H-1 \rrbracket$.
    Let the parameter of the network $\Wbf = (W_H,\ldots,W_1)$ be defined as follows:
\begin{align}\label{eq param W associé à S}
    \left\{ \begin{array}{l}
    {W}_H = [U_{\Ss} , U_Q Z_H] \\
    {W}_1 = \begin{bmatrix}
    U_{\Ss}^T\Sigma_{YX}\Sigma_{XX}^{-1} \\ 
    Z_1
    \end{bmatrix}  \\
    {W}_h = \left[
    \begin{array}{c  c}
    I_r & 0 \\
    0 & Z_h
    \end{array}
    \right] \quad \forall h \in \llbracket 2 , H-1 \rrbracket \;.
    \end{array}
    \right.
\end{align}
Note that the above definition of $\Wbf$ does not involve the matrices $D_h \in \mathbb{R}^{d_h \times d_h}$.
In fact, using Lemma \ref{same nature}, it suffices to prove that, when $r=r_{max}$ or there exist $h_1 \neq h_2$ such that $Z_{h_1} = 0 $ and $Z_{h_2} = 0$, the $\Wbf$ defined above is a first-order critical point to conclude that Proposition \ref{reciproque param} holds .\\
We have 
\begin{align*}
    W_H \cdots W_1 &=  [U_{\Ss} , U_Q Z_H] \left[
    \begin{array}{c  c}
    I_r & 0 \\
    0 & Z_{H-1}
    \end{array}
    \right]
    \cdots
    \left[
    \begin{array}{c  c}
    I_r & 0 \\
    0 & Z_2
    \end{array}
    \right]
    \begin{bmatrix}
    U_{\Ss}^T\Sigma_{YX}\Sigma_{XX}^{-1} \\ 
    Z_1
    \end{bmatrix} \\
    &= U_{\Ss} U_{\Ss}^T\Sigma_{YX}\Sigma_{XX}^{-1} + U_Q Z_H Z_{H-1} \cdots Z_2 Z_1
\end{align*}
If there exists $h_1 \neq h_2$ such that $Z_{h_1} = 0$ and $Z_{h_2} = 0$, it immediately follows that $W_H \cdots W_1 = U_{\Ss} U_{\Ss}^T\Sigma_{YX}\Sigma_{XX}^{-1} $.\\
If $r=r_{max}$, then there exists $h \in \llbracket 0,H \rrbracket$ such that $r = d_h$.
\begin{itemize}
    \item If $r = d_H = d_y$, then $U_Q \in \mathbb{R}^{d_y \times 0}$ and $Z_H \in \mathbb{R}^{0 \times (d_{H-1}-r)}$, which, using conventions in Section \ref{Settings}, gives 
    \begin{align}\label{block vide 1}
        U_Q Z_H = 0_{d_y \times (d_{H-1}-r)}.
    \end{align}
    Therefore, $W_H \cdots W_1 = U_{\Ss} U_{\Ss}^T\Sigma_{YX}\Sigma_{XX}^{-1} $.
    \item If $r= d_0 = d_x$, then, since $d_x \geq d_y$, we have $r=d_y$, which we have already treated in the previous item.
    \item If $r=d_h$ for some $h \in \llbracket2,H-1\rrbracket$, then
    $Z_{h+1} \in \mathbb{R}^{(d_{h+1}-r) \times 0}$ and $ Z_h \in \mathbb{R}^{0 \times (d_{h-1}-r)}$, which, using the conventions on Section \ref{Settings}, gives 
    \begin{align}\label{block vide 2}
        Z_{h+1}Z_h = 0_{(d_{h+1}-r) \times (d_{h-1}-r)}.
    \end{align}
    Therefore, $W_H \cdots W_1 = U_{\Ss} U_{\Ss}^T\Sigma_{YX}\Sigma_{XX}^{-1} $.
    \item If $r=d_1$, then $Z_{2} \in \mathbb{R}^{(d_{2}-r) \times 0}$ and $ Z_1 \in \mathbb{R}^{0 \times d_x}$, which, using the conventions on Section \ref{Settings}, gives
    \begin{align}\label{block vide 3}
        Z_{2}Z_1 = 0_{(d_{2}-r) \times d_x}.
    \end{align}
    Therefore, $W_H \cdots W_1 = U_{\Ss} U_{\Ss}^T \Sigma_{YX}\Sigma_{XX}^{-1} $.
\end{itemize}
Note that these results still hold if there is more than one layer with the minimum width.\\
Therefore, in all cases, when $r=r_{max}$ or there exist $h_1 \neq h_2$ such that $Z_{h_1}=0$ and $Z_{h_2}=0$ we have, 
\begin{align}\label{eq inter}
W_H \cdots W_1 = U_{\Ss} U_{\Ss}^T \Sigma_{YX}\Sigma_{XX}^{-1} \;.
\end{align}
Let us prove that the gradient of $L$ at $\Wbf$ is equal to zero.\\
Recall that from Lemma \ref{partial derivatives of L} we have
\begin{align*}
\nabla_{W_h} L(\Wbf)  &= 2 (W_H \cdots W_{h+1})^T(W_H \cdots W_1 \Sigma_{X X} - \Sigma_{Y X})(W_{h-1} \cdots W_1)^T \quad \forall h \in \llbracket 2,H-1 \rrbracket \\
\nabla_{W_H} L(\Wbf) & = 2 (W_H \cdots W_1 \Sigma_{X X} - \Sigma_{Y X})(W_{H-1} \cdots W_1)^T \\
\nabla_{W_1} L(\Wbf)  &= 2 (W_H \cdots W_{2})^T(W_H \cdots W_1 \Sigma_{X X} - \Sigma_{Y X}) \;.
\end{align*}
Using \eqref{eq inter} and Lemma \ref{properties of orhtogonality}, we have
\begin{align*}
    W_H \cdots W_1 \Sigma_{X X} - \Sigma_{Y X} &= U_{\Ss} U_{\Ss}^T \Sigma_{YX}\Sigma_{XX}^{-1} \Sigma_{X X} - \Sigma_{Y X} \\
    &= (U_{\Ss} U_{\Ss}^T - I_{d_y})\Sigma_{Y X} \\
    &= - U_{Q} U_{Q}^T \Sigma_{Y X} \;.
\end{align*}
Also, using \eqref{eq param W associé à S}, for all $h \in \llbracket 1,H-1 \rrbracket$,
\begin{align*}
    W_H \cdots W_{h+1} &= [U_{\Ss} , U_Q Z_H]  \left[
    \begin{array}{c  c}
    I_r & 0 \\
    0 & Z_{H-1}
    \end{array}
    \right]
    \cdots
    \left[
    \begin{array}{c  c}
    I_r & 0 \\
    0 & Z_{h+1}
    \end{array}
    \right] \\
    &= [U_{\Ss} , U_Q Z_H Z_{H-1} \cdots Z_{h+1}]
\end{align*}
and, for all $h \in \llbracket 2,H \rrbracket$,
\begin{align*}
    W_{h-1} \cdots W_1 &= \left[
    \begin{array}{c  c}
    I_r & 0 \\
    0 & Z_{h-1}
    \end{array}
    \right]
    \cdots
    \left[
    \begin{array}{c  c}
    I_r & 0 \\
    0 & Z_{2}
    \end{array}
    \right]
    \begin{bmatrix}
    U_{\Ss}^T\Sigma_{YX}\Sigma_{XX}^{-1} \\ 
    Z_1
    \end{bmatrix} \\
    &= \begin{bmatrix}
    U_{\Ss}^T\Sigma_{YX}\Sigma_{XX}^{-1} \\ 
    Z_{h-1} \cdots Z_2 Z_1
    \end{bmatrix} \;.
\end{align*}
We have, for all $h \in \llbracket 2,H-1 \rrbracket$,
\begin{align*}
    \frac{1}{2}\left(\nabla_{W_h} L(\Wbf)\right)^T &= (W_{h-1} \cdots W_1)(W_H \cdots W_1 \Sigma_{X X} - \Sigma_{Y X})^T (W_H \cdots W_{h+1}) \\
    &= - \begin{bmatrix}
    U_{\Ss}^T\Sigma_{YX}\Sigma_{XX}^{-1} \\ 
    Z_{h-1} \cdots Z_2 Z_1
    \end{bmatrix}
    (U_{Q} U_{Q}^T \Sigma_{Y X})^T
    [U_{\Ss} , U_Q Z_H Z_{H-1} \cdots Z_{h+1}] \\
    &= - \begin{bmatrix}
    U_{\Ss}^T\Sigma_{YX}\Sigma_{XX}^{-1} \\ 
    Z_{h-1} \cdots Z_2 Z_1
    \end{bmatrix}
    \Sigma_{X Y} U_{Q} U_{Q}^T 
    [U_{\Ss} , U_Q Z_H Z_{H-1} \cdots Z_{h+1}] \\
    &= - \begin{bmatrix}
    U_{\Ss}^T\Sigma_{YX}\Sigma_{XX}^{-1} \Sigma_{X Y} U_{Q}
    \\
    Z_{h-1} \cdots Z_2 Z_1 \Sigma_{X Y} U_{Q}
    \end{bmatrix}
    [U_{Q}^T U_{\Ss} , U_{Q}^T U_Q Z_H Z_{H-1} \cdots Z_{h+1}] \;.
\end{align*}
Using the definition of $\Sigma$, Lemma \ref{properties of orhtogonality} and Lemma \ref{U_S sigma U_Q = 0}, we have
\begin{align*}
    \frac{1}{2}\left(\nabla_{W_h} L(\Wbf)\right)^T &= - \begin{bmatrix}
    U_{\Ss}^T\Sigma U_{Q}
    \\
    Z_{h-1} \cdots Z_2 Z_1 \Sigma_{X Y} U_{Q}
    \end{bmatrix}
    [0_{(d_y-r) \times r} , Z_H Z_{H-1} \cdots Z_{h+1}] \\
    &= - \begin{bmatrix}
    0_{r \times (d_y-r)}
    \\
    Z_{h-1} \cdots Z_2 Z_1 \Sigma_{X Y} U_{Q}
    \end{bmatrix}
    [0_{(d_y-r) \times r} , Z_H Z_{H-1} \cdots Z_{h+1}] \\
    &= - \begin{bmatrix}
    0_{r \times r} & 0_{r \times (d_h-r)} \\
    0_{(d_{h-1}-r) \times r} & Z_{h-1} \cdots Z_2 Z_1 \Sigma_{X Y} U_{Q} Z_H Z_{H-1} \cdots Z_{h+1}
    \end{bmatrix} \;.
\end{align*}
Proceeding similarly, we obtain
\begin{align*}
    \frac{1}{2}\left(\nabla_{W_H} L(\Wbf)\right)^T &= - \begin{bmatrix}
    0_{r \times d_y} \\
    Z_{H-1} \cdots Z_2 Z_1 \Sigma_{X Y} U_{Q} U_{Q}^T
    \end{bmatrix}
\end{align*}
and
\begin{align*}
    \frac{1}{2}\left(\nabla_{W_1} L(\Wbf)\right)^T &= - [
    0_{d_x \times r} \ , \ \Sigma_{X Y} U_{Q} Z_H Z_{H-1} \cdots Z_{2}
    ] \;.
\end{align*}
If there exists $h_1 \neq h_2$ such that $Z_{h_1} = 0$ and $Z_{h_2} = 0$, we can easily see that the gradient is equal to zero, i.e., $\Wbf$ is a first-order critical point. \\
If $r=r_{max}$, then there exists $h' \in \llbracket 1,H \rrbracket$ such that $r=d_{h'}$.
Using the same arguments as above that yielded \eqref{block vide 1}, \eqref{block vide 2} and \eqref{block vide 3},
we have,
\begin{itemize}
    \item For $h=1$, 
    \begin{itemize}
        \item if $r=d_1$, we have $Z_2 \in \mathbb{R}^{(d_2-r) \times 0}$ and therefore $\Sigma_{X Y} U_{Q} Z_H Z_{H-1} \cdots Z_{2} \in \mathbb{R}^{d_x \times 0}$.
        \item if $r = d_{H}$, then $U_Q Z_{H} = 0_{d_y \times (d_{H-1}-r)}$ and therefore $\Sigma_{X Y} U_{Q} Z_H Z_{H-1} \cdots Z_{2}  = 0_{d_x \times (d_1-r)}$.
        \item if $r = d_{h'}$ for some $h' \in \llbracket2,H-1 \rrbracket$, then $Z_{h'+1} Z_{h'} = 0_{(d_{h'+1}-r) \times (d_{h'-1}-r)}$ and therefore $\Sigma_{X Y} U_{Q} Z_H Z_{H-1} \cdots Z_{2}  = 0_{d_x \times (d_1-r)}$.
    \end{itemize} 
Hence, in all cases, $\nabla_{W_1} L(\Wbf) = 0 $.
    \item For $h=H$,
    \begin{itemize}
        \item if $r = d_{H} = d_y$, then $U_Q U_Q^T = 0_{d_y \times d_y}$ and therefore $Z_{H-1} \cdots Z_2 Z_1 \Sigma_{X Y} U_{Q} U_{Q}^T = 0_{(d_{H-1}-r) \times d_y}$ .
        \item if $r = d_{H-1}$, then $Z_{H-1} \in \mathbb{R}^{0 \times (d_{H-2}-r)}$ and therefore $Z_{H-1} \cdots Z_2 Z_1 \Sigma_{X Y} U_{Q} U_{Q}^T \in \mathbb{R}^{0 \times d_y} $.
        \item if $r=d_{h'}$ for some $h' \in \llbracket 2,H-2 \rrbracket$, then $Z_{h'+1} Z_{h'} = 0_{(d_{h'+1}-r) \times (d_{h'-1}-r)}$ and therefore $Z_{H-1} \cdots Z_2 Z_1 \Sigma_{X Y} U_{Q} U_{Q}^T = 0_{(d_{H-1}-r) \times d_y}$.
        \item if $r=d_{1}$, then $Z_{2} Z_{1} = 0_{(d_{2}-r) \times d_{x}}$ and therefore $Z_{H-1} \cdots Z_2 Z_1 \Sigma_{X Y} U_{Q} U_{Q}^T = 0_{(d_{H-1}-r) \times d_y}$.
    \end{itemize}
Hence, in all cases, $\nabla_{W_H} L(\Wbf) = 0 $.
    \item For $h \in \llbracket2,H-1 \rrbracket$,
    \begin{itemize}
        \item if $r=d_{h-1}$, then $Z_{h-1} \in \mathbb{R}^{0 \times (d_{h-2}-r)}$ and therefore $Z_{h-1} \cdots Z_2 Z_1 \Sigma_{X Y} U_{Q} Z_H Z_{H-1} \cdots Z_{h+1} \in \mathbb{R}^{0 \times (d_h-r)}$.
        \item if $r=d_{h}$, then $Z_{h+1} \in \mathbb{R}^{(d_{h+1}-r) \times 0}$ and therefore $Z_{h-1} \cdots Z_2 Z_1 \Sigma_{X Y} U_{Q} Z_H Z_{H-1} \cdots Z_{h+1} \in \mathbb{R}^{(d_{h-1}-r) \times 0}$.
        \item if $r=d_{H}$, then $U_Q Z_{H} = 0_{d_y \times (d_{H-1}-r)}$ and therefore $Z_{h-1} \cdots Z_2 Z_1 \Sigma_{X Y} U_{Q} Z_H Z_{H-1} \cdots Z_{h+1} = 0_{(d_{h-1}-r) \times (d_h-r)}$.
        \item if $r=d_{1}$, then $Z_{2} Z_{1} = 0_{(d_{2}-r) \times d_x}$ and therefore $Z_{h-1} \cdots Z_2 Z_1 \Sigma_{X Y} U_{Q} Z_H Z_{H-1} \cdots Z_{h+1} = 0_{(d_{h-1}-r) \times (d_h-r)}$.
        \item if $r=d_{h'}$ for some $h' \in \llbracket 2,H-1 \rrbracket \setminus \{h,h-1\}$, then $Z_{h'+1} Z_{h'} = 0_{(d_{h'+1}-r) \times (d_{h'-1}-r)}$ and therefore \\ $Z_{h-1} \cdots Z_2 Z_1 \Sigma_{X Y} U_{Q} Z_H Z_{H-1} \cdots Z_{h+1} = 0_{(d_{h-1}-r) \times (d_h-r)}$.
    \end{itemize}
Hence, in all cases, $\nabla_{W_h} L(\Wbf) = 0 $.
\end{itemize}
Therefore, when $r=r_{max}$, $\Wbf$ is also a first-order critical point of $L$.

\subsection{Proof of Proposition \ref{reciproque S vers W} } \label{proof of reciproque S vers W}

Let $\mathcal{S} \subset \llbracket 1,d_y \rrbracket$ such that $|\Ss| = r \leq r_{max}$, and $Q = \llbracket 1,d_y \rrbracket \setminus \Ss$. \\
We define $\Wbf = (W_H,\ldots,W_1)$ by:
\begin{align*}
        {W}_H &= [U_{\Ss} , 0_{d_y \times (d_{H-1}-r)}]  \\
        {W}_h &= \left[
        \begin{array}{c  c}
        I_r & 0_{r \times (d_{h-1}-r)}  \\
        0_{(d_h-r) \times r} & 0_{(d_h-r) \times (d_{h-1}-r)}
        \end{array}
        \right] \quad \forall h \in \llbracket 2 , H-1 \rrbracket \\
        {W}_1 &= \begin{bmatrix}
        U_{\Ss}^T\Sigma_{YX}\Sigma_{XX}^{-1} \\ 
        0_{(d_1-r) \times d_x}
        \end{bmatrix},
\end{align*}

By Proposition \ref{reciproque param}, $\Wbf$ is a first-order critical point of $L$.
Moreover, we have $W_H \cdots W_1 = U_{\Ss} U_{\Ss}^T\Sigma_{YX}\Sigma_{XX}^{-1}$.
Therefore, $\Wbf$ is a first-order critical point associated with $\Ss$.

\subsection{Proof of Proposition \ref{all blocks ranks geq than r}}\label{proof of all blocks ranks geq than r}

Let $\Wbf = (W_H, \ldots , W_1)$ be a first-order critical point and $r=\rk(W_H \cdots W_1)$, using Proposition \ref{global map and critical values} there exists a unique $\Ss \subset \llbracket1,d_y \rrbracket$ of size $r$ such that
$$ W_H \cdots W_1 = U_{\Ss} U_{\Ss}^T \Sigma_{YX} \Sigma_{XX}^{-1} ,$$
which implies
$$ W_H \cdots W_1 \Sigma_{XY} = U_{\Ss} U_{\Ss}^T \Sigma .$$
Let $i,j \in \llbracket1,H \rrbracket$ such that $i>j$.
The complementary blocks are $W_{j-1}\cdots W_1\Sigma_{XY}W_H \cdots W_{i+1}$ and $W_{i-1} \cdots W_{j+1}$.
\\
Using Lemma \ref{sigma invertible} and $U_{\Ss}^TU_{\Ss} = I_r$, we have, for the second complementary block,
\begin{align*}
   \rk(W_{i-1} \cdots W_{j+1}) & \geq \rk( W_H \cdots W_1 \Sigma_{XY})
   = \rk( U_{\Ss} U_{\Ss}^T \Sigma ) \geq \rk( U_{\Ss}^T(U_{\Ss} U_{\Ss}^T \Sigma) \Sigma^{-1}U_{\Ss} ) =\rk(I_r)  = r \;.
\end{align*}
For the first complementary block, using the same arguments, we have
\begin{align*}
    \rk(W_{j-1} \cdots W_1 \Sigma_{XY} W_H \cdots W_{i+1}) & \geq \rk(W_H \cdots W_1 \Sigma_{XY} W_H \cdots W_1 \Sigma_{XY}) \\
    &= \rk(U_{\Ss} U_{\Ss}^T \Sigma U_{\Ss} U_{\Ss}^T \Sigma) \\
    & \geq \rk\left( U_{\Ss}^T (U_{\Ss} U_{\Ss}^T \Sigma U_{\Ss} U_{\Ss}^T \Sigma)\Sigma^{-1} U_{\Ss} \right) \\
    &= \rk(U_{\Ss}^T \Sigma U_{\Ss}) \;.
\end{align*}
Recall that, from the diagonalization of $\Sigma$, we have $\Sigma U = U \Lambda$, hence, $\Sigma U_{\Ss} = U_{\Ss} \text{diag}((\lambda_s)_{s \in \Ss})$
\begin{align*}
    \rk(W_{j-1} \cdots W_1 \Sigma_{XY} W_H \cdots W_{i+1}) & \geq \rk( U_{\Ss}^T U_{\Ss} \text{diag}((\lambda_s)_{s \in \Ss})) \\
    &= \rk( \text{diag}((\lambda_s)_{s \in \Ss})) \\
    &= r \;.
\end{align*}
This concludes the proof.

\subsection{Proof of Proposition \ref{existence of tightened and non tightened critical points}}\label{proof of existence of tightened and non tightened critical points}

Let $H \geq 3$, $\Ss=\llbracket 1,r \rrbracket$ with $0 \leq r<r_{max}$.
We define $\Wbf$ as follows:
\begin{equation}\label{eqpourex}
\begin{cases}
    W_H = [U_{\Ss} , 0] \\
    W_h = \left[
        \begin{array}{c  c}
        I_r & 0 \\
        0 & Z_h
        \end{array}
        \right] \quad \text{for} \quad h \in \llbracket 2,H-1\rrbracket \\
    W_1 = \begin{bmatrix}
        U_{\Ss}^T\Sigma_{YX}\Sigma_{XX}^{-1} \\ 
        0
        \end{bmatrix}.
\end{cases}
\end{equation}
Using Proposition \ref{reciproque param}, $\Wbf$ is a first-order critical point associated with $\Ss$.
Let us show that depending on the choice of $(Z_h)_{h=2..H-1}$, $\Wbf$ can be tightened or non-tightened.\\
Since $H \geq 3$, there exists $h \in \llbracket 2,H-1 \rrbracket$.
If we choose $Z_{H-1},\ldots,Z_2$ such that $Z_{H-1} \cdots Z_2 \neq 0$ (e.g. when only the top left entry of each $Z_h$ is nonzero, which is possible since $r<r_{max} = \min(d_H,\ldots,d_0)$) then $\Wbf$ is non-tightened.
Indeed, the pivot $(H,1)$ is non-tightened because $\rk(\Sigma_{XY}) = d_y > r$ and $\rk(W_{H-1} \cdots W_2) = \rk\left(\left[
        \begin{array}{c  c}
        I_r & 0 \\
        0 & Z_{H-1} \cdots Z_2
        \end{array}
        \right]\right) > r$. \\
If we choose $Z_{H-1},\ldots,Z_2$ such that $Z_{H-1} \cdots Z_2 = 0$ (e.g. $Z_2=0$), then $\Wbf$ is tightened.
Indeed, the pivot $(H,1)$ is tightened because 
$W_{H-1} \cdots W_2 =
\left[
\begin{array}{c  c}
I_r & 0 \\
0 & 0
\end{array}
\right]$
is of rank $r$, and by construction we have $\rk(W_H)= \rk(W_1) = r$.
Hence, all the other pivots are tightened because at least one of their complementary blocks includes $W_H$ or $W_1$, and therefore, using Proposition \ref{all blocks ranks geq than r}, is of rank $r$. 
Therefore, $\Wbf$ is tightened.

\section{Parameterization of First-order Critical Points and Global Minimizers}\label{Appendix Parameterization of first-order critical points and global minimizers}

In this section, we prove Propositions \ref{simplif mat} and \ref{prop parameterization of global minimizers} that were stated in Section \ref{Parameterization of first-order critical points and global minimizers}.

\subsection{Proof of Proposition \ref{simplif mat}}\label{proof simplif mat}
Before proving Proposition \ref{simplif mat}, we introduce and prove two lemmas.

\begin{lemma}\label{simplif 1}
    Let $r$ be a nonnegative integer, and let $n$ and $p$ be two positive integers larger than or equal to $r$.
    Let $\Ss \subset \llbracket 1,d_y \rrbracket$ of size $r$ and let $Q = \llbracket 1,d_y \rrbracket \setminus \Ss$.
    Let $A \in \mathbb{R}^{d_y \times n}$ and $B \in \mathbb{R}^{n \times p}$ be two matrices such that 
    \begin{align*}
        AB = [U_{\Ss} , 0] \;.   
    \end{align*}
    Then, there exist an invertible matrix $D \in \mathbb{R}^{n \times n}$ and two matrices $N \in \mathbb{R}^{(d_y-r) \times (n-r)}$ and $B_{DR} \in \mathbb{R}^{(n-r) \times (p-r)}$ such that 
    \begin{align}
        AD &= [U_{\Ss} , U_Q N] \label{AD} \\
        D^{-1}B &= \left[
        \begin{array}{c  c}
        I_r & 0 \\
        0 & B_{DR}
        \end{array}
        \right]. \label{D-1B}
    \end{align}
\end{lemma}
In the proof below, we can easily see that the result still holds for $r=0$ and $r=\min(d_y,n,p)$ with the conventions adopted in Section \ref{Settings}.
\begin{proof}
Let $n$ and $p$ be non-negative integers such that $n,p \geq r$ and $A \in \mathbb{R}^{d_y \times n}$ and $B \in \mathbb{R}^{n \times p}$ such that 
\begin{align}
    AB = [U_{\Ss} , 0] . \label{AB lemme 1}
\end{align}

Recall that for any matrix $C$ with $n$ columns we write $C = [C_1,C_2, \ldots ,C_n]$ where $C_i$ represents the $i$-th column of $C$. \\
We have from \eqref{AB lemme 1},
\begin{align}\label{AB_r = U_s}
    A[B_1,B_2, \ldots ,B_r] = U_{\Ss} 
    \ .
\end{align}
Since the columns of $U$ are linearly independent, we have
$$\rk(A[B_1,B_2, \ldots ,B_r]) = \rk(U_{\Ss}) = r $$
and $\{B_1, \ldots ,B_r\}$ are necessarily also linearly independent.
Using the incomplete basis theorem, we complement $(B_1, \ldots ,B_r)$ to form a basis $(B_1, \ldots ,B_r,E_{r+1}, \ldots ,E_{n})$. We set $E=[B_1, \ldots ,B_r,E_{r+1}, \ldots ,E_{n}] \in \mathbb{R}^{n \times n}$.
By construction, the matrix $E$ is invertible.

We now set $A' = A E$ and $B' = E^{-1} B$. In particular $A'B' = A B$. \\
Also, note that $$ E \begin{bmatrix}
I_r \\ 0
\end{bmatrix} = [B_1, \ldots ,B_r] \;, $$
so that 
$$ E^{-1} [B_1, \ldots ,B_r] = \begin{bmatrix}
I_r \\ 0
\end{bmatrix} \;.$$
Therefore, we can write
\begin{align}
     B' = E^{-1} B = \left[
\begin{array}{c  c}
I_r & B_{UR} \\
0 & B_{DR}
\end{array}
\right]  \ , \label{B' lemme 1}
\end{align}
with $B_{UR} \in \mathbb{R}^{r \times (p-r)}$ and $B_{DR} \in \mathbb{R}^{(n-r) \times (p-r)}$ such that $$\begin{bmatrix}
B_{UR} \\ B_{DR}
\end{bmatrix} = E^{-1} [B_{r+1}, \ldots ,B_{p}] \;.$$
We define $L \in \mathbb{R}^{r \times (n - r)} $ and $N \in \mathbb{R}^{(d_y - r) \times (n - r)}$ by $\begin{bmatrix}
L \\ N
\end{bmatrix} = [U_{\Ss} , U_Q]^{-1}[AE_{r+1}, \ldots ,AE_n]$. 
We have
\begin{align}\label{L,N}
    [AE_{r+1}, \ldots ,AE_n] = [U_{\Ss} , U_Q] \begin{bmatrix}
    L \\ N
    \end{bmatrix} = U_{\Ss}L + U_Q N \;.
\end{align}
We also define the invertible matrix $F =  \left[
\begin{array}{c  c}
I_r & L \\
0 & I_{n-r}
\end{array}
\right] \in \mathbb{R}^{n \times n}$. Using \eqref{AB_r = U_s} and \eqref{L,N} we have
\begin{align*}
    A' &= A E \\
    &= A [B_1, \ldots ,B_r,E_{r+1}, \ldots ,E_{n}] \\
    &= [U_{\Ss} , U_{\Ss}L + U_Q N] \\
    &= [U_{\Ss} , U_Q N] \left[
    \begin{array}{c  c}
    I_r & L \\
    0 & I_{n-r}
    \end{array}
    \right] \\
    &= [U_{\Ss} , U_Q N] F \;.
\end{align*}
Therefore, defining the invertible matrix $D = E F^{-1} \in \mathbb{R}^{n \times n}$, we finally have
\begin{align}\label{17}
    AD = AE F^{-1} =  [U_{\Ss} , U_Q N] \;.
\end{align}
This proves \eqref{AD}. \\
We also have, using \eqref{B' lemme 1} and the definition of $F$
\begin{align}\label{18}
    D^{-1}B &= F E^{-1} B \nonumber \\
    &= FB' \nonumber \\
    &= \left[
    \begin{array}{c  c}
    I_r & L \\
    0 & I_{n-r}
    \end{array}
    \right] \left[
    \begin{array}{c  c}
    I_r & B_{UR} \\
    0 & B_{DR}
    \end{array}
    \right] \nonumber \\
    &= \left[
    \begin{array}{c  c}
    I_r & B_{UR} + L B_{DR} \\
    0 & B_{DR}
    \end{array}
    \right] \;.
\end{align}
However, noticing that, since \eqref{AB lemme 1} holds,
$$(AD)(D^{-1}B) = AB = [U_{\Ss} , 0] \ ,$$
and using \eqref{17} and \eqref{18} we obtain
\begin{align*}
    [U_{\Ss} , U_Q N] \left[
    \begin{array}{c  c}
    I_r & B_{UR} + L B_{DR} \\
    0 & B_{DR}
    \end{array}
    \right] &= [U_{\Ss} , 0] \ .
\end{align*}
Therefore $U_{\Ss}(B_{UR} + L B_{DR}) + U_Q N B_{DR} = 0$ .
Since $[U_{\Ss},U_{Q}]$ is invertible we get $B_{UR} + L B_{DR}=0$ and $N B_{DR}=0$ . \\
Finally, \eqref{18} becomes
$$D^{-1}B = \left[
\begin{array}{c  c}
I_r & 0 \\
0 & B_{DR}
\end{array}
\right] \;.$$
This proves \eqref{D-1B} and concludes the proof.
\end{proof}
The second lemma states that if the product of two factors takes the format of \eqref{D-1B}, then up to the product by an invertible matrix, the two factors have the same format.
In the proof of Proposition \ref{simplif mat}, we will use this property several times to establish \eqref{simplif W_h}.      
\begin{lemma}\label{simplif 2}
    Let $r$, $q$, $n$ and $p$ be positive integers such that $r \leq \min(q,n,p)$. Let $B \in \mathbb{R}^{q \times n}$, $C \in  \mathbb{R}^{n \times p}$ and $P \in \mathbb{R}^{(q-r) \times (p-r)}$ such that 
    \begin{align*}
        BC = \left[
        \begin{array}{c  c}
        I_r & 0 \\
        0 & P
        \end{array}
        \right] \ .
    \end{align*}
    Then, there exist an invertible matrix $D \in  \mathbb{R}^{n \times n}$ and two matrices $B_{DR} \in \mathbb{R}^{(q-r) \times (n-r)}$ and $C_{DR} \in \mathbb{R}^{(n-r) \times (p-r)}$ such that
    \begin{align}
        BD &= \left[
        \begin{array}{c  c}
        I_r & 0 \\
        0 & B_{DR}
        \end{array}
        \right] \label{BD} \\
        D^{-1}C &= \left[
        \begin{array}{c  c}
        I_r & 0 \\
        0 & C_{DR}
        \end{array}
        \right] . \label{D-1C}
    \end{align}
\end{lemma}
In the proof below, we can easily see that the result still holds for $r=0$ and $r=\min(q,n,p)$ with the conventions adopted in Section \ref{Settings}.
\begin{proof}
Let $r$, $q$, $n$ and $p$ be positive integers such that $r \leq \min(q,n,p)$.
Let $B \in \mathbb{R}^{q \times n}$, $C \in  \mathbb{R}^{n \times p}$ and $P \in \mathbb{R}^{(q-r) \times (p-r)}$ such that 
\begin{align}\label{BC}
    BC = \left[
    \begin{array}{c  c}
    I_r & 0 \\
    0 & P
    \end{array}
    \right] \ .
\end{align}
We have 
\begin{align}
    B[C_1,C_2, \ldots ,C_r] = \begin{bmatrix}
    I_r \\ 0
    \end{bmatrix}. \label{B C_r = I_r} 
\end{align}
Since the columns of $\begin{bmatrix}
I_r \\ 0
\end{bmatrix}$ are linearly independent, $$ \rk(B[C_1,C_2, \ldots ,C_r]) = r $$ and the vectors $C_1, \ldots ,C_r$ are necessarily also linearly independent. Using the incomplete basis theorem, we complement $(C_1, \ldots ,C_r)$ to form a basis $(C_1, \ldots ,C_r,E_{r+1}, \ldots ,E_{n})$. We denote $E=[C_1, \ldots ,C_r,E_{r+1}, \ldots ,E_{n}] \in \mathbb{R}^{n \times n}$.
By construction, the matrix $E$ is invertible. \\
We now set $B' = BE$ and $C' = E^{-1} C$. In particular 
\begin{align}\label{B'C'=BC}
    B'C' = BC.
\end{align}
Also notice that $$ E \begin{bmatrix}
I_r \\ 0
\end{bmatrix} = [C_1, \ldots ,C_r] \;, $$
so that 
$$ E^{-1} [C_1, \ldots ,C_r] = \begin{bmatrix}
I_r \\ 0
\end{bmatrix} \;. $$
Therefore, we can write
\begin{align}\label{C'}
     C' = E^{-1} C = \left[
\begin{array}{c  c}
I_r & C_{UR} \\
0 & C_{DR}
\end{array}
\right]  \ , 
\end{align}
where $C_{UR} \in \mathbb{R}^{r \times (p-r)}$ and $C_{DR} \in \mathbb{R}^{(n-r) \times (p-r)}$ are such that $\begin{bmatrix}
C_{UR} \\ C_{DR}
\end{bmatrix} = E^{-1} [C_{r+1}, \ldots ,C_{p}]$. \\
Now notice that, using \eqref{B C_r = I_r},
\begin{align*}
    B' &= BE \\
    &= B [C_1, \ldots ,C_r,E_{r+1}, \ldots ,E_{n}] \\
    &= \left[
    \begin{array}{c  c}
    I_r & B_{UR} \\
    0 & B_{DR}
    \end{array}
    \right] \ , \numberthis \label{B'}
\end{align*}
where $B_{UR} \in \mathbb{R}^{r \times (n-r)}$ and $B_{DR} \in \mathbb{R}^{(q-r) \times (n-r)}$ are such that $\begin{bmatrix}
B_{UR} \\ B_{DR}
\end{bmatrix} = B[E_{r+1}, \ldots ,E_n]$ \;. \\
Plugging \eqref{B'}, \eqref{C'} and \eqref{BC} in the equality \eqref{B'C'=BC}, we obtain
$$\left[
\begin{array}{c  c}
I_r & B_{UR} \\
0 & B_{DR}
\end{array}
\right]  \left[
\begin{array}{c  c}
I_r & C_{UR} \\
0 & C_{DR}
\end{array}
\right] = \left[
\begin{array}{c  c}
I_r & 0 \\
0 & P
\end{array}
\right]\;,$$
which yields
$$\left[
\begin{array}{c  c}
I_r & C_{UR} + B_{UR} C_{DR} \\
0 & B_{DR} C_{DR}
\end{array}
\right] =\left[
\begin{array}{c  c}
I_r & 0 \\
0 & P
\end{array}
\right]\;.$$
Therefore, $C_{UR} + B_{UR} C_{DR} = 0$ or, equivalently ,
\begin{align}\label{C_UR}
    C_{UR} = -B_{UR} C_{DR} \;.
\end{align}
Define $F = \left[
\begin{array}{c  c}
I_r & -B_{UR} \\
0 & I_{n-r}
\end{array}
\right]$. The matrix $F$ is invertible. Moreover, using \eqref{C'} and \eqref{C_UR} we have
\begin{align*}
    C' &= \left[
    \begin{array}{c  c}
    I_r & -B_{UR} C_{DR} \\
    0 & C_{DR}
    \end{array}
    \right] \\
    &= \left[
    \begin{array}{c  c}
    I_r & -B_{UR} \\
    0 & I_{n-r}
    \end{array}
    \right]  \left[
    \begin{array}{c  c}
    I_r & 0 \\
    0 & C_{DR}
    \end{array}
    \right] \\
    &= F \left[
    \begin{array}{c  c}
    I_r & 0 \\
    0 & C_{DR}
    \end{array}
    \right] \;.
\end{align*}
Therefore, if we define $D = EF$, $D$ is invertible and
$$D^{-1}C = F^{-1} E^{-1} C = F^{-1} C' = \left[
\begin{array}{c  c}
I_r & 0 \\
0 & C_{DR}
\end{array}
\right].$$
This proves \eqref{D-1C}. \\
In order to prove \eqref{BD}, we remark that, using \eqref{B'} and the definition of $F$, we also have
\begin{align*}
    BD &= BEF \\
    &= B'F \\
    &= \left[
    \begin{array}{c  c}
    I_r & B_{UR} \\
    0 & B_{DR}
    \end{array}
    \right]  \left[
    \begin{array}{c  c}
    I_r & -B_{UR} \\
    0 & I_{n-r}
    \end{array}
    \right] \\
    &= \left[
    \begin{array}{c  c}
    I_r & 0 \\
    0 & B_{DR}
    \end{array}
    \right] \;.
\end{align*}
This proves \eqref{BD} and concludes the proof.
\end{proof}

Now we prove Proposition \ref{simplif mat}.
\begin{proof}[Proof of Proposition \ref{simplif mat}]

Let $\Wbf = (W_H, \ldots ,W_1)$ be a first-order critical point of $L$. Then using Lemma \ref{clem} there exist $D \in \mathbb{R}^{d_{1} \times d_{1}} $ invertible and a matrix $M \in \mathbb{R}^{d_{1} \times d_{x}}$ which satisfies $W_H \cdots W_2M = 0$ such that 
\begin{align}
    W_H \cdots W_2 &= [U_{\Ss} \ , \ 0 ] D \label{W_H...W_2} \\
    W_1 &= D^{-1} 
    \begin{bmatrix}
    U_{\Ss}^T\Sigma_{YX}\Sigma_{XX}^{-1} \\ 
    0 
    \end{bmatrix}
    + M \;. \label{W_1}
\end{align}
Denoting $D_1 = D^{-1}$ and using \eqref{W_H...W_2}, we have $W_H \cdots W_2 D_1= [U_{\Ss} , 0]$. Then applying Lemma \ref{simplif 1} with $A=W_H$ and $B=W_{H-1} \cdots W_2 D_1$ , there exist an invertible matrix $D_{H-1} \in \mathbb{R}^{d_{H-1} \times d_{H-1}}$, and matrices $Z_H \in \mathbb{R}^{(d_y - r) \times (d_{H-1} - r)}$ and $B_{DR} \in \mathbb{R}^{(d_{H-1} - r) \times (d_1 - r)}$ such that 
\begin{align}
    \widetilde{W}_H := W_H D_{H-1} &= [U_{\Ss} , U_Q Z_H] \nonumber \\
    D_{H-1}^{-1} W_{H-1} \cdots W_2 D_1 &= \left[
    \begin{array}{c  c}
    I_r & 0 \\
    0 & B_{DR}
    \end{array}
    \right] \;. \label{tema}
\end{align}
The first equality proves \eqref{simplif W_H}. \\
Then applying Lemma \ref{simplif 2} to \eqref{tema} with  $B = D_{H-1}^{-1} W_{H-1}$ and $C = W_{H-2} \cdots W_2 D_1$ we get the existence of an invertible matrix $D_{H-2} \in \mathbb{R}^{d_{H-2} \times d_{H-2}}$, $C_{DR} \in \mathbb{R}^{(d_{H-2}-r) \times (d_1-r)}$ and $Z_{H-1} \in \mathbb{R}^{(d_{H-1}-r) \times (d_{H-2}-r)}$ such that
$$ \widetilde{W}_{H-1} := D_{H-1}^{-1} W_{H-1} D_{H-2} = \left[
\begin{array}{c  c}
I_r & 0 \\
0 & Z_{H-1}
\end{array}
\right]\;,$$ and
$$D_{H-2}^{-1} W_{H-2} \cdots W_2 D_1 = \left[
\begin{array}{c  c}
I_r & 0 \\
0 & C_{DR}
\end{array}
\right] \;.$$
Reiterating the process by using Lemma \ref{simplif 2} multiple times with $B = D_{h}^{-1} W_h$ and $C = W_{h-1} \cdots W_2 D_1$ for $h$ decreasing from $H-2$ to 3, we can conclude that there exist invertible matrices $D_h \in \mathbb{R}^{d_{h} \times d_{h}}$ and matrices $Z_h \in \mathbb{R}^{(d_h-r) \times (d_{h-1}-r)}$, for $h \in \llbracket 2,H-1 \rrbracket$, such that $$ \widetilde{W}_h := D_h^{-1}W_h D_{h-1} = \left[
\begin{array}{c  c}
I_r & 0 \\
0 & Z_h
\end{array}
\right] \quad \forall h \in \llbracket 2,H-1 \rrbracket \;. $$ \\
This entails \eqref{simplif W_h}. \\
We also have from \eqref{W_1} that
$W_1 =D_1 
\begin{bmatrix}
U_{\Ss}^T\Sigma_{YX}\Sigma_{XX}^{-1} \\ 
0 
\end{bmatrix}
+ M$  with $W_H \cdots W_2 M = 0$.
Therefore,
$$D_1^{-1}W_1 = 
\begin{bmatrix}
U_{\Ss}^T\Sigma_{YX}\Sigma_{XX}^{-1} \\ 
0 
\end{bmatrix}
+ D_1^{-1}M \;.$$ 
Using \eqref{W_H...W_2}, $D_1 = D^{-1}$ and $W_H \cdots W_2 M = 0$, we obtain
$$[U_{\Ss} \ , \ 0] D_1^{-1} M = 0 \;.$$
Writing $D_1^{-1}M = \begin{bmatrix}
Z_0 \\ Z_1
\end{bmatrix}$, where $Z_0 \in \mathbb{R}^{r \times d_x}$ and $Z_1 \in \mathbb{R}^{(d_1 - r) \times d_x}$, we have
\begin{align*}
    0 &= [U_{\Ss} \ , \ 0] D_1^{-1}M \\
    &= [U_{\Ss} \ , \  0]\begin{bmatrix}
    Z_0 \\ Z_1
     \end{bmatrix} \\
     &= U_{\Ss}Z_0 \;.
\end{align*}
Multiplying on the left by $U_{\Ss}^T$ we obtain 
$$ Z_0 = 0.$$
Therefore $D_1^{-1} M = \begin{bmatrix}
0 \\ Z_1
\end{bmatrix}$, which yields
$$ \widetilde{W}_1 := D_1^{-1}W_1 = \begin{bmatrix}
U_{\Ss}^T\Sigma_{YX}\Sigma_{XX}^{-1} \\ 
Z_1  
\end{bmatrix} \;.$$
This proves \eqref{simplif W_1}. \\
Finally we have 
\begin{align*}
    \widetilde{W}_H \cdots \widetilde{W}_2 &= (W_H D_{H-1}) (D_{H-1}^{-1} W_{H-1} D_{H-2}) \cdots (D_2^{-1} W_2 D_1) \\
    &= W_H \cdots W_2 D_1 \\
    &= [ U_{\Ss} \ , \ 0] \ ,
\end{align*}
where the last equality is due to \eqref{W_H...W_2} and $D_1 = D^{-1}$.
This entails \eqref{simplif W_H...W_2} and concludes the proof.
\end{proof}

\subsection{Proof of Proposition \ref{prop parameterization of global minimizers}}\label{proof of prop parameterization of global minimizers}

We first make a comment about notational subtleties to help understand the statement of Proposition~\ref{prop parameterization of global minimizers}, and then prove the proposition.

Recall that $r_{max} = \min(d_H, \ldots, d_0)$, and $d_x=d_0 \geq d_y=d_H$ by assumption.
Therefore, in the statement of Proposition \ref{prop parameterization of global minimizers}, some blocks $Z_h$ have 0 lines or 0 columns, and thus do not exist.
For example, depending on the value of $r_{max}$, we have
\begin{align*}
    \begin{cases}
        W_H = U_{\Ss_{max}} D_{H-1}^{-1}  &\text{if} \ r_{max} = d_{H-1} \\
        W_1 = D_1 U_{\Ss_{max}}^T\Sigma_{YX}\Sigma_{XX}^{-1}  &\text{if} \ r_{max} = d_{1}
    \end{cases}
\end{align*}
and for $h \in \llbracket 2,H-1 \rrbracket$
\begin{align*}
    W_h =
    \begin{cases}
        D_{h}\left[
        \begin{array}{c  c}
        I_{r_{max}} & 0 
        \end{array}
        \right]D_{h-1}^{-1} \ &\text{if} \  r_{max} = d_h < d_{h-1}  \\ 
        D_{h}\left[
        \begin{array}{c  c}
        I_{r_{max}} \\
        0 
        \end{array}
        \right]D_{h-1}^{-1} \ &\text{if} \ r_{max} = d_{h-1} < d_h \\
        D_{h}I_{r_{max}}D_{h-1}^{-1} \ &\text{if} \ r_{max} = d_h = d_{h-1}
    \end{cases}
\end{align*}
Also, if $r_{max}=d_y$, then $Q_{max} = \emptyset$, hence $U_{Q_{max}} \in \mathbb{R}^{d_y \times 0} $ and $Z_H \in \mathbb{R}^{0 \times (d_{H-1}-r_{max})}$. Then, using the convention in Section \ref{Settings}, $U_{Q_{max}} Z_H = 0_{d_y \times (d_{H-1}-r_{max})}$, so that $ {W}_H = [U_{\Ss_{max}} , 0_{d_y \times (d_{H-1}-r_{max})}] D_{H-1}^{-1} \in \mathbb{R}^{d_y \times d_{H-1}}$.

We are now ready to prove the proposition.
\begin{proof}[Proof of Proposition \ref{prop parameterization of global minimizers}]

Let $\Ss_{max} = \llbracket 1,r_{max} \rrbracket$.
Let us first prove that $\Wbf$ is a global minimizer of $L$ if and only if $\Wbf$ is a first-order critical point of $L$ associated with $\Ss_{max}$.
From Lemma \ref{lemma yun}, we have 
\begin{align*}
     U_{\Ss_{max}} U_{\Ss_{max}}^T \Sigma_{YX} \Sigma_{XX}^{-1} \in \argmin_{\substack{R \in \mathbb{R}^{d_y \times d_x} \\ \rk(R) \leq r_{max}}} \|RX - Y \|^2 \;.
\end{align*}
Let $\Wbf$ be a first-order critical point associated with $\Ss_{max}$ (note that from Proposition \ref{reciproque S vers W}, such $\Wbf$ exist).
We have $W_H \cdots W_1 = U_{\Ss_{max}} U_{\Ss_{max}}^T \Sigma_{YX} \Sigma_{XX}^{-1}$, hence, for all $\Wbf' = (W'_H,\ldots,W'_1)$, since $\rk(W'_H \cdots W'_1) \leq r_{max}$, we have
$$ L(\Wbf') \geq \min_{\substack{R \in \mathbb{R}^{d_y \times d_x} \\ \rk(R) \leq r_{max}}} \|RX - Y \|^2 = \|W_H \cdots W_1 X - Y \|^2 = L(\Wbf) \;.$$
As a consequence, $\Wbf$ is a global minimizer of $L$. \\
Conversely, if $\Wbf$ is a global minimizer of $L$, then $\Wbf$ is a first-order critical point of $L$.
From Proposition \ref{global map and critical values}, there exist $\Ss \subset \llbracket 1,d_y \rrbracket$ of size $r \in \llbracket0,r_{max} \rrbracket$ such that $W_H \cdots W_1 = U_{\Ss} U_{\Ss}^T \Sigma_{YX} \Sigma_{XX}^{-1}$, and we have $L(\Wbf) = \tr(\Sigma_{YY}) - \sum_{i \in \Ss} \lambda_i$.
But we have from Assumption \ref{Assump H}, $\lambda_1 > \ldots > \lambda_{d_y}$, and, since $\Sigma$ is invertible (see Lemma \ref{sigma invertible}), then $\lambda_{d_y}>0$.
Therefore, using Proposition \ref{reciproque S vers W}, $\Wbf$ is a global minimizer of $L$ implies that $\Ss = \llbracket 1,r_{max} \rrbracket = \Ss_{max}$.
Hence, $\Wbf$ is a global minimizer of $L$ if and only if $\Wbf$ is a first-order critical point of $L$ associated with $\Ss_{max}$.\\
Let us now prove Proposition \ref{prop parameterization of global minimizers}.\\
Let $\Wbf = (W_H, \ldots , W_1)$ be a first-order critical point associated with $\Ss_{max} = \llbracket 1,r_{max} \rrbracket$.
Using Proposition \ref{simplif mat}, there exist invertible matrices $D_{H-1} \in \mathbb{R}^{d_{H-1} \times d_{H-1}}, \ldots ,D_1 \in \mathbb{R}^{d_{1} \times d_{1}}$, and matrices $Z_H \in \mathbb{R}^{(d_y-r_{max}) \times (d_{H-1}-r_{max})}$, $Z_h \in \mathbb{R}^{(d_h-r_{max}) \times (d_{h-1}-r_{max})}$ for $h \in \llbracket 2 , H-1 \rrbracket$, and $Z_1 \in \mathbb{R}^{(d_1-r_{max}) \times d_x}$ such that:
    \begin{align*}
        {W}_H &= [U_{\Ss_{max}} , U_{Q_{max}} Z_H] D_{H-1}^{-1} \\
        {W}_1 &= D_1 \begin{bmatrix}
        U_{\Ss_{max}}^T\Sigma_{YX}\Sigma_{XX}^{-1} \\ 
        Z_1
        \end{bmatrix} \\
        {W}_h &= D_h \left[
        \begin{array}{c  c}
        I_{r_{max}} & 0 \\
        0 & Z_h
        \end{array}
        \right] D_{h-1}^{-1} \qquad \forall h \in \llbracket 2 , H-1 \rrbracket \;.
    \end{align*}
    
Conversely, consider matrices $D_h$, for $h \in \llbracket1,H-1 \rrbracket$ and $Z_h$, for $h \in \llbracket1,H \rrbracket$ as in Proposition \ref{prop parameterization of global minimizers}, and
\begin{align*}
        {W}_H &= [U_{\Ss_{max}} , U_{Q_{max}} Z_H] D_{H-1}^{-1} \\
        {W}_1 &= D_1 \begin{bmatrix}
        U_{\Ss_{max}}^T\Sigma_{YX}\Sigma_{XX}^{-1} \\ 
        Z_1
        \end{bmatrix} \\
        {W}_h &= D_h \left[
        \begin{array}{c  c}
        I_{r_{max}} & 0 \\
        0 & Z_h
        \end{array}
        \right] D_{h-1}^{-1} \qquad \forall h \in \llbracket 2 , H-1 \rrbracket \;.
    \end{align*}
Since $|\Ss_{max}| = r_{max}$, using Proposition \ref{reciproque param}, we have that $\Wbf$ is a first-order critical point associated with $\Ss_{max}$.
This concludes the proof.
\end{proof}

\section{Global Minimizers and Simple Strict Saddle Points (Proof of Proposition \ref{main prop partie baldi})} \label{proof main prop partie baldi}

Recall that $r_{max} = \min(d_H,\ldots,d_0)$.

Let $\Wbf = (W_H, \ldots , W_1)$ be a first-order critical point of $L$ associated with $\Ss$ of size $r= \rk(W_H \cdots W_1) \leq r_{max}$. \\

\textbf{Case 1:} $\Ss = \llbracket 1,r_{max} \rrbracket = \Ss_{max}$.
In this case, using Lemma \ref{lemma yun},
\begin{align*}
    W_H \cdots W_1 = U_{\Ss_{max}} U_{\Ss_{max}}^T \Sigma_{YX} \Sigma_{XX}^{-1} \in \argmin_{\substack{R \in \mathbb{R}^{d_y \times d_x} \\ \rk(R) \leq r_{max}}} \|RX - Y \|^2 \;.
\end{align*}
Moreover, for all $\Wbf' = (W'_H,\ldots,W'_1)$, since $\rk(W'_H \cdots W'_1) \leq r_{max}$, we have
$$ L(\Wbf') \geq \min_{\substack{R \in \mathbb{R}^{d_y \times d_x} \\ \rk(R) \leq r_{max}}} \|RX - Y \|^2 = \|W_H \cdots W_1 X - Y \|^2 = L(\Wbf) \;.$$
As a consequence, $\Wbf$ is a global minimizer of $L$.

\textbf{Case 2:}
In order to prove the two remaining statements, we assume that $\Ss \neq \llbracket 1,r \rrbracket$ with $0 < r \leq r_{max}$, and show that $\Wbf$ is not a second-order critical point. \\
To do this we will find $\Wbf' = (W'_H, \ldots , W'_1)$ such that $c_2(\Wbf,\Wbf')<0$ (see Lemma \ref{Characterization of critical points in our settings} ).
More precisely, we find a linear trajectory of the form $W_h(t) = W_h + t W'_h$ such that the second-order coefficient of the asymptotic expansion of $L(\left(W_{h}(t)\right)_{h=1 . . H})$ around $t=0$ is negative.
This proves that $\Wbf$ is not a second-order critical point.\\

Since $\Ss \neq \llbracket 1,r \rrbracket$, and the eigenvalues ${(\lambda_k)}_{k \in \llbracket 1,d_y \rrbracket}$ are distinct and in decreasing order (see Section \ref{Settings}), there exist $j \in \Ss$ and $ i \not \in \Ss $ such that
\begin{align}\label{lambda i > lambda j}
    \lambda_i > \lambda_j \;.
\end{align}
We denote by $\mathcal{S}=\left\{i_{1}, \dots, i_{r}\right\}$, hence there exists $g \in \llbracket 1,r \rrbracket$ such that $j = i_g $. \\
Note that,
\begin{align*}
    U_{\Ss} = U \sum_{k=1}^r E_{i_k,k}
\end{align*}
where $E_{l,k} \in \mathbb{R}^{d_y \times r}$ is the matrix whose entries are all 0 except the one in position $(l,k)$ which is equal to 1. \\
Denote by $U_t$ the matrix formed by replacing in $U_{\Ss}$ the column corresponding to $u_j$ by $u_j + t u_i$. 
More precisely, set
\begin{align*}
    U_t = U_{\Ss} + t U E_{i,g} \;.
\end{align*}
Set $V = U E_{i,g} \in \mathbb{R}^{d_y \times r}$ and
\begin{align}\label{expression V_t}
    V_t = \sum_{k=1}^r E_{i_k,k} + t E_{i,g} \in \mathbb{R}^{d_y \times r}.
\end{align}
Hence we have
\begin{align} \label{U_t = U V_t}
    U_t = U_{\Ss} + t V = U V_t \;.
\end{align}
Considering $D \in \mathbb{R}^{d_1 \times d_1}$ as provided by Lemma \ref{clem},
we set
$$
\begin{cases}
    W'_1 = D^{-1} \left[
    \begin{array}{c}
    V^T\Sigma_{YX}\Sigma_{XX}^{-1} \\
    0_{(d_1-r) \times d_x}
    \end{array}
    \right] \\
    W'_h=0 \qquad \forall h \in \llbracket 2,H-1 \rrbracket \\
    W'_H = VU_{\Ss}^T W_H \;. \\
\end{cases}
$$
and for all $ h \in \llbracket 1,H \rrbracket $, $W_h(t) = W_h + t W'_h$.
Note that
$$W_H(t) = W_H + t W'_H = (I_{d_y} + tVU_{\Ss}^T) W_H \;,$$
and therefore
$$ K(t) := W_H(t) \cdots W_2(t) = (I_{d_y} + tVU_{\Ss}^T)K \;,$$
where $K = W_H \cdots W_2$.
Using Lemma \ref{clem}, there exists $M \in \mathbb{R}^{d_1 \times d_x}$ satisfying $KM=0$ such that 
$$W_1 = D^{-1} \left[
\begin{array}{c}
U_{\Ss}^T\Sigma_{YX}\Sigma_{XX}^{-1} \\
0_{(d_1-r) \times d_x}
\end{array}
\right] + M \;.$$
Hence,
$$ W_1(t) = D^{-1} \left[
\begin{array}{c}
U_{\Ss}^T\Sigma_{YX}\Sigma_{XX}^{-1} \\
0_{(d_1-r) \times d_x}
\end{array}
\right] + M + tW'_1
=
D^{-1} \left[
\begin{array}{c}
(U_{\Ss}^T +tV^T)\Sigma_{YX}\Sigma_{XX}^{-1} \\
0_{(d_1-r) \times d_x}
\end{array}
\right] + M \;,$$
where $M \in \mathbb{R}^{d_1 \times d_x}$ is such that $KM=0$.
Therefore
\begin{align*}
W_t: &= W_H(t) \cdots W_1(t) \\
& = K(t) W_1(t) \\
&= (I_{d_y} + tVU_{\Ss}^T) \left( K D^{-1} \left[
\begin{array}{c}
(U_{\Ss}^T +tV^T)\Sigma_{YX}\Sigma_{XX}^{-1} \\
0_{(d_1-r) \times d_x}
\end{array}
\right] + KM\right) \;.
\end{align*}
From Lemma \ref{clem}, using that $KM=0$ and $K=[U_{\Ss} \quad 0_{d_y \times (d_1-r)} ] D$, this becomes
\begin{align*}
W_t & = (I_{d_y} + t V U_{\Ss}^T) [U_{\Ss} \quad 0_{d_y \times (d_1-r)} ] D D^{-1} \left[
\begin{array}{c}
(U_{\Ss}^T +t V^T)\Sigma_{YX}\Sigma_{XX}^{-1} \\
0_{(d_1-r) \times d_x}
\end{array}
\right]\\
& = (I_{d_y} + t V U_{\Ss}^T) U_{\Ss}(U_{\Ss}^T + t V^T)\Sigma_{YX}\Sigma_{XX}^{-1} \;. \\
\end{align*}
Using that $U_{\Ss}^T U_{\Ss} = I_r$ (see Lemma \ref{properties of orhtogonality}), we obtain
\begin{align*}
W_t  = (U_{\Ss} + t V)(U_{\Ss}^T + t V^T)\Sigma_{YX}\Sigma_{XX}^{-1}  = U_t U_t^T \Sigma_{YX}\Sigma_{XX}^{-1} \;. \numberthis \label{W_t baldi}
\end{align*}
Recall that our goal is to show that the asymptotic expansion of $\eqref{L(W_t)}$ around $t=0$ has a negative second-order coefficient.
We calculate
\begin{align}
    L(\left(W_{h}(t)\right)_{h=1 . . H}) &= \|W_t X - Y\|^2 \nonumber \\
    &= \tr\left(W_t \Sigma_{XX} W_t^T\right) - 2 \tr(W_t\Sigma_{XY}) + \tr(\Sigma_{YY}) \;. \label{L(W_t)}
\end{align}
Let us simplify $\tr(W_t \Sigma_{XX} W_t^T)$ first. Using \eqref{W_t baldi}, we have 
\begin{align*}
W_t\Sigma_{XX}W_t^T = U_t U_t^T \Sigma_{YX}\Sigma_{XX}^{-1} \Sigma_{XX} \Sigma_{XX}^{-1} \Sigma_{XY} U_t U_t^T = U_t U_t^T \Sigma U_t U_t^T \;.  
\end{align*}
Using \eqref{U_t = U V_t}, $U^TU = I_{d_y}$, $\Sigma = U \Lambda U^T$ and the cyclic property of the trace, we obtain
\begin{align*}
\tr\left(W_t\Sigma_{XX}W_t^T\right) = \tr\left(UV_t V_t^T U^T U \Lambda U^T U V_t V_t^T U^T\right) = \tr\left(V_t V_t^T \Lambda V_t V_t^T\right) = \tr\left(\left(V_t V_t^T\right)^2 \Lambda\right) \;.
\end{align*}
We define $(\overline{E}_{k,l})_{k=1..d_y,l=1..d_y}$ the canonical basis of $\mathbb{R}^{d_y \times d_y}$.
More precisely, $\overline{E}_{k,l} \in \mathbb{R}^{d_y \times d_y}$ has all its entries equal to 0, except a $1$ at position $(k,l)$.
Note that for all $a,c \in \llbracket 1,d_y \rrbracket$ and $b,d \in \llbracket 1,r \rrbracket$
\begin{align*}
    E_{a,b} E_{c,d}^T = \delta_{b,d} \overline{E}_{a,c} \;,
\end{align*}
where $\delta_{b,d}$ equals $1$ if $b=d$ and $0$ otherwise.
Using the definition of $V_t$ in \eqref{expression V_t} and $j=i_g$, for $g \in \llbracket 1,r \rrbracket$, we have
\begin{align*}
V_t V_t^T &= \left(\sum_{k=1}^r E_{i_k,k} + t E_{i,g}\right)\left(\sum_{k'=1}^r E^T_{i_{k'},k'} + t E^T_{i,g}\right) \\
&= \left(\sum_{k=1}^r \overline{E}_{i_k,i_k}\right) + t \overline{E}_{i_g,i} + t \overline{E}_{i,i_g} + t^2 \overline{E}_{i,i} \\
&= \left(\sum_{k \in \Ss} \overline{E}_{k,k}\right) + t \overline{E}_{j,i} + t \overline{E}_{i,j} + t^2 \overline{E}_{i,i} \;. \numberthis \label{V_t V_t^T}
\end{align*}
We also have for all $a,b,c,d \in \llbracket 1,d_y \rrbracket$ 
\begin{align*}
    \overline{E}_{a,b} \overline{E}_{c,d} = \delta_{b,c} \overline{E}_{a,d} \;.
\end{align*}
Recalling that $j \in \Ss$ and $i \notin \Ss$, we obtain
\begin{align*}
(V_t V_t^T)^2 &= \left(\left(\sum_{k \in \Ss} \overline{E}_{k,k}\right) + t \overline{E}_{j,i} + t \overline{E}_{i,j} + t^2 \overline{E}_{i,i}\right) \left(\left(\sum_{k' \in \Ss} \overline{E}_{k',k'}\right) + t \overline{E}_{j,i} + t \overline{E}_{i,j} + t^2 \overline{E}_{i,i}\right) \\
&= \left(\left(\sum_{k \in \Ss} \overline{E}_{k,k}\right) + t \overline{E}_{j,i} + 0 + 0\right) + (0 + 0 + t^2 \overline{E}_{j,j} + t^3 \overline{E}_{j,i}) \\
& \quad + (t \overline{E}_{i,j} + t^2 \overline{E}_{i,i} + 0 + 0) + (0 + 0 + t^3 \overline{E}_{i,j} + t^4 \overline{E}_{i,i}) \\
&= \left(\sum_{k \in \Ss} \overline{E}_{k,k}\right) + t^2(1+t^2) \overline{E}_{i,i} + t^2 \overline{E}_{j,j} + t(1+t^2) \overline{E}_{i,j} + t(1+t^2) \overline{E}_{j,i} \;.
\end{align*}
Finally, since for all $a,b \in \llbracket 1,d_y \rrbracket$
\begin{align}\label{E a b Lambda}
    \overline{E}_{a,b} \Lambda = \lambda_b \overline{E}_{a,b}
\end{align}
we have 
\begin{align}
    \tr\left(W_t \Sigma_{XX} W_t^T\right) = \tr\left(\left(V_t V_t^T\right)^2 \Lambda \right) = \sum_{k \in \Ss} \lambda_k + t^2(1+t^2)\lambda_i + t^2 \lambda_j \;. \label{tr(W_t sigma xx W_t^T}
\end{align}
Coming back to \eqref{L(W_t)}, we calculate the other term $\tr(W_t\Sigma_{XY})$.
Using \eqref{W_t baldi}, \eqref{U_t = U V_t} and $\Sigma = U \Lambda U^T $, we obtain
\begin{align*}
\tr(W_t \Sigma_{XY}) = \tr(U_t U_t^T \Sigma) = \tr(UV_tV_t^TU^TU\Lambda U^T) = \tr(V_t V_t^T \Lambda) \;.
\end{align*}
Combining with \eqref{V_t V_t^T} and \eqref{E a b Lambda}, we get 
\begin{align*}
    \tr(W_t \Sigma_{XY}) = \tr(V_t V_t^T \Lambda) 
    = \sum_{k \in \Ss} \lambda_k + t^2 \lambda_i \;. \numberthis \label{tr(W_t sigma x y)}
\end{align*}
Finally, substituting \eqref{tr(W_t sigma xx W_t^T} and \eqref{tr(W_t sigma x y)} in \eqref{L(W_t)}, we have
\begin{align*}
L(\left(W_{h}(t)\right)_{h=1 . . H}) &= \tr(\Sigma_{YY}) +  \sum_{k \in \Ss} \lambda_k + t^2(1+t^2) \lambda_i + t^2 \lambda_j -2 \sum_{k \in \Ss} \lambda_k -2 t^2 \lambda_i \\
&= \tr(\Sigma_{YY}) -  \sum_{k \in \Ss} \lambda_k + t^2(\lambda_j - \lambda_i) + \lambda_i t^4 \;.
\end{align*}
Using Proposition \ref{global map and critical values} and recalling \eqref{lambda i > lambda j}, we finally get as $t \to 0$,  $$L(\left(W_{h}(t)\right)_{h=1 . . H}) = L(\Wbf) + c t^2 + o(t^2) \quad \textrm{with} \quad c = \lambda_j - \lambda_i<0 \;.$$
Therefore, we conclude from Lemma \ref{Characterization of critical points in our settings} that $\Wbf = (W_H, \ldots ,W_1)$ is not a second-order critical point.

\section{Strict Saddle Points with \texorpdfstring{$\Ss = \llbracket 1,r \rrbracket, \ r<r_{max}$}{} (Proof of Proposition \ref{main prop subtle strict saddles})} \label{Appendix subtle strict saddles}

We refer the reader to Section \ref{subtle strict saddles}, which introduces the 4 cases proved below.
Recall that $\Ss=\llbracket 1,r \rrbracket$ and we set $Q = \llbracket 1,d_y \rrbracket \setminus \Ss = \llbracket r+1,d_y \rrbracket$.\\
In this section, for each vector space $\mathbb{R}^{d_h}$, we will denote by $e_m$ the $m$-th element of the canonical basis of $\mathbb{R}^{d_h}$.
That is, the entries of $e_m \in \mathbb{R}^{d_h}$ are all equal to $0$ except for the $m$-th coordinate which is equal to $1$. The size of $e_m$ will not be ambiguous, once in context, so we do not include it in the notation.\\
\textbf{Remark about $r=0$:}
Using the conventions of Section \ref{Settings}, in this case we have $\Ss = \emptyset$ and $Q = \llbracket 1,d_y \rrbracket$.
Hence $U_{\Ss}$ is the matrix with no column, $U_Q = U$,  and $U_{\Ss} U_{\Ss}^T = 0_{d_y \times d_y}$.  For example, we still have $I_{d_y} = U_{\Ss} U_{\Ss}^T + U_{Q} U_{Q}^T$.
We can easily follow the proofs below with these conventions and see that the result still holds.

\subsection{1st Case:  \texorpdfstring{$i \in \llbracket 2 , H-1 \rrbracket$}{} and \texorpdfstring{$j=1$}{}} \label{1st case}
In this case, the two complementary blocks are $\Sigma_{XY} W_H \cdots W_{i+1}$ and $W_{i-1} \cdots W_2$.
Recall that $\Ss = \llbracket 1,r \rrbracket$ and $r<r_{max} = \min(d_H, \ldots,d_0)$.
Note that $\rk(\Sigma_{XY}W_H \cdots W_{i+1}) = \rk(W_H \cdots W_{i+1})$ because $\Sigma_{XY}$ is of full column rank (see Assumption \ref{Assump H}, in Section \ref{Settings}) .\\
Since the pivot $(i,j)$ is not tightened, using Proposition \ref{all blocks ranks geq than r}, we have 
\begin{align}\label{(i,1) not tightened}
    \begin{cases}
     \rk(W_H \cdots W_{i+1}) >r \\ \rk(W_{i-1} \cdots W_2)>r.
    \end{cases}
\end{align}
Let us first show that there exists $k \in \llbracket r+1, d_y \rrbracket$ and $l \in \llbracket 1,d_i \rrbracket$ such that 
\begin{align} \label{ U_k^T(W_H cdots W_i+1)_.,l neq 0}
    U_k^T(W_H \cdots W_{i+1})_{.,l} \neq 0 \;.
\end{align}
Indeed, assume by contradiction that for all $ k \in \llbracket r+1, d_y \rrbracket$ and $ l \in \llbracket 1,d_i \rrbracket$ we have 
$$U_k^T(W_H \cdots W_{i+1})_{.,l} = 0.$$
Recalling that $Q =  \llbracket 1, d_y \rrbracket \setminus \Ss = \llbracket r+1, d_y \rrbracket$, we obtain $U_Q^T W_H \cdots W_{i+1} = 0$.
Using from Lemma \ref{properties of orhtogonality} that $I_{d_y} = U_{\Ss} U_{\Ss}^T + U_Q U_Q^T$, we have
\begin{align*}
    W_H \cdots W_{i+1} &= (U_{\Ss} U_{\Ss}^T + U_Q U_Q^T) W_H \cdots W_{i+1} \\
    &= U_{\Ss} U_{\Ss}^T W_H \cdots W_{i+1}.
\end{align*}
Therefore,
$$ \rk(W_H \cdots W_{i+1}) = \rk(U_{\Ss} U_{\Ss}^T W_H \cdots W_{i+1}). $$
The latter is impossible since $\rk(U_{\Ss} U_{\Ss}^T W_H \cdots W_{i+1}) \leq |\Ss| = r $, which is not compatible with \eqref{(i,1) not tightened}.
Therefore \eqref{ U_k^T(W_H cdots W_i+1)_.,l neq 0} holds.\\
Since $\Wbf$ is a first-order critical point, using Lemma \ref{clem}, there exists an invertible matrix $D \in \mathbb{R}^{d_1 \times d_1}$ such that
\begin{align} \label{W_H cdots W_2 = [U_Ss,0]D 1er cas}
    W_H \cdots W_2 = [U_{\Ss},0_{d_y \times (d_1-r)}]D
\end{align}
and since $\Wbf$ is associated with $\Ss$, we have
\begin{align}\label{W_H cdots W_1 1er cas}
    W_H \cdots W_1 = U_{\Ss} U_{\Ss}^T \Sigma_{YX}\Sigma_{XX}^{-1} \;.
\end{align}
\sloppy Using \eqref{(i,1) not tightened} and $D$ invertible, we have $\textit{rk}(W_{i-1} \cdots W_2 D^{-1}) = \textit{rk}(W_{i-1} \cdots W_2 ) > r$.
Hence there exists $g \in \llbracket r+1,d_1 \rrbracket$ such that 
\begin{align*}
    (W_{i-1} \cdots W_2D^{-1})_{.,g} \neq 0 \;.
\end{align*}
Therefore, there exists $a \in \mathbb{R}^{d_{i-1}}$ such that 
\begin{align}\label{eq pour a 1er cas}
    a^T(W_{i-1} \cdots W_2 D^{-1})_{.,g} = 1 \;.
\end{align}
Recall that $k,l$ satisfy \eqref{ U_k^T(W_H cdots W_i+1)_.,l neq 0}.
We define $\Wbf'_{\beta} = (W'^{\beta}_H, \ldots , W'^{\beta}_1)$ by 
\begin{align*}
\begin{cases}
    W'^{\beta}_i = \beta W'_i = \beta e_l a^T \in \mathbb{R}^{d_i \times d_{i-1}}, \text{ where } e_l \in \mathbb{R}^{d_i}
    \\
    W'^{\beta}_1 = W'_1 =  D^{-1} e_g U_k^T \Sigma_{YX}\Sigma_{XX}^{-1} \in \mathbb{R}^{d_1 \times d_x}, \text{ where } e_g \in \mathbb{R}^{d_1} \\
    W'^{\beta}_h = 0 \quad \forall h \in \llbracket 2,H \rrbracket \setminus \{i\}
\end{cases}
\end{align*}
We set $\Wbf^{\beta}(t) = (W^{\beta}_H(t), \ldots , W^{\beta}_1(t))$ such that 
$W^{\beta}_h(t) = W_h + t W'^{\beta}_h$ for $h \in \llbracket 1 , H \rrbracket$.
We have
\begin{align*}
W^{\beta}(t) :&= W^{\beta}_H(t) \cdots W^{\beta}_1(t) \\
&= W_H \cdots W_{i+1} (W_i + t \beta W'_i) W_{i-1} \cdots W_2 (W_1 + t W'_1) \\
&= W_H \cdots W_1 + t(\beta W_H \cdots W_{i+1} W'_i W_{i-1} \cdots W_1 + W_H \cdots W_2 W'_1) \\
& \quad +\beta t^2 W_H \cdots W_{i+1} W'_i W_{i-1} \cdots W_2 W'_1 \;.
\end{align*}
Using \eqref{W_H cdots W_2 = [U_Ss,0]D 1er cas} and \eqref{W_H cdots W_1 1er cas}, we obtain
\begin{align*}
W^{\beta}(t) &= U_{\Ss} U_{\Ss}^T \Sigma_{YX}\Sigma_{XX}^{-1} + t (\beta W_H \cdots W_{i+1} W'_i W_{i-1} \cdots W_1 + [U_{\Ss},0]D D^{-1} e_{g} U_k^T \Sigma_{YX}\Sigma_{XX}^{-1}) \\
& \quad + \beta t^2 (W_H \cdots W_{i+1})_{.,l} a^T (W_{i-1} \cdots W_2 D^{-1})_{.,g} U_k^T \Sigma_{YX}\Sigma_{XX}^{-1} \;.
\end{align*}
Using \eqref{eq pour a 1er cas} and $g \in \llbracket r+1,d_1 \rrbracket$, we have
\begin{align*}
W^{\beta}(t) &= U_{\Ss} U_{\Ss}^T \Sigma_{YX}\Sigma_{XX}^{-1} + t \beta W_H \cdots W_{i+1} W'_i W_{i-1} \cdots W_1 + \beta t^2 (W_H \cdots W_{i+1})_{.,l} U_k^T \Sigma_{YX}\Sigma_{XX}^{-1} \;.
\end{align*}
Denoting $N = W_H \cdots W_{i+1} W'_i W_{i-1} \cdots W_1$, we have
\begin{align*}
L(\Wbf^{\beta}(t)) &= \| W^{\beta}(t) X - Y \|^2 \\
&= \| U_{\Ss} U_{\Ss}^T \Sigma_{YX}\Sigma_{XX}^{-1}X - Y + t\beta NX + \beta t^2 (W_H \cdots W_{i+1})_{.,l} U_k^T \Sigma_{YX}\Sigma_{XX}^{-1}X \|^2 \;.
\end{align*}
Expanding the square, 
the second-order term $c_2(\Wbf , \Wbf'_{\beta}) t^2$ has a coefficient equal to
\begin{align*}
c_2(\Wbf , \Wbf'_{\beta}) &= \beta^2\|NX\|^2 + 2 \beta \tr((W_H \cdots W_{i+1})_{.,l} U_k^T \Sigma_{YX}\Sigma_{XX}^{-1}X X^T  \Sigma_{XX}^{-1}\Sigma_{XY} U_{\Ss} U_{\Ss}^T) \\
& \quad -2 \beta \tr((W_H \cdots W_{i+1})_{.,l} U_k^T \Sigma_{YX}\Sigma_{XX}^{-1}XY^T) \\
&= \beta^2\|NX\|^2 + 2 \beta \tr((W_H \cdots W_{i+1})_{.,l} U_k^T \Sigma U_{\Ss} U_{\Ss}^T) -2 \beta \tr((W_H \cdots W_{i+1})_{.,l} U_k^T \Sigma) \\
&= \beta^2\|NX\|^2 -2 \beta \lambda_k U_k^T (W_H \cdots W_{i+1})_{.,l} \;,
\end{align*}
where the last equality follows from Lemma \ref{U_S sigma U_Q = 0} and $k \notin \Ss$, and $U^T \Sigma = \Lambda U^T$ and the cyclic property of the trace.\\ 
Using Lemma \ref{sigma invertible} and \eqref{ U_k^T(W_H cdots W_i+1)_.,l neq 0}, we have $\lambda_k U_k^T (W_H \cdots W_{i+1})_{.,l} \neq 0$, hence we can choose $\beta$ according to \eqref{choosing beta}, such that $c_2(\Wbf , \Wbf'_{\beta})<0$.
Therefore, $\Wbf$ is not a second-order critical point.

\subsection{2nd Case:  \texorpdfstring{$i=H$}{} and \texorpdfstring{$j=1$}{}} \label{2nd case}
In this case, the two complementary blocks are $\Sigma_{XY}$ and $W_{H-1} \cdots W_2$.
We follow again the same lines as above.
Since the pivot $(i,j)$ is not tightened, using Proposition \ref{all blocks ranks geq than r}, we have
\begin{align} \label{(H,1) not tightened}
    \rk(W_{H-1} \cdots W_2)>r \;.
\end{align}
Again, since $\Wbf$ is a first-order critical point, using Lemma \ref{clem}, there exists an invertible matrix $D \in \mathbb{R}^{d_1 \times d_1}$ such that
\begin{align}\label{W_H cdots W_2 = [U_Ss,0]D 2e cas}
    W_H \cdots W_2 = [U_{\Ss},0_{d_y \times (d_1-r)}]D
\end{align}
and since $\Wbf$ is associated with $\Ss$, we have
\begin{align}\label{W_H cdots W_1 2e cas}
    W_H \cdots W_1 = U_{\Ss} U_{\Ss}^T \Sigma_{YX}\Sigma_{XX}^{-1} \;.
\end{align}
Using \eqref{(H,1) not tightened} and $D$ invertible, we have $\textit{rk}(W_{H-1} \cdots W_2 D^{-1}) = \textit{rk}(W_{H-1} \cdots W_2 ) > r$.
Hence there exists $g \in \llbracket r+1,d_1 \rrbracket$ such that 
\begin{align*}
    (W_{i-1} \cdots W_2D^{-1})_{.,g} \neq 0 \;.
\end{align*}
Therefore, there exists $a \in \mathbb{R}^{d_{H-1}}$ such that 
\begin{align}\label{eq a 2nd case}
    a^T(W_{H-1} \cdots W_2 D^{-1})_{.,g} = 1 \;.
\end{align}
We define $\Wbf'_{\beta} = (W'^{\beta}_H, \ldots , W'^{\beta}_1)$ by
\begin{align*}
    \begin{cases}
        W'^{\beta}_H = \beta W'_H = \beta U_{r+1} a^T \in \mathbb{R}^{d_y \times d_{H-1}} 
        \\ 
        W'^{\beta}_1 = W'_1 = D^{-1} e_{g} U_{r+1}^T \Sigma_{YX}\Sigma_{XX}^{-1} \in \mathbb{R}^{d_1 \times d_x}, \text{ where } e_g \in \mathbb{R}^{d_1} 
        \\
        W'^{\beta}_h = 0 \quad \forall h \in \llbracket 2,H-1 \rrbracket \;.
    \end{cases}
\end{align*}
We set $\Wbf^{\beta}(t) = (W^{\beta}_H(t), \ldots , W^{\beta}_1(t))$ such that 
$W^{\beta}_h(t) = W_h + t W'^{\beta}_h$, for all $h \in \llbracket 1 , H \rrbracket$.
We have
\begin{align*}
W^{\beta}(t) :&= W^{\beta}_H(t) \cdots W^{\beta}_1(t) \\
&= (W_H + t \beta W'_H) W_{H-1} \cdots W_2 (W_1 + t W'_1) \\
&= W_H \cdots W_1 + t(\beta W'_H W_{H-1} \cdots W_1 + W_H \cdots W_2 W'_1) +\beta t^2 W'_H W_{H-1} \cdots W_2 W'_1 \;.
\end{align*}
Using \eqref{W_H cdots W_2 = [U_Ss,0]D 2e cas} and \eqref{W_H cdots W_1 2e cas}, then \eqref{eq a 2nd case} and $g \in \llbracket r+1,d_1 \rrbracket$, we obtain
\begin{align*}
W^{\beta}(t) &= U_{\Ss} U_{\Ss}^T \Sigma_{YX}\Sigma_{XX}^{-1} + t (\beta W'_H W_{H-1} \cdots W_1 + [U_{\Ss},0]D D^{-1} e_{g} U_{r+1}^T \Sigma_{YX}\Sigma_{XX}^{-1}) \\
& \quad + \beta t^2 U_{r+1} a^T(W_{H-1} \cdots W_{2}D^{-1})_{.,g} U_{r+1}^T \Sigma_{YX}\Sigma_{XX}^{-1} \\
&= U_{\Ss} U_{\Ss}^T \Sigma_{YX}\Sigma_{XX}^{-1} + t \beta W'_H W_{H-1} \cdots W_1 + \beta t^2 U_{r+1} U_{r+1}^T \Sigma_{YX}\Sigma_{XX}^{-1} \;.
\end{align*}
Denoting by $N = W'_H W_{H-1} \cdots W_1$, we have
\begin{align*}
L(\Wbf^{\beta}(t)) &= \| W^{\beta}(t) X - Y \|^2 \\
&= \| U_{\Ss} U_{\Ss}^T \Sigma_{YX}\Sigma_{XX}^{-1}X - Y + t \beta NX + \beta t^2 U_{r+1} U_{r+1}^T \Sigma_{YX}\Sigma_{XX}^{-1}X \|^2 \;.
\end{align*}
As previously, expanding the square,
we can see that the second-order coefficient $c_2(\Wbf , \Wbf'_{\beta})$ of the polynomial $L(\Wbf^{\beta}(t))$ is given by
\begin{align*}
c_2(\Wbf , \Wbf'_{\beta}) &= \beta^2\|NX\|^2 + 2 \beta \tr(U_{r+1} U_{r+1}^T \Sigma_{YX}\Sigma_{XX}^{-1}X X^T  \Sigma_{XX}^{-1}\Sigma_{XY} U_{\Ss} U_{\Ss}^T) \\ & \quad  -2 \beta \tr(U_{r+1} U_{r+1}^T \Sigma_{YX}\Sigma_{XX}^{-1}XY^T) \\
&= \beta^2\|NX\|^2 + 2 \beta \tr(U_{r+1} U_{r+1}^T \Sigma U_{\Ss} U_{\Ss}^T)-2 \beta \tr(U_{r+1} U_{r+1}^T \Sigma) \;.
\end{align*}
Using the cyclic property of the trace, $U_{\Ss}^T U_{r+1} = 0$ (see Lemma \ref{properties of orhtogonality}), and $\Sigma U_{r+1} = \lambda_{r+1} U_{r+1}$, we obtain
\begin{align*}
c_2(\Wbf , \Wbf'_{\beta}) &= \beta^2\|NX\|^2 -2 \beta \lambda_{r+1} U_{r+1}^T U_{r+1} \\
&= \beta^2\|NX\|^2 -2 \beta \lambda_{r+1} \;.
\end{align*}
Using Lemma \ref{sigma invertible}, we have $\lambda_{r+1} \neq 0$, hence we can choose $\beta$ according to \eqref{choosing beta} such that $c_2(\Wbf , \Wbf'_{\beta})<0$.
Therefore $\Wbf$ is not a second-order critical point.

\subsection{3rd Case: \texorpdfstring{$i=H$}{} and \texorpdfstring{$j \in \llbracket 2,H-1 \rrbracket$}{}} \label{3rd case}
In this case, the two complementary blocks are $W_{j-1} \cdots W_1 \Sigma_{XY}$ and $W_{H-1} \cdots W_{j+1}$.
We follow again the same lines as above.
Since the pivot $(i,j)$ is not tightened, using Proposition \ref{all blocks ranks geq than r}, we have
\begin{align}\label{(H,j) not tightened}
    \begin{cases}
     \rk(W_{H-1} \cdots W_{j+1})>r \\
     \rk(W_{j-1} \cdots W_1 \Sigma_{XY})>r \;.
    \end{cases}
\end{align}
Let us first show that there exist $k \in \llbracket r+1 , d_y \rrbracket$ and $l \in \llbracket 1, d_{j-1}\rrbracket$ such that 
\begin{align} \label{(W_j-1 cdots W_1)_l,. Sigma_XY U_k neq 0}
    (W_{j-1} \cdots W_1)_{l,.} \Sigma_{XY} U_k \neq 0 \;.
\end{align}
Indeed, assume by contradiction that for all $k \in \llbracket r+1 , d_y \rrbracket$ and $ l \in \llbracket 1,d_{j-1} \rrbracket$ we have
$$(W_{j-1} \cdots W_1)_{l,.} \Sigma_{XY} U_k = 0.$$
Recalling that $Q = \llbracket1,d_y \rrbracket \setminus \Ss = \llbracket r+1,d_y \rrbracket$, we obtain $W_{j-1} \cdots W_1 \Sigma_{XY} U_Q = 0$, and using, from Lemma \ref{properties of orhtogonality}, that $I_{d_y} = U_{\Ss} U_{\Ss}^T + U_Q U_Q^T$, we have
\begin{align*}
    W_{j-1} \cdots W_1 \Sigma_{XY} &= W_{j-1} \cdots W_1 \Sigma_{XY} (U_{\Ss} U_{\Ss}^T + U_{Q} U_{Q}^T) \\
    &= W_{j-1} \cdots W_1 \Sigma_{XY} U_{\Ss} U_{\Ss}^T \;.
\end{align*}
Therefore, $$ \rk(W_{j-1} \cdots W_1 \Sigma_{XY}) = \rk(W_{j-1} \cdots W_1 \Sigma_{XY} U_{\Ss} U_{\Ss}^T). $$
The latter is impossible since $\rk(W_{j-1} \cdots W_1 \Sigma_{XY} U_{\Ss} U_{\Ss}^T) \leq |\Ss| = r$ is not compatible with \eqref{(H,j) not tightened}.
Therefore \eqref{(W_j-1 cdots W_1)_l,. Sigma_XY U_k neq 0} holds. \\
We know that $\rk(W_H \cdots W_{j+1}) \geq \rk(W_H \cdots W_{1}) = r$.
Therefore, depending on the value of $\rk(W_H \cdots W_{j+1})$, we distinguish two situations: either $\rk(W_H \cdots W_{j+1})>r$ or $\rk(W_H \cdots W_{j+1})=r$.\\
When $\rk(W_H \cdots W_{j+1})>r$, since $\Sigma_{XY}$ is of full column rank, we have $\rk(\Sigma_{XY} W_H \cdots W_{j+1}) = \rk(W_H \cdots W_{j+1})>r$.
Also, using \eqref{(H,j) not tightened}, we have $\rk(W_{j-1} \cdots W_2) \geq \rk(W_{j-1} \cdots W_1 \Sigma_{XY}) > r$.
Hence, in this case, the pivot $(j,1)$ is not tightened either.
We have already proved in Section \ref{1st case} (beware that the pivot is denoted $(i,1)$, not $(j,1)$, in Section \ref{1st case}) that, when such a pivot is not tightened, $\Wbf$ is not a second-order critical point.
This concludes the proof in the case $\rk(W_H \cdots W_{j+1}) > r.$ \\
In the rest of the section we assume that $\rk(W_H \cdots W_{j+1})=r$. \\
Using \eqref{(H,j) not tightened}, we have $\rk(W_{H-1} \cdots W_{j+1}) > r = \rk(W_H \cdots W_{j+1})$.
Applying the rank-nullity theorem
we obtain $$\text{Ker}(W_{H-1} \cdots W_{j+1}) \subsetneq \text{Ker}(W_{H} \cdots W_{j+1}).$$
Therefore there exists $b \in \mathbb{R}^{d_j}$ such that
\begin{align} \label{b 3rd case}
    \begin{cases}
        b \in \text{Ker}(W_H \cdots W_{j+1}) \\
        b \notin \text{Ker}(W_{H-1} \cdots W_{j+1}) \;.
    \end{cases}
\end{align}
Hence, there also exists $a \in \mathbb{R}^{d_{H-1}}$ such that 
\begin{align}\label{eq a 3rd case}
    a^T W_{H-1} \cdots W_{j+1} b = 1 \;.
\end{align}
Recall that $k,l$ satisfy \eqref{(W_j-1 cdots W_1)_l,. Sigma_XY U_k neq 0}.
We define $\Wbf'_{\beta} = (W'^{\beta}_H, \ldots , W'^{\beta}_1)$ by
\begin{align*}
    \begin{cases}
        W'^{\beta}_H = \beta W'_H = \beta U_k a^T \in \mathbb{R}^{d_y \times d_{H-1}} 
        \\
        W'^{\beta}_j = W'_j = b e_l^T \in \mathbb{R}^{d_j \times d_{j-1}}, \text{ where } e_l \in \mathbb{R}^{d_{j-1}} \\
        W'^{\beta}_h = 0 \quad \forall h \in \llbracket 1,H \rrbracket \setminus \{i,j\}
    \end{cases}
\end{align*}
We set $\Wbf^{\beta}(t) = (W^{\beta}_H(t), \ldots , W^{\beta}_1(t))$ such that 
$W^{\beta}_h(t) = W_h + t W'^{\beta}_h$ for $h \in \llbracket 1 , H \rrbracket$.
We have
\begin{align*}
W^{\beta}(t) :&= W^{\beta}_H(t) \cdots W^{\beta}_1(t) \\
&= (W_H + t \beta W'_H) W_{H-1} \cdots W_{j+1} (W_j + t W'_j) W_{j-1} \cdots W_1 \\
&= W_H \cdots W_1 + t(\beta W'_H W_{H-1} \cdots W_1 + W_H \cdots W_{j+1} W'_j W_{j-1} \cdots W_1) \\
& \quad + t^2 \beta W'_H \cdots W_{j+1} W'_j W_{j-1} \cdots W_1 \;.
\end{align*}
Using Proposition \ref{global map and critical values} and the definition of $\Wbf'_{\beta}$ above, we obtain
\begin{align*}
W^{\beta}(t) &= U_{\Ss} U_{\Ss}^T \Sigma_{YX}\Sigma_{XX}^{-1} + t (\beta W'_H W_{H-1} \cdots W_1 + W_H \cdots W_{j+1} b e_l^T W_{j-1} \cdots W_1) \\
& \quad + \beta t^2 U_k a^T W_{H-1} \cdots W_{j+1} b(W_{j-1} \cdots W_1)_{l,.} \\
&= U_{\Ss} U_{\Ss}^T \Sigma_{YX}\Sigma_{XX}^{-1} + t \beta W'_H W_{H-1} \cdots W_1 + \beta t^2 U_k (W_{j-1} \cdots W_1)_{l,.} \;,
\end{align*}
where the last equality follows from \eqref{b 3rd case} and \eqref{eq a 3rd case}
. \\
Denoting $N = W'_H W_{H-1} \cdots W_1$, we have
\begin{align*}
L(\Wbf^{\beta}(t)) &= \| W^{\beta}(t) X - Y \|^2 \\
&= \| U_{\Ss} U_{\Ss}^T \Sigma_{YX}\Sigma_{XX}^{-1}X - Y + t \beta NX + \beta t^2 U_k (W_{j-1} \cdots W_1)_{l,.}X \|^2 \;.
\end{align*}
Using the cyclic property of the trace, and, since $k \notin \Ss$, $U_{\Ss}^T U_k = 0$, we get in this case a second-order coefficient equal to
\begin{align*}
c_2(\Wbf , \Wbf'_{\beta}) &= \beta^2\|NX\|^2 + 2 \beta \tr\left(U_k (W_{j-1} \cdots W_1)_{l,.}X X^T  \Sigma_{XX}^{-1}\Sigma_{XY} U_{\Ss} U_{\Ss}^T\right) \\ & \quad  -2 \beta \tr(U_k(W_{j-1} \cdots W_1)_{l,.}\Sigma_{XY}) \\
&= \beta^2\|NX\|^2 -2 \beta (W_{j-1} \cdots W_1)_{l,.} \Sigma_{XY} U_k \;.
\end{align*}
Since from \eqref{(W_j-1 cdots W_1)_l,. Sigma_XY U_k neq 0}, $ (W_{j-1} \cdots W_1)_{l,.} \Sigma_{XY} U_k \neq 0$, we can choose $\beta$ according to \eqref{choosing beta}, such that $c_2(\Wbf , \Wbf'_{\beta})<0$.
Therefore $\Wbf$ is not a second-order critical point.

\subsection{4th Case: \texorpdfstring{$i,j \in \llbracket 2,H-1 \rrbracket$}{}, with \texorpdfstring{$i>j$}{}} \label{4th case}
In this case, the two complementary blocks are $W_{j-1} \cdots W_1 \Sigma_{XY} W_H \cdots W_{i+1}$ and $W_{i-1} \cdots W_{j+1}$.
We follow again the same lines as above.
Since the pivot $(i,j)$ is not tightened,
using Proposition \ref{all blocks ranks geq than r}, we have 
\begin{align}\label{(i,j) not tightened}
    \begin{cases}
     \rk(W_{i-1} \cdots W_{j+1})>r \\
     \rk(W_{j-1} \cdots W_1\Sigma_{XY}W_H \cdots W_{i+1})>r \;.
    \end{cases}
\end{align}
Let us first show that there exist $k \in \llbracket 1, d_i\rrbracket$ and $l \in \llbracket 1, d_{j-1}\rrbracket$ such that 
\begin{align} \label{W_j-1...W_1 sigma xy U_Q U_Q^T W_H... W_i+1 neq 0}
    (W_{j-1} \cdots W_1)_{l,.}\Sigma_{XY} U_Q U_Q^T (W_{H} \cdots W_{i+1})_{.,k} \neq 0 \;.
\end{align}
Indeed, assume by contradiction that, for all $ k \in \llbracket 1, d_i\rrbracket$ and $ l \in \llbracket 1, d_{j-1}\rrbracket$, we have $$(W_{j-1} \cdots W_1)_{l,.}\Sigma_{XY} U_Q U_Q^T (W_{H} \cdots W_{i+1})_{.,k} = 0 .$$ Then $W_{j-1} \cdots W_1\Sigma_{XY} U_Q U_Q^T W_{H} \cdots W_{i+1} = 0$, and so, using $I_{d_y} = U_{\Ss} U_{\Ss}^T + U_Q U_Q^T$, we would have
\begin{align*}
W_{j-1} \cdots W_1\Sigma_{XY} W_{H} \cdots W_{i+1} &= W_{j-1} \cdots W_1\Sigma_{XY}I_{d_y} W_{H} \cdots W_{i+1} \\
&= W_{j-1} \cdots W_1\Sigma_{XY} \left(U_{\Ss}U_{\Ss}^T + U_Q U_Q^T\right) W_{H} \cdots W_{i+1} \\
&= W_{j-1} \cdots W_1\Sigma_{XY} U_{\Ss}U_{\Ss}^T W_{H} \cdots W_{i+1} \;.
\end{align*}
Therefore,
$$ \rk(W_{j-1} \cdots W_1\Sigma_{XY} W_{H} \cdots W_{i+1}) = \rk(W_{j-1} \cdots W_1\Sigma_{XY} U_{\Ss}U_{\Ss}^T W_{H} \cdots W_{i+1}) .$$
The latter is impossible since $\rk(W_{j-1} \cdots W_1\Sigma_{XY} U_{\Ss}U_{\Ss}^T W_{H} \cdots W_{i+1}) \leq |\Ss| = r$ is not compatible with \eqref{(i,j) not tightened}.
Therefore \eqref{W_j-1...W_1 sigma xy U_Q U_Q^T W_H... W_i+1 neq 0} holds.\\

We know that $\rk(W_H \cdots W_{j+1}) \geq \rk(W_H \cdots W_{1}) = r$.
Therefore, depending on the value of $\rk(W_H \cdots W_{j+1})$, we distinguish two situations: either $\rk(W_H \cdots W_{j+1})>r$ or $\rk(W_H \cdots W_{j+1})=r$.\\
When $\rk(W_H \cdots W_{j+1})>r$, since $\Sigma_{XY}$ is of full column rank, we have $\rk(\Sigma_{XY} W_H \cdots W_{j+1}) = \rk(W_H \cdots W_{j+1})>r$.
Also, using \eqref{(i,j) not tightened}, we have $\rk(W_{j-1} \cdots W_2) \geq \rk(W_{j-1} \cdots W_1 \Sigma_{XY} W_H \cdots W_{i+1}) > r$.
Hence, in this case, the pivot $(j,1)$ is not tightened either.
We have already proved in Section \ref{1st case} (beware that the pivot is denoted $(i,1)$, not $(j,1)$, in Section \ref{1st case}) that, when such a pivot is not tightened, $\Wbf$ is not a second-order critical point.
This concludes the proof when $\rk(W_H \cdots W_{j+1}) > r .$ \\
In the rest of the section we assume that $\rk(W_H \cdots W_{j+1})=r$. \\
Using \eqref{(i,j) not tightened}, we have $\rk(W_{i-1} \cdots W_{j+1}) > r = \rk(W_H \cdots W_{j+1})$.
Applying the rank-nullity theorem, we obtain
$$\text{Ker}(W_{i-1} \cdots W_{j+1}) \subsetneq \text{Ker}(W_{H} \cdots W_{j+1}).$$
Therefore there exists $b \in \mathbb{R}^{d_j}$ such that
\begin{align} \label{b 4th case}
    \begin{cases}
        b \in \text{Ker}(W_H \cdots W_{j+1}) \\
        b \notin \text{Ker}(W_{i-1} \cdots W_{j+1})\;.
    \end{cases}
\end{align}
Hence, there also exists $a \in \mathbb{R}^{d_{i-1}}$ such that 
\begin{align}\label{eq a 4th case}
    a^T W_{i-1} \cdots W_{j+1} b = 1 \;.
\end{align}
Recall that $k,l$ satisfy \eqref{W_j-1...W_1 sigma xy U_Q U_Q^T W_H... W_i+1 neq 0}.
We define 
$\Wbf'_{\beta} = (W'^{\beta}_H, \ldots , W'^{\beta}_1)$ by
\begin{align*}
    \begin{cases}
        W'^{\beta}_i = \beta W'_i = \beta e_k a^T \in \mathbb{R}^{d_i \times d_{i-1}} \text{ where } e_k \in \mathbb{R}^{d_{i}}
        \\
        W'^{\beta}_j = W'_j = b e_l^T \in \mathbb{R}^{d_j \times d_{j-1}} \text{ where } e_l \in \mathbb{R}^{d_{j-1}} \\
        W'^{\beta}_h = 0 \quad \forall h \in \llbracket 1,H \rrbracket \setminus \{i,j\} \;.
    \end{cases}
\end{align*}
We set $\Wbf^{\beta}(t) = (W^{\beta}_H(t), \ldots , W^{\beta}_1(t))$ with 
$W^{\beta}_h(t) = W_h + t W'^{\beta}_h$ for all $h \in \llbracket 1 , H \rrbracket$.
We have,
\begin{align*}
W^{\beta}(t) :&= W^{\beta}_H(t) \cdots W^{\beta}_1(t) \\
&= W_H \cdots W_{i+1} (W_i + t \beta W'_i) W_{i-1} \cdots W_{j+1} (W_j + t W'_j) W_{j-1} \cdots W_1 \\ 
&= W_H \cdots W_1 + t(\beta W_H  \cdots W_{i+1}W'_i W_{i-1} \cdots W_1 + W_H \cdots W_{j+1} W'_j W_{j-1} \cdots W_1) \\
& \quad + \beta t^2 W_H \cdots W_{i+1}W'_i W_{i-1} \cdots W_{j+1} W'_j W_{j-1} \cdots W_1 \;.
\end{align*}
Using Proposition \ref{global map and critical values} and the definition of $\Wbf'_{\beta}$ above, we obtain
\begin{align*}
W^{\beta}(t) &= U_{\Ss} U_{\Ss}^T \Sigma_{YX}\Sigma_{XX}^{-1} + t (\beta W_H  \cdots W_{i+1}W'_i W_{i-1} \cdots W_1 + W_H \cdots W_{j+1} b e_l^T W_{j-1} \cdots W_1) \\
& \quad + \beta t^2 (W_{H} \cdots W_{i+1})_{.,k}a^T W_{i-1} \cdots W_{j+1} b (W_{j-1} \cdots W_1)_{l,.} \\
&= U_{\Ss} U_{\Ss}^T \Sigma_{YX}\Sigma_{XX}^{-1} + t \beta W_H  \cdots W_{i+1}W'_i W_{i-1} \cdots W_1 + \beta t^2 (W_{H} \cdots W_{i+1})_{.,k} (W_{j-1} \cdots W_1)_{l,.} \;,
\end{align*}
where the last equality follows from \eqref{b 4th case} and \eqref{eq a 4th case}
. \\
Denoting $N = W_H  \cdots W_{i+1}W'_i W_{i-1} \cdots W_1$, we have
\begin{align*}
L(\Wbf^{\beta}(t)) &= \| W^{\beta}(t) X - Y \|^2 \\
&= \| U_{\Ss} U_{\Ss}^T \Sigma_{YX}\Sigma_{XX}^{-1}X - Y + t \beta NX
+ \beta t^2(W_{H} \cdots W_{i+1})_{.,k} (W_{j-1} \cdots W_1)_{l,.}X \|^2 \;.
\end{align*}
The second-order coefficient of $L(\Wbf^{\beta}(t))$ is equal to
\begin{align*}
c_2(\Wbf , \Wbf'_{\beta}) &= \beta^2\|NX\|^2
+ 2 \beta \tr\left((W_{H} \cdots W_{i+1})_{.,k} (W_{j-1} \cdots W_1)_{l,.}X X^T  \Sigma_{XX}^{-1}\Sigma_{XY} U_{\Ss} U_{\Ss}^T\right) \\
& \quad - 2 \beta \tr\left((W_{H} \cdots W_{i+1})_{.,k} (W_{j-1} \cdots W_1)_{l,.}\Sigma_{XY}\right) \\
&= \beta^2\|NX\|^2 + 2 \beta \tr\left((W_{H} \cdots W_{i+1})_{.,k} (W_{j-1} \cdots W_1)_{l,.}\Sigma_{XY} (U_{\Ss} U_{\Ss}^T - I_{d_y})\right) \;.
\end{align*}
Using, from Lemma \ref{properties of orhtogonality}, that $U_{\Ss} U_{\Ss}^T - I_{d_y} = - U_Q U_Q^T$, and then the cyclic property of the trace, we obtain
\begin{align*}
c_2(\Wbf , \Wbf'_{\beta}) &= \beta^2\|NX\|^2 - 2 \beta \tr\left((W_{H} \cdots W_{i+1})_{.,k} (W_{j-1} \cdots W_1)_{l,.}\Sigma_{XY}U_Q U_Q^T\right) \\
&= \beta^2\|NX\|^2 -2 \beta  (W_{j-1} \cdots W_1)_{l,.} \Sigma_{XY} U_Q U_Q^T(W_{H} \cdots W_{i+1})_{.,k} \;.
\end{align*}
Since from \eqref{W_j-1...W_1 sigma xy U_Q U_Q^T W_H... W_i+1 neq 0}, $ (W_{j-1} \cdots W_1)_{l,.} \Sigma_{XY} U_Q U_Q^T(W_{H} \cdots W_{i+1})_{.,k} \neq 0$, we can choose $\beta$ according to \eqref{choosing beta} such that $c_2(\Wbf , \Wbf'_{\beta})<0$.
Therefore, $\Wbf$ is not a second-order critical point.\\

\section{Non-strict Saddle Points}\label{Appendix non strict saddles}

In this section, we prove the results related to non-strict saddle points (see Section \ref{non-strict saddles}).

\subsection{Proof of Proposition \ref{prop 4}} \label{proof prop 4}
To prove Proposition \ref{prop 4}, we show that for any $\Wbf'$, $c_2(\Wbf,\Wbf') \geq 0$, which is equivalent to say (see Lemma \ref{Characterization of critical points in our settings}) that $\Wbf$ is a second-order critical point.
We follow the proof strategy sketched in Section \ref{non-strict saddles} after the statement of Proposition \ref{prop 4}, and use the same notation introduced therein.
Note that a first-order critical point can only be tightened if $H \geq 3$.
Therefore, in all of this section we make the assumption $H \geq 3$.
Recall that $m$ is the number of examples in our sample, $\Ss = \llbracket 1,r \rrbracket $, with $r<r_{max}$.
We set $Q = \llbracket r+1,d_y \rrbracket$.\\
Recall also that
\begin{align*}
    \Sigma^{1/2} = \Sigma_{YX}\Sigma_{XX}^{-1}X \in \mathbb{R}^{d_y \times m}.
\end{align*}
and
\begin{align*}
    \Sigma^{1/2} = U \Delta V^T 
\end{align*}
is a Singular Value Decomposition of $\Sigma^{1/2}$, where $\Delta \in \mathbb{R}^{d_y \times m}$ is such that $\Delta_{ii} = \sqrt{\lambda_i}$ for all $ i \in  \llbracket 1,d_y \rrbracket$, and ${(\lambda_i)}_{i=1..d_y}$ are the eigenvalues of $\Sigma$. \\
We denote 
\begin{align}\label{delta S}
    \Delta^{({\Ss})} = \text{diag}(\sqrt{\lambda_{1}}, \ldots ,\sqrt{\lambda_{r}}) \in \mathbb{R}^{r \times r}
\end{align}
and
\begin{align}\label{delta Q}
    \Delta^{(Q)} = \text{diag}(\sqrt{\lambda_{r+1}}, \ldots ,\sqrt{\lambda_{d_y}}) \in \mathbb{R}^{(d_y-r) \times (d_y-r)} \;.
\end{align}
Recall that, from Section \ref{non-strict saddles}, $c_2(\Wbf,\Wbf') = FT + ST$ .\\
In what follows, we are going to present a key lemma, then various quick technical lemmas, then we simplify the expressions of $FT$ and $ST$ and conclude the proof of Proposition \ref{prop 4}.
Then, we prove all the lemmas of Appendix \ref{proof prop 4}.\\
We present a lemma which uses that $\Wbf$ is tightened to simplify some products of weight matrices and lighten further calculations.
This is a key lemma as it introduces indices $p$ and $q$ which will be used multiple times in the proof.
\begin{lemma}\label{E}
    Suppose Assumption \ref{Assump H} in Section \ref{Settings} holds true.
    Let $\Wbf = (W_H, \ldots, W_1)$ be a first-order critical point of $L$ verifying the hypotheses of Proposition \ref{prop 4}, and $r$, $\Ss$, $Q$, $(Z_h)_{h=1..H}$ as in Proposition \ref{prop 4}.
    If $\Wbf$ is tightened,
    then, there exist $p \in \llbracket 3 , H \rrbracket$ and $q \in \llbracket 1, \min(p-1 , H-2) \rrbracket$ such that:
    \begin{align}
        \forall i \in \llbracket 1,p-1 \rrbracket, \qquad
        & W_H \cdots W_{i+1} = \left[ U_{\Ss}\ , \ 0 \right] \label{W_H...W_p simplified} \\
        \forall i \in \llbracket p,H\rrbracket, \qquad 
        & W_{i-1} \cdots W_2 = \left[
        \begin{array}{c  c}
        I_r & 0 \\
        0 & 0
        \end{array}
        \right] \label{W_i-1...W_2 simplified} \\
        \forall i \in \llbracket q+1 , H \rrbracket, \qquad 
        & Z_{i-1} \cdots Z_2 Z_1 \Sigma_{XY} U_Q = 0 \label{hammou} \\
        \forall i \in \llbracket 1,q \rrbracket, \qquad
        & W_{H-1} \cdots W_{i+1} = \left[
        \begin{array}{c  c}
        I_r & 0 \\
        0 & 0
        \end{array}
        \right] \;. \label{W_H-1...W_k+1 simplified}
    \end{align}
\end{lemma}
The proof of Lemma \ref{E} is in Appendix \ref{proof E}.\\

\subsubsection{Useful Technical Lemmas}
We now present technical lemmas which will be useful in Sections \ref{ section simplifying FT}, \ref{section simplifying ST} and \ref{section concluding}.
In all of these Lemmas, we have $\Ss = \llbracket 1,r \rrbracket $ and $Q = \llbracket r+1,d_y \rrbracket$, and Assumption \ref{Assump H} holds true.
\begin{lemma} \label{Sigma_XY U_Q}
    We have $$\Sigma_{XY} U_Q = X V_Q \Delta^{(Q)} \;.$$
\end{lemma}
The proof of Lemma \ref{Sigma_XY U_Q} is in Appendix  \ref{proof Sigma_XY U_Q}. 
\begin{lemma} \label{the formula}
    Let $n$ be a positive integer. For any matrices $A \in \mathbb{R}^{d_y \times n}$ and $B \in \mathbb{R}^{r \times n}$ we have
    \begin{align*}
        \| A + U_{\Ss}B \|^2 = \| U_{\Ss}^T A + B\|^2 + \|U_Q^T A\|^2 \;.
    \end{align*}
\end{lemma}
The proof of Lemma \ref{the formula} is in Appendix \ref{proof the formula}.
\begin{lemma}\label{C}
    Let $n$ be any positive integer. For any matrices $ A \in \mathbb{R}^{n \times r}$ and $B \in \mathbb{R}^{n \times (d_y-r)}$ we have: 
    $$\left\langle  A U_{\Ss}^T \Sigma_{YX}\Sigma_{XX}^{-1}X \ , \ B V_Q^T\right\rangle= 0 \;.$$
\end{lemma}
The proof of Lemma \ref{C} is in Appendix \ref{proof C}.
\begin{lemma}\label{D}
    Let $n$ be any positive integer.
    Let $\Wbf = (W_H, \ldots, W_1)$ be a first-order critical point of $L$ verifying the hypotheses of Proposition \ref{prop 4}, and $r$, $\Ss$, $Q$, $(Z_h)_{h=1..H}$ as in Proposition \ref{prop 4}.
    If $W$ is tightened, then, for $q$ as in Lemma \ref{E},
    for any matrices $A \in \mathbb{R}^{n \times (d_q - r)} $ and $B \in \mathbb{R}^{n \times (d_y - r)}$, we have: 
    $$\left\langle  A Z_q \cdots Z_2 Z_1 X \ ,\ B V_Q^T  \right\rangle =0 \;.$$
\end{lemma}
The proof of Lemma \ref{D} is in Appendix \ref{proof D}.
\begin{lemma} \label{decomposition A delta s carre}
    For any matrix $A \in \mathbb{R}^{(d_y - r) \times r}$ we have
    $$ \| A U_{\Ss}^T \Sigma_{YX}\Sigma_{XX}^{-1}X \|^2 = \sum_{a=1}^r \sum_{b=r+1}^{d_y} (\lambda_a - \lambda_b) ( A_{b-r,a})^2 + \| \Delta^{(Q)} A \|^2 \;.$$
\end{lemma}
The proof of Lemma \ref{decomposition A delta s carre} is in Appendix  \ref{proof decomposition A delta s carre}.
\begin{lemma}\label{simplif cross-product}
     Let $\Wbf = (W_H, \ldots , W_1)$ be a first-order critical point associated with $\Ss$. %
     For any matrix $A \in \mathbb{R}^{d_y \times d_x}$, we have
     $$ \left\langle AX \ ,\ W_H \cdots W_1X - Y \right\rangle = \left\langle  A \ , \ -U_{Q} U_{Q}^T \Sigma_{YX} \right\rangle \;.$$
\end{lemma}
The proof of Lemma \ref{simplif cross-product} is in Appendix \ref{proof simplif cross-product}.\\ 

\subsubsection{Simplifying \texorpdfstring{$FT$}{}} \label{ section simplifying FT}
In this section and the next one, we simplify the expressions of $FT$ and $ST$ as defined in \eqref{FT} and \eqref{ST}.
In order to decompose $FT = a_1 + \|A_2\|^2 + \|A_3\|^2 + \|A_4\|^2$, with $a_1 \geq 0$, we first simplify the terms $T_i$, for $i \in \llbracket 1 , H \rrbracket$, defined in \eqref{T_i}.
Let us first consider $\Wbf$ tightened satisfying the hypotheses of Proposition \ref{prop 4}, and $p$ and $q$ defined as in Lemma \ref{E}.
The simplification of $T_i$ depends on the position of $i$ with regard to $1$ , $q$, $p$ and $H$.
We define $J_1 = \llbracket p, H-1\rrbracket$, $J_2 = \llbracket q+1, p-1\rrbracket$ and $J_3 = \llbracket 2, q\rrbracket$. \\
Note that, according to the convention in Section \ref{Settings}, these sets could be empty.
\begin{itemize}
    \item if $p=H$, $J_1 = \emptyset$
    \item if $q=p-1$, $J_2 = \emptyset$
    \item if $q=1$, $J_3 = \emptyset$ \;.
\end{itemize}
Note also that $\{1\},J_3,J_2,J_1,\{H\}$ are disjoint and $\{1\} \cup J_3 \cup J_2 \cup J_1 \cup \{H\} = \llbracket 1,H \rrbracket$.\\
Depending on the position of $i$, we need to distinguish four cases, in order to simplify $T_i$.

\begin{lemma} \label{simplif all T_i}
    Suppose Assumption \ref{Assump H} in Section \ref{Settings} holds true.
    Let $\Wbf = (W_H, \ldots , W_1)$ be a first-order critical point satisfying the hypotheses of Proposition \ref{prop 4}, and $r$, $\Ss$, $Q$, $(Z_h)_{h=1..H}$ as in Proposition \ref{prop 4}.
    Let $i \in \llbracket 1,H \rrbracket$.
    For any $\Wbf' = (W'_H, \ldots, W'_1)$, recall that 
    , as defined in \eqref{T_i},
    $$T_{i} = W_H \cdots W_{i+1} W'_i W_{i-1} \cdots W_1 X \;.$$
    If $\Wbf$ is tightened, then, for $p$ and $q$ as defined in Lemma \ref{E} and $J_1$,$J_2$,$J_3$ as defined above, we have
    \begin{itemize}
        \item For $i=H$: 
        \begin{align}
            T_H = (W'_H)_{.,1:r} U_{\Ss}^T \Sigma_{YX}\Sigma_{XX}^{-1}X \label{T_H simplified} 
        \end{align}
        \item For $i \in J_1$:
        \begin{align}
            T_i &= U_{\Ss} (W'_i)_{1:r,1:r} U_{\Ss}^T \Sigma_{YX}\Sigma_{XX}^{-1}X + U_Q Z_H Z_{H-1} \cdots Z_{i+1} (W'_i)_{r+1:d_i,1:r} U_{\Ss}^T \Sigma_{YX}\Sigma_{XX}^{-1}X \label{T_i J_1 simplified}
        \end{align}
        \item For $i \in J_2 \cup J_3$:
        \begin{align}
            T_i &= U_{\Ss} (W'_i)_{1:r,1:r} U_{\Ss}^T \Sigma_{YX}\Sigma_{XX}^{-1}X + U_{\Ss} (W'_i)_{1:r,r+1:d_{i-1}} Z_{i-1} \cdots Z_2Z_1 X \label{T_i J_2 J_3 simplified}
        \end{align}
        \item For $i=1$:
        \begin{align}
            T_1 &= U_{\Ss} (W'_1)_{1:r,.} X \label{T_1 simplified}
        \end{align}
    \end{itemize}
\end{lemma}
The proof of Lemma \ref{simplif all T_i} is in Appendix \ref{proof simplif all T_i}. \\
We now simplify $FT$.
Substituting the formulas of Lemma \ref{simplif all T_i} in \eqref{FT} we have
\begin{align*}
    FT &= \left\| \sum_{i=1}^H T_i \right\|^2 \\
    &= \left\| (W'_H)_{.,1:r} U_{\Ss}^T \Sigma_{YX}\Sigma_{XX}^{-1}X \right. \\
    & \quad + \sum_{i \in J_1} \left(U_{\Ss} (W'_i)_{1:r,1:r} U_{\Ss}^T \Sigma_{YX}\Sigma_{XX}^{-1}X + U_Q Z_H Z_{H-1} \cdots Z_{i+1} (W'_i)_{r+1:d_i,1:r} U_{\Ss}^T \Sigma_{YX}\Sigma_{XX}^{-1}X \right) \\
    & \quad + \left. \sum_{i \in J_2 \cup J_3}  \left(U_{\Ss} (W'_i)_{1:r,1:r} U_{\Ss}^T \Sigma_{YX}\Sigma_{XX}^{-1}X + U_{\Ss} (W'_i)_{1:r,r+1:d_{i-1}} Z_{i-1} \cdots Z_2 Z_1 X\right) + U_{\Ss} (W'_1)_{1:r,.} X \right\|^2 \;.
\end{align*}
$FT$ can be identified with a term as $\|A + U_{\Ss}B\|^2$ if we take $$A = (W'_H)_{.,1:r} U_{\Ss}^T \Sigma_{YX}\Sigma_{XX}^{-1}X + \sum_{i \in J_1} U_QZ_H Z_{H-1} \cdots Z_{i+1} (W'_i)_{r+1:d_i,1:r}U_{\Ss}^T \Sigma_{YX}\Sigma_{XX}^{-1}X \;.$$
and 
\begin{align*}
    B &= \sum_{i \in J_1} (W'_i)_{1:r,1:r} U_{\Ss}^T \Sigma_{YX}\Sigma_{XX}^{-1}X \\
    & \quad + \sum_{i \in J_2 \cup J_3}  \left( (W'_i)_{1:r,1:r} U_{\Ss}^T \Sigma_{YX}\Sigma_{XX}^{-1}X + (W'_i)_{1:r,r+1:d_{i-1}} Z_{i-1} \cdots Z_2 Z_1 X\right) + (W'_1)_{1:r,.} X \;.
\end{align*}
Applying Lemma \ref{the formula}, $FT$ becomes:
\begingroup
\allowdisplaybreaks
\begin{align*}
    FT &= \|U_{\Ss}^T A + B \|^2 + \|U_Q^T A\|^2 \\
    &= \left\| U_{\Ss}^T (W'_H)_{.,1:r} U_{\Ss}^T \Sigma_{YX}\Sigma_{XX}^{-1}X + \sum_{i \in J_1} U_{\Ss}^T U_Q Z_H Z_{H-1} \cdots Z_{i+1} (W'_i)_{r+1:d_i,1:r} U_{\Ss}^T \Sigma_{YX}\Sigma_{XX}^{-1}X \right. \\
    & \quad + \sum_{i \in J_1} (W'_i)_{1:r,1:r} U_{\Ss}^T \Sigma_{YX}\Sigma_{XX}^{-1}X \\
    & \quad + \left. \sum_{i \in J_2 \cup J_3}  \left( (W'_i)_{1:r,1:r} U_{\Ss}^T \Sigma_{YX}\Sigma_{XX}^{-1}X + (W'_i)_{1:r,r+1:d_{i-1}} Z_{i-1} \cdots Z_2 Z_1 X \right) + (W'_1)_{1:r,.} X \right\|^2 \\
    & \quad + \left\| U_{Q}^T (W'_H)_{.,1:r} U_{\Ss}^T \Sigma_{YX}\Sigma_{XX}^{-1}X + \sum_{i \in J_1} U_{Q}^T U_Q Z_H Z_{H-1} \cdots Z_{i+1} (W'_i)_{r+1:d_i,1:r} U_{\Ss}^T \Sigma_{YX}\Sigma_{XX}^{-1}X \right\|^2 \;.
\end{align*}
Using Lemma \ref{properties of orhtogonality}, we have $U_{\Ss}^T U_Q = 0$ and $U_Q^T U_Q = I_{d_y-r}$, hence we can write
\begin{align*}
    FT &= FT_1 + FT_2 \;,
\end{align*}
where
\begin{align*}
    FT_1 &= \left\| U_{\Ss}^T (W'_H)_{.,1:r} U_{\Ss}^T \Sigma_{YX}\Sigma_{XX}^{-1}X + \sum_{i \in J_1} (W'_i)_{1:r,1:r} U_{\Ss}^T \Sigma_{YX}\Sigma_{XX}^{-1}X \right. \\
    & \quad + \left. \sum_{i \in J_2 \cup J_3} \left((W'_i)_{1:r,1:r} U_{\Ss}^T \Sigma_{YX}\Sigma_{XX}^{-1}X + (W'_i)_{1:r,r+1:d_{i-1}} Z_{i-1} \cdots Z_2 Z_1 X \right) + (W'_1)_{1:r,.} X \right\|^2 \;,
\end{align*}
and
\begin{align*}
    FT_2 &= \left\| U_{Q}^T (W'_H)_{.,1:r} U_{\Ss}^T \Sigma_{YX}\Sigma_{XX}^{-1}X + \sum_{i \in J_1} Z_H Z_{H-1} \cdots Z_{i+1} (W'_i)_{r+1:d_i,1:r} U_{\Ss}^T \Sigma_{YX}\Sigma_{XX}^{-1}X \right\|^2 \;.
\end{align*}
Let us first simplify $FT_1$. \\
Recall that $m$ is the number of examples in our sample, $V \in \mathbb{R}^{m \times m}$ is the orthogonal matrix defined in \eqref{svd de sigma 1/2} and $Q = \llbracket r+1,d_y \rrbracket$.
We set $\Ss' = \Ss \cup \llbracket d_y+1,m \rrbracket = \llbracket 1,r \rrbracket \cup \llbracket d_y+1,m \rrbracket $ such that $\Ss' \cup Q = \llbracket 1,m \rrbracket$. \\
Reordering the terms and, since $V$ is orthogonal, using $I_m = VV^T = V_{\Ss'}V_{\Ss'}^T + V_{Q}V_{Q}^T$ , we have
\begin{align*}
    FT_1 &= \left\| \left(U_{\Ss}^T (W'_H)_{.,1:r} + \sum_{i \in J_1 \cup J_2 \cup J_3} (W'_i)_{1:r,1:r}\right) U_{\Ss}^T \Sigma_{YX}\Sigma_{XX}^{-1}X \right. \\
    & \quad + \sum_{i \in J_2} (W'_i)_{1:r,r+1:d_{i-1}} Z_{i-1} \cdots Z_2 Z_1 X \\
    & \quad +\left. \left(\sum_{i \in J_3}  (W'_i)_{1:r,r+1:d_{i-1}} Z_{i-1} \cdots Z_2 Z_1 X + (W'_1)_{1:r,.} X\right) \left(V_{\Ss'}V_{\Ss'}^T+V_Q V_Q^T\right) \right\|^2 \;.
\end{align*}
Since for $i \in J_2$, we have $i-1 \geq q$, we denote
\begin{align*}
    N &:= \sum_{i \in J_2} (W'_i)_{1:r,r+1:d_{i-1}} Z_{i-1} \cdots Z_{q+1} \;,
\end{align*}
Recall that, using the convention in Section \ref{Settings}, for $i-1=q$,
we have $Z_{i-1} \cdots Z_{q+1} = I_{d_q-r}$.\\
We also denote
\begin{align*}
    M &:= \sum_{i \in J_3}  (W'_i)_{1:r,r+1:d_{i-1}} Z_{i-1} \cdots Z_2 Z_1 X V_Q + (W'_1)_{1:r,.}X V_Q \;, \\
    J &:= \sum_{i \in J_3}  (W'_i)_{1:r,r+1:d_{i-1}} Z_{i-1} \cdots Z_2 Z_1 X V_{\Ss'} + (W'_1)_{1:r,.}X V_{\Ss'}  \;,\\
    L &:= U_{\Ss}^T (W'_H)_{.,1:r} + \sum_{i \in J_1 \cup J_2 \cup J_3} (W'_i)_{1:r,1:r} \;.
\end{align*}
Therefore, we obtain 
\begin{align*}
    FT_1 &= \left\|   L U_{\Ss}^T \Sigma_{YX}\Sigma_{XX}^{-1}X + N Z_q \cdots Z_2 Z_1 X + J V_{\Ss'}^T + M V_{Q}^T \right\|^2 \\
    &= \left\|   L U_{\Ss}^T \Sigma_{YX}\Sigma_{XX}^{-1}X + N Z_q \cdots Z_2 Z_1 X + J V_{\Ss'}^T \right\|^2 + \left\| M V_{Q}^T \right\|^2 \\
    & \quad + 2 \left\langle L U_{\Ss}^T \Sigma_{YX}\Sigma_{XX}^{-1}X + N Z_q \cdots Z_2 Z_1 X +  J V_{\Ss'}^T \ , \ M V_{Q}^T  \right\rangle \;. \\
\end{align*}
Using Lemma \ref{C} and Lemma \ref{D} and $V_Q^T V_{\Ss'} = 0$ (since $V$ is orthogonal), the cross-product is equal to zero. \\
Noting also that since $V$ is orthogonal $\|M V_Q^T\|^2 = \tr(M V_Q^T V_Q M^T) = \tr(M M^T) = \|M\|^2 = \|M^T\|^2$, we have
\begin{align*}
    FT_1 &= \left\|   L U_{\Ss}^T \Sigma_{YX}\Sigma_{XX}^{-1}X + N Z_q \cdots Z_2 Z_1 X + J V_{\Ss'}^T \right\|^2 + \left\| M^T \right\|^2 \\
    &= \|A_2\|^2 + \|A_4\|^2 
\end{align*}
where
\begin{align}
    A_2 &:= L U_{\Ss}^T \Sigma_{YX}\Sigma_{XX}^{-1}X + N Z_q \cdots Z_2 Z_1 X + J V_{\Ss'}^T \nonumber \\
    & = U_{\Ss}^T(W'_H)_{.,1:r} U_{\Ss}^T \Sigma_{YX}\Sigma_{XX}^{-1}X + \sum_{i \in J_1}  (W'_i)_{1:r,1:r} U_{\Ss}^T \Sigma_{YX}\Sigma_{XX}^{-1}X \nonumber \\
    &\quad + \sum_{i \in J_2}  \left((W'_i)_{1:r,1:r} U_{\Ss}^T \Sigma_{YX}\Sigma_{XX}^{-1}X + (W'_i)_{1:r,r+1:d_{i-1}} Z_{i-1} \cdots Z_2 Z_1 X \right) \nonumber \\
    &\quad + \sum_{i \in J_3} \left( (W'_i)_{1:r,1:r} U_{\Ss}^T \Sigma_{YX}\Sigma_{XX}^{-1}X + (W'_i)_{1:r,r+1:d_{i-1}} Z_{i-1} \cdots Z_2 Z_1 X V_{\Ss'} V_{\Ss'}^T \right) 
    + (W'_1)_{1:r,.} X V_{\Ss'} V_{\Ss'}^T \label{A_2} \\
    A_4 &:= M^T = \left(\sum_{i \in J_3}  (W'_i)_{1:r,r+1:d_{i-1}} Z_{i-1} \cdots Z_2 Z_1 X V_{Q} + (W'_1)_{1:r,.} X V_{Q} \right)^T \;. \label{A_4}
\end{align}
Let us now simplify $FT_2$.\\
We have $FT_2 = \left\| A U_{\Ss}^T \Sigma_{YX}\Sigma_{XX}^{-1}X \right\|^2$, with $$A := U_Q^T(W'_H)_{.,1:r} + \sum_{i \in J_1} Z_H Z_{H-1} \cdots Z_{i+1} (W'_i)_{r+1:d_i,1:r} \in \mathbb{R}^{(d_y - r) \times r} \;.$$
Hence, using Lemma \ref{decomposition A delta s carre}, we have
\begin{align*}
    FT_2  &= \sum_{a=1}^r \sum_{b=r+1}^{d_y} (\lambda_a - \lambda_b) ( A_{b-r,a})^2 + \| \Delta^{(Q)} A \|^2 \\
    &= \sum_{a=1}^r \sum_{b=r+1}^{d_y} (\lambda_a - \lambda_b) \left(U_b^T(W'_H)_{.,a} + \sum_{i \in J_1} (Z_H)_{b-r,.} Z_{H-1} \cdots Z_{i+1} (W'_i)_{r+1:d_i,a}\right)^2 \\
    & \quad + \left\| \Delta^{(Q)} \left(U_Q^T(W'_H)_{.,1:r} + \sum_{i \in J_1} Z_H Z_{H-1} \cdots Z_{i+1} (W'_i)_{r+1:d_i,1:r}\right) \right\|^2 \\
    &= a_1 + \|A_3\|^2 \;,
\end{align*}
where
\begin{align}
    a_1 &:= \sum_{a=1}^r \sum_{b=r+1}^{d_y} (\lambda_a - \lambda_b) \left(U_b^T(W'_H)_{.,a} + \sum_{i \in J_1} (Z_H)_{b-r,.} Z_{H-1} \cdots Z_{i+1} (W'_i)_{r+1:d_i,a}\right)^2 \label{a_1} \\
    A_3 &:= \Delta^{(Q)} \left(U_Q^T(W'_H)_{.,1:r} + \sum_{i \in J_1} Z_H Z_{H-1} \cdots Z_{i+1} (W'_i)_{r+1:d_i,1:r}\right) \label{A_3} 
\end{align}
Finally,
\begin{align*}
    FT &= FT_1 + FT_2 \\
    &= a_1 + \|A_2\|^2 + \|A_3\|^2 + \|A_4\|^2 \;, \numberthis \label{FT final}
\end{align*}
where $a_1, A_2 , A_3 , A_4$ are defined in \eqref{a_1}, \eqref{A_2}, \eqref{A_3}, \eqref{A_4}.
Notice that, since $\lambda_1 > \cdots > \lambda_{d_y}$, we have 
\begin{align}\label{a_1 geq 0}
    a_1 \geq 0 \;.
\end{align}

\subsubsection{Simplifying \texorpdfstring{$ST$}{}}\label{section simplifying ST}
In this section, we prove that $ST = -2 \left\langle A_3 , A_4 \right\rangle $, where $ST$, $A_3$ and $A_4$ are defined in \eqref{ST}, \eqref{A_3} and \eqref{A_4}.
In order to do so, we first state a lemma that simplifies the terms $T_{i,j}$ defined in \eqref{T_ij}.
We remind that the sets $J_1$, $J_2$ and $J_3$ are defined at the beginning of Section \ref{ section simplifying FT}. 
\begin{lemma} \label{simplif all T_ij}
    Suppose Assumption \ref{Assump H} in Section \ref{Settings} holds true.
    Let $\Wbf = (W_H, \cdots , W_1)$ be a first-order critical point satisfying the hypotheses of Proposition \ref{prop 4}, and $r$, $\Ss$, $Q$, $(Z_h)_{h=1..H}$ defined as in Proposition \ref{prop 4}.
    Let $(i,j) \in \llbracket 1,H \rrbracket^2$, with $i>j$.
    For any $\Wbf' = (W'_H, \ldots, W'_1)$, recall that 
    , as defined in \eqref{T_ij},
    $$T_{i,j} = \left\langle  W_H \cdots W_{i+1} W'_i W_{i-1} \cdots W_{j+1} W'_j W_{j-1} \cdots W_1 X  \ ,\ W_H \cdots W_1 X - Y \right\rangle \;.$$
    If $\Wbf$ is tightened, then, for $p$ and $q$ as defined in Lemma \ref{E} and $J_1$,$J_2$,$J_3$ as defined above, we have
    \begin{itemize}
        \item For $i=H$:
        \begin{itemize}
            \item For $j \in J_3$:
            \begin{align}
                T_{H,j} &= - \left\langle  \Delta^{(Q)} U_Q^T (W'_H)_{.,1:r}\ ,\ \left((W'_j)_{1:r,r+1:d_{j-1}} Z_{j-1} \cdots Z_2 Z_1 X V_Q\right)^T \right\rangle \;. \label{T_Hj simplified}
            \end{align}
            \item For $j=1$:
            \begin{align}
                T_{H,1} &= -\left\langle  \Delta^{(Q)} U_Q^T (W'_H)_{.,1:r}  \ , \ \left((W'_1)_{1:r,.} X V_Q \right)^T \right\rangle \;. \label{T_H1 simplified}
            \end{align}
            \item For $j \in J_1 \cup J_2$:
            \begin{align}
                T_{H,j} &= 0 \;. \label{T_Hj j in J_1 cup J_2}
            \end{align}
        \end{itemize}
        \item For $i \in J_1$:
        \begin{itemize}
            \item For $j \in J_3$:
            \begin{align}
                T_{i,j} &= - \left\langle  \Delta^{(Q)} Z_H Z_{H-1} \cdots Z_{i+1} (W'_{i})_{r+1:d_{i},1:r} \  , \  \left((W'_j)_{1:r,r+1:d_{j-1}} Z_{j-1} \cdots Z_2 Z_1 X V_Q\right)^T \right\rangle \;. \label{T_ij J_1 x J_3 simplified}
            \end{align}
            \item For $j=1$:
            \begin{align}
                T_{i,1} &= - \left\langle \Delta^{(Q)} Z_H Z_{H-1} \cdots Z_{i+1} (W'_{i})_{r+1:d_{i},1:r} \ , \ ((W'_1)_{1:r,.} X V_Q)^T \right\rangle \;. \label{T_i1 simplified}
            \end{align}
            \item For $j \in J_1 \cup J_2$:
            \begin{align}
                T_{i,j} &= 0 \;. \label{T_ij i in J_1, j in J_1 cup J_2}
            \end{align}
        \end{itemize}
        \item For $i \in J_2 \cup J_3$, for all $j<i$, we have 
        \begin{align}
            T_{i,j} = 0 \;. \label{T_ij, i in J_2 cup J_3}
        \end{align}
    \end{itemize}
\end{lemma}
The proof of Lemma \ref{simplif all T_ij} is in Appendix \ref{proof simplif all T_ij}.\\
Let us now prove that $ST = -2 \left\langle  A_3 , A_4 \right\rangle$.
We remind that $\llbracket 1,H \rrbracket = \{H\} \cup J_1 \cup J_2 \cup J_3 \cup \{1\}$ and separate the sum appearing in \eqref{ST} accordingly. \\
We then substitute the formulas of Lemma \ref{simplif all T_ij} in \eqref{ST} and obtain
\begin{align*}
    ST &= 2 \sum_{H \geq i>j \geq 1} T_{i,j} \\
    &= 2 \left(\sum_{j \in J_1 \cup J_2} T_{H,j} + \sum_{j \in J_3} T_{H,j} + T_{H,1}
    + \sum_{i \in J_1} \sum_{\substack{j \in J_1 \cup J_2,\\ j<i}} T_{i,j}
    + \sum_{i \in J_1} \sum_{j \in J_3} T_{i,j} + \sum_{i \in J_1} T_{i,1} + \sum_{i \in J_2 \cup J_3} \sum_{j=1}^{i-1} T_{i,j} \right) \\
    &= - 2\sum_{j \in J_3} \left\langle  \Delta^{(Q)} U_Q^T (W'_H)_{.,1:r}\ ,\ \left((W'_j)_{1:r,r+1:d_{j-1}} Z_{j-1} \cdots Z_2 Z_1 X V_Q\right)^T \right\rangle \\
    & \quad -2 \left\langle  \Delta^{(Q)} U_Q^T (W'_H)_{.,1:r}  \ , \ ((W'_1)_{1:r,.} X V_Q )^T \right\rangle \\
    & \quad - 2 \sum_{i \in J_1} \sum_{j \in J_3} \left\langle  \Delta^{(Q)} Z_H Z_{H-1} \cdots Z_{i+1} (W'_{i})_{r+1:d_{i},1:r} \ , \ \left((W'_j)_{1:r,r+1:d_{j-1}}Z_{j-1} \cdots Z_2 Z_1 X V_Q\right)^T \right\rangle \\
    & \quad - 2\sum_{i \in J_1} \left\langle \Delta^{(Q)}Z_H Z_{H-1} \cdots Z_{i+1} (W'_{i})_{r+1:d_{i},1:r} \ , \ ((W'_1)_{1:r,.} X V_Q)^T \right\rangle \\
    & = -2 \left\langle \Delta^{(Q)} U_Q^T (W'_H)_{.,1:r} \ , \ \left(\sum_{j \in J_3} \ (W'_j)_{1:r,r+1:d_{j-1}}Z_{j-1} \cdots Z_2 Z_1 X V_Q + (W'_1)_{1:r,.} X V_Q\right)^T \right\rangle \\
    & \quad -2 \left\langle \Delta^{(Q)} \sum_{i \in J_1} Z_H Z_{H-1} \cdots Z_{i+1} (W'_{i})_{r+1:d_{i},1:r} \  , \right.\\
    & \qquad\qquad
     \left.\ \left(\sum_{j \in J_3} \ (W'_j)_{1:r,r+1:d_{j-1}}Z_{j-1} \cdots Z_2 Z_1 X V_Q + (W'_1)_{1:r,.} X V_Q\right)^T \right\rangle \\
    &= - 2\left\langle  \Delta^{(Q)} \left( U_Q^T (W'_H)_{.,1:r} + \sum_{i \in J_1} Z_H Z_{H-1} \cdots Z_{i+1} (W'_{i})_{r+1:d_{i},1:r}\right) \right. \ , \\
    & \qquad \qquad \left. \left(\sum_{j \in J_3} \ (W'_j)_{1:r,r+1:d_{j-1}}Z_{j-1} \cdots Z_2 Z_1 X V_Q + (W'_1)_{1:r,.} X V_Q\right)^T \right\rangle \\
    &= - 2 \left\langle  A_3 , A_4 \right\rangle \;, \numberthis \label{ST final}
\end{align*}
where we remind that $A_3$ and $A_4$ are defined in \eqref{A_3} and \eqref{A_4}.

\subsubsection{Concluding the Proof of Proposition \ref{prop 4}}\label{section concluding}
Using the simplifications  \eqref{FT final} and \eqref{ST final} above, for any $\Wbf$ satisfying the hypotheses of Proposition \ref{prop 4}, if $\Wbf$ is tightened, then for any $\Wbf'$,
\begin{align*}
    c_2(\Wbf , \Wbf') &= FT + ST \\
    &= a_1 + \|A_2\|^2 +\|A_3\|^2 + \|A_4\|^2 - 2 \left\langle  A_3, A_4 \right\rangle \\
    &= a_1 + \|A_2\|^2 +\|A_3 - A_4\|^2 \;.
\end{align*}
Using \eqref{a_1 geq 0}, we find 
$c_2(\Wbf , \Wbf') \geq 0$.\\
Therefore, $\Wbf =(W_H, \ldots ,W_1)$ is a second-order critical point.

\subsubsection{Proof of Lemma \ref{E}}\label{proof E}
First note that, for $r=0$, we can easily follow the same proof and see that the result still holds with the conventions adopted in Section \ref{Settings}.\\
\textbf{Let us prove \eqref{W_H...W_p simplified}}.\\
Consider the pivot $(i,j) = (2,1)$.
Its complementary blocks are $\Sigma_{XY} W_H \cdots W_3$ and $I_{d_1}$.
Since $\Wbf$ is tightened and $\rk(I_{d_1}) = d_1 \geq r_{max} > r$, we have $\rk(\Sigma_{XY} W_H \cdots W_3)=r$.
Since $\Sigma_{XY}$ is full-column rank, we obtain $\rk( W_H \cdots W_3)=r$. \\
Let $p \in \llbracket 3,H \rrbracket$ be the largest index such that 
\begin{align}\label{def p}
    \textit{rk}(W_H \cdots W_p)=r \;.
\end{align}
Using \eqref{W_H simplified} and \eqref{W_h simplified}, we have $W_H \cdots W_p = \left[ U_{\Ss}\ , \ U_Q Z_H Z_{H-1} \cdots Z_p \right]$.\\
Since $\textit{rk}(W_H \cdots W_p)=r$ and since the columns of $ U_Q Z_H Z_{H-1} \cdots Z_p $ are in the vector space spanned by the columns of $U_Q$ (which are orthogonal to the columns of $U_{\Ss})$, \eqref{def p} implies
$$ Z_H Z_{H-1} \cdots Z_p = 0 \;.$$ 
Therefore,
\begin{align*}
    W_H \cdots W_p = \left[ U_{\Ss}\ , \ 0 \right]\;. 
\end{align*}
Using \eqref{W_h simplified}, for all $i \in \llbracket 1,p-1 \rrbracket$,
\begin{align*}
    W_H \cdots W_{i+1} &= (W_H \cdots W_{p}) (W_{p-1} \cdots W_{i+1}) \\
    &= \left[ U_{\Ss}\ , \ 0 \right] \begin{bmatrix}
    I_r & 0 \\
    0 & Z_{p-1}\cdots Z_{i+1}
    \end{bmatrix} \\
    &= \left[ U_{\Ss}\ , \ 0 \right] \;.
\end{align*}
This proves \eqref{W_H...W_p simplified}. \\
\textbf{Let us prove \eqref{W_i-1...W_2 simplified}}.\\
We consider the pivot $(p,1)$.
Its complementary blocks are $\Sigma_{XY} W_H \cdots W_{p+1}$ and $W_{p-1} \cdots W_2$.
We have, by definition of $p$, $\rk(W_H \cdots W_{p+1})>r $.
Therefore, since $\Sigma_{XY}$ is full-column rank, we have $\rk(\Sigma_{XY} W_H \cdots W_{p+1}) = \rk(W_H \cdots W_{p+1})>r$.
Note that this holds both for $p=H$ and for $p<H$.
Hence, since $\Wbf$ is tightened, the second complementary block is of rank $r$, i.e. 
\begin{align*}
    \rk(W_{p-1} \cdots W_2)=r \;.
\end{align*}
 
Using \eqref{W_h simplified}, we also have $W_{p-1} \cdots W_2 = \left[
\begin{array}{c  c}
I_r & 0 \\
0 & Z_{p-1} \cdots Z_2
\end{array}
\right]$.\\
Then, since $\rk(W_{p-1} \cdots W_2)=r$, we have $Z_{p-1} \cdots Z_2 = 0$ and $$W_{p-1} \cdots W_2 = \left[
\begin{array}{c  c}
I_r & 0 \\
0 & 0
\end{array}
\right].$$\\
Using \eqref{W_h simplified} again, for all $i \in \llbracket p,H \rrbracket$,
\begin{align*}
    W_{i-1} \cdots W_2 &= (W_{i-1} \cdots W_p)(W_{p-1} \cdots W_2) \\
    &= \begin{bmatrix}
    I_r & 0 \\
    0 & Z_{i-1} \cdots Z_p 
    \end{bmatrix}
    \begin{bmatrix}
    I_r & 0 \\
    0 & 0 
    \end{bmatrix} \\
    &= 
    \left[
    \begin{array}{c  c}
    I_r & 0 \\
    0 & 0
    \end{array}
    \right] \;.
\end{align*}
This proves \eqref{W_i-1...W_2 simplified}. \\
\textbf{Let us now prove \eqref{hammou}}.\\
Using Proposition \ref{global map and critical values}, Lemma \ref{sigma invertible} and Lemma \ref{properties of orhtogonality}, we have
\begin{align*}
    \rk(W_{p-1} \cdots W_1 \Sigma_{XY}) \geq \rk(W_H \cdots W_1 \Sigma_{XY}) = \rk( U_{\Ss}U_{\Ss}^T \Sigma) \geq \rk( U_{\Ss}^T (U_{\Ss}U_{\Ss}^T \Sigma) \Sigma^{-1} U_{\Ss}) =  \rk( I_r) = r \;.
\end{align*}
Using 
\eqref{W_i-1...W_2 simplified} for $i=p$, we also have $ \rk(W_{p-1} \cdots W_1 \Sigma_{XY}) \leq \rk(W_{p-1} \cdots W_2) = r$.
Hence, $\rk(W_{p-1} \cdots W_1 \Sigma_{XY}) = r.$ \\
Notice that, considering the tightened pivot $(H,H-1)$, since $\rk(I_{d_{H-1}}) = d_{H-1} \geq r_{max} > r$, we obtain $\rk(W_{H-2} \cdots W_1 \Sigma_{XY}) = r$. \\
We consider $q \in \llbracket 1, \min(p-1,H-2) \rrbracket$ the smallest index such that $\textit{rk}(W_q \cdots W_1\Sigma_{XY})=r$.\\
Using \eqref{W_h simplified} and \eqref{W_1 simplified}, we have
\begin{align*}
    W_q \cdots W_1\Sigma_{XY} &=  \left[
    \begin{array}{c}
    U_{\Ss}^T\Sigma \\
    Z_q \cdots Z_2 Z_1 \Sigma_{XY}
    \end{array}
    \right] \\
    &= \left[
    \begin{array}{c}
    \lambda_1 U_1^T \\
    \vdots \\
    \lambda_r U_r^T \\
    Z_q \cdots Z_2 Z_1 \Sigma_{XY}
    \end{array}
    \right] \;.
\end{align*}
Since $\rk(W_q \cdots W_1\Sigma_{XY})=r$, every row of $Z_q \cdots Z_2 Z_1 \Sigma_{XY}$ lies in $\text{Vec}(U_1^T,\ldots,U_r^T)$, hence we have
\begin{align*}
    Z_q \cdots Z_2 Z_1 \Sigma_{XY} U_Q =0 \;.
\end{align*}
Finally, we conclude that, for all $i \in \llbracket q+1 , H \rrbracket$,
\begin{align*}
    Z_{i-1} \cdots Z_2 Z_1 \Sigma_{XY} U_Q &= Z_{i-1} \cdots Z_{q+1} Z_{q} \cdots Z_2 Z_1 \Sigma_{XY} U_Q \\
    &= Z_{i-1} \cdots Z_{q+1} 0 \\
    &= 0 \;.
\end{align*}
This proves \eqref{hammou}. \\
\textbf{Let us now prove \eqref{W_H-1...W_k+1 simplified}}.\\
Consider the pivot $(H,q)$.
Its complementary blocks are $W_{q-1} \cdots W_1 \Sigma_{XY}$ and $W_{H-1} \cdots W_{q+1}$.
We have, by definition of $q$, $\rk(W_{q-1} \cdots W_1 \Sigma_{XY})>r$.
Hence, since $\Wbf$ is tightened, the other complementary block is of rank $r$, i.e. $\rk(W_{H-1} \cdots W_{q+1}) = r$.
Using \eqref{W_h simplified}, we have
$$W_{H-1} \cdots W_{q+1} = \left[
\begin{array}{c  c}
I_r & 0 \\
0 & Z_{H-1} \cdots Z_{q+1}
\end{array}
\right].$$
Therefore, $Z_{H-1} \cdots Z_{q+1} = 0$ and $$W_{H-1} \cdots W_{q+1} = \left[
\begin{array}{c  c}
I_r & 0 \\
0 & 0
\end{array}
\right].$$
Finally, using \eqref{W_h simplified}, for all $i \in \llbracket 1,q \rrbracket$,
\begin{align*}
    W_{H-1} \cdots W_{i+1} &= W_{H-1} \cdots W_{q+1} W_{q} \cdots W_{i+1} \\
    &= \begin{bmatrix}
    I_r & 0 \\
    0 & 0
    \end{bmatrix}
    \begin{bmatrix}
    I_r & 0 \\
    0 & Z_q \cdots Z_{i+1}
    \end{bmatrix} \\
    &=
    \left[
    \begin{array}{c  c}
    I_r & 0 \\
    0 & 0
    \end{array}
    \right] \;.
\end{align*}
This proves \eqref{W_H-1...W_k+1 simplified} and concludes the proof. \\

\subsubsection{Proof of Lemma \ref{Sigma_XY U_Q}} \label{proof Sigma_XY U_Q}
Recall that $\Sigma^{1/2} = \Sigma_{YX} \Sigma_{XX}^{-1} X$.
We have
\begin{align*}
    \Sigma_{XY} &= XY^T \\
    &= XX^T(XX^T)^{-1}XY^T \\
    &= X (\Sigma^{1/2})^T \;.
\end{align*}
Using \eqref{svd de sigma 1/2}, we obtain $$\Sigma_{XY} = X V \Delta^T U^T , $$
and, since $U$ is orthogonal, we have $$\Sigma_{XY}U =X V \Delta^T.$$ Restricting the equality to the columns in $Q$, we obtain $$\Sigma_{XY}U_Q =X V_Q \Delta^{(Q)} \;,$$
where $\Delta^{(Q)}$ is defined in \eqref{delta Q}.
This concludes the proof.

\subsubsection{Proof of Lemma \ref{the formula}} \label{proof the formula}
Let $A \in \mathbb{R}^{d_y \times n}$ and $B \in \mathbb{R}^{r \times n}$.
We have
\begin{align*}
    \| A + U_{\Ss}B \|^2 &= \|A\|^2 + \|U_{\Ss}B\|^2 + 2 \left\langle  A \ , \ U_{\Ss}B \right\rangle \\
    &= tr(A^TA) + tr(B^T U_{\Ss}^T U_{\Ss}B) + 2 \left\langle  U_{\Ss}^T A \ , \ B \right\rangle \;.
\end{align*}
Using Lemma \ref{properties of orhtogonality}, this becomes
\begin{align*}
    \| A + U_{\Ss}B \|^2 &= tr\left(A^T(U_{\Ss} U_{\Ss}^T + U_Q U_Q^T)A\right) + tr\left(B^T B\right) + 2 \left\langle  U_{\Ss}^T A \ , \ B \right\rangle \\
    &= tr(A^T U_Q U_Q^T A) + tr(A^TU_{\Ss} U_{\Ss}^TA) + tr(B^T B) + 2 \left\langle  U_{\Ss}^T A \ , \ B \right\rangle \\
    &= \|U_Q^T A \|^2 + \|U_{\Ss}^T A \|^2 + \|B\|^2 + 2 \left\langle  U_{\Ss}^T A \ , \ B \right\rangle \\
    &= \|U_Q^T A \|^2 + \| U_{\Ss}^T A + B\|^2 \;.
\end{align*}

\subsubsection{Proof of Lemma \ref{C}} \label{proof C}
Recall that $\Sigma^{1/2} = \Sigma_{YX}\Sigma_{XX}^{-1}X$ has a Singular Value Decomposition $\Sigma^{1/2} = U \Delta V^T $ (see \eqref{svd de sigma 1/2}).
Hence, we have $\Sigma^{1/2}V = U \Delta $ and therefore $\Sigma^{1/2}V_Q = U_Q \Delta^{(Q)} $, where $\Delta^{(Q)}$ is defined in \eqref{delta Q}. \\
As a consequence,
\begin{align*}
    U_{\Ss}^T \Sigma_{YX}\Sigma_{XX}^{-1}X V_Q &= U_{\Ss}^T \Sigma^{1/2}V_Q \\
    &= U_{\Ss}^T U_Q \Delta^{(Q)} \\
    &= 0 \;,
\end{align*}
where the last equality follows from Lemma \ref{properties of orhtogonality}.
Finally, we obtain for any $A \in \mathbb{R}^{n \times r}$, $B \in \mathbb{R}^{n \times (d_y-r)}$
\begin{align*}
    \left\langle  A U_{\Ss}^T \Sigma_{YX}\Sigma_{XX}^{-1}X \ , \ BV_Q^T\right\rangle&= tr(A U_{\Ss}^T \Sigma_{YX}\Sigma_{XX}^{-1}X V_Q B^T) \\
    &= 0 \;.
\end{align*}

\subsubsection{Proof of Lemma \ref{D}} \label{proof D}
Using Lemma \ref{Sigma_XY U_Q}, we have $\Sigma_{XY}U_Q =X V_Q \Delta^{(Q)}$, then 
replacing this formula in \eqref{hammou} with $i=q+1$, we have  $$Z_q \cdots Z_2 Z_1 X V_Q \Delta^{(Q)} =0 \;.$$
Since $\Delta^{(Q)}$ is diagonal and its diagonal elements are non-zero, it is invertible, hence
$$Z_q \cdots Z_2 Z_1 X V_Q =0 \;. $$
Finally, for any matrices $A \in \mathbb{R}^{n \times (d_q - r)}$ and $ B \in \mathbb{R}^{n \times (d_y - r)} $, we have
\begin{align*}
    \left\langle  AZ_q \cdots Z_2 Z_1 X \ ,\ B V_Q^T  \right\rangle &= tr(AZ_q \cdots Z_2 Z_1 X V_Q B^T)  \\
    &= 0 \;.
\end{align*}

\subsubsection{Proof of Lemma \ref{decomposition A delta s carre}} \label{proof decomposition A delta s carre}
Recall that $\Delta^{(\Ss)}$ is defined in \eqref{delta S} and $\Sigma = U \Lambda U^T$.
Let $ A \in \mathbb{R}^{(d_y - r) \times r} $, we have
\begin{align*}
    \left\| A U_{\Ss}^T \Sigma_{YX}\Sigma_{XX}^{-1}X \right\|^2 
    &= tr(A U_{\Ss}^T \Sigma U_{\Ss} A^T) \\
    &= tr(A \ \text{diag}(\lambda_1, \ldots ,\lambda_r)  A^T) \\
    &= \left\| A \Delta^{(\Ss)} \right\|^2 \\
    &= \sum_{a=1}^r \sum_{b=r+1}^{d_y} \lambda_a (A_{b-r,a})^2 \\
    &= \sum_{a=1}^r \sum_{b=r+1}^{d_y} (\lambda_a - \lambda_b) ( A_{b-r,a})^2 + \sum_{a=1}^r \sum_{b=r+1}^{d_y} \lambda_b ( A_{b-r,a})^2 \\
    &= \sum_{a=1}^r \sum_{b=r+1}^{d_y} (\lambda_a - \lambda_b) ( A_{b-r,a})^2 + \| \Delta^{(Q)} A \|^2 \;.
\end{align*}

\subsubsection{Proof of Lemma \ref{simplif cross-product}}\label{proof simplif cross-product}
Let $\Wbf = (W_H, \ldots , W_1)$ be a first-order critical point associated with $\Ss$ verifying the hypotheses of Proposition \ref{prop 4} and let $A \in \mathbb{R}^{d_y \times d_x}$. Using \eqref{W_H...W_2 simplified}, \eqref{W_1 simplified}, and Lemma \ref{properties of orhtogonality}, we have
\begin{align*}
    \left\langle  AX \ ,\ W_H \cdots W_1X - Y \right\rangle &= \left\langle  A \ , \ W_H \cdots W_1 XX^T - YX^T \right\rangle \\
    &= \left\langle  A \ , \ U_{\Ss} U_{\Ss}^T \Sigma_{YX}\Sigma_{XX}^{-1}XX^T - \Sigma_{YX} \right\rangle \\
    &= \left\langle  A \ , \ U_{\Ss} U_{\Ss}^T \Sigma_{YX} - \Sigma_{YX} \right\rangle \\
    &= \left\langle  A \ , \ -U_{Q} U_{Q}^T \Sigma_{YX} \right\rangle \;.
\end{align*}

\subsubsection{Proof of Lemma \ref{simplif all T_i}}\label{proof simplif all T_i}
Let $\Wbf = (W_H, \cdots , W_1)$ be a tightened first-order critical point satisfying the hypotheses of Proposition \ref{prop 4}, and $r$, $\Ss$, $Q$, $(Z_h)_{h=1..H}$ defined as in Proposition \ref{prop 4}.
Since $\Wbf$ satisfies the hypotheses of Proposition \ref{prop 4}, we are going to use all the equations \eqref{W_H simplified}, \eqref{W_1 simplified}, \eqref{W_h simplified} and \eqref{W_H...W_2 simplified} defined by these hypotheses and \eqref{W_H...W_p simplified}, \eqref{W_i-1...W_2 simplified}, \eqref{hammou} and \eqref{W_H-1...W_k+1 simplified} of Lemma \ref{E}.\\
Let $\Wbf' = (W'_H, \ldots, W'_1)$ and $i \in \llbracket 1,H \rrbracket$.
Recall that $T_{i}$ is defined in \eqref{T_i} and $J_1 = \llbracket p,H-1 \rrbracket$, $J_2 = \llbracket q+1 , p-1 \rrbracket$, $J_3 = \llbracket 2,q \rrbracket$, where $p$ and $q$ are defined as in Lemma \ref{E}. \\

 \textbf{Consider the case $i=H$.} \\
Substituting \eqref{W_i-1...W_2 simplified} and \eqref{W_1 simplified} in \eqref{T_i}, we have
\begin{align*}
    T_H &=  W'_H(W_{H-1} \cdots W_2)W_1 X  \\
    &= W'_H \left[
    \begin{array}{c  c}
    I_r & 0 \\
    0 & 0
    \end{array}
    \right] \left[
    \begin{array}{c}
    U_{\Ss}^T\Sigma_{YX}\Sigma_{XX}^{-1} \\
    Z_1
    \end{array}
    \right]X \\
    &= W'_H\left[
    \begin{array}{c}
    U_{\Ss}^T\Sigma_{YX}\Sigma_{XX}^{-1} \\
    0
    \end{array}
    \right]X \\
    &= (W'_H)_{.,1:r} U_{\Ss}^T \Sigma_{YX}\Sigma_{XX}^{-1}X \;.
\end{align*}
This proves \eqref{T_H simplified}. \\

 \textbf{Consider now the case $i \in J_1$.} \\
Substituting \eqref{W_H simplified}, \eqref{W_h simplified}, \eqref{W_i-1...W_2 simplified} and \eqref{W_1 simplified}, in \eqref{T_i}, we have, for $i \in J_1$
\begin{align*}
    T_i &= W_H(W_{H-1} \cdots W_{i+1})W'_i(W_{i-1} \cdots W_2)W_1X \\
    &= \left[U_{\Ss} \ ,\ U_Q Z_H\right] \left[
    \begin{array}{c  c}
    I_r & 0 \\
    0 & Z_{H-1}
    \end{array}
    \right] \cdots  \left[
    \begin{array}{c  c}
    I_r & 0 \\
    0 & Z_{i+1}
    \end{array}
    \right] W'_i \left[
    \begin{array}{c  c}
    I_r & 0 \\
    0 & 0
    \end{array}
    \right] \left[
    \begin{array}{c}
    U_{\Ss}^T\Sigma_{YX}\Sigma_{XX}^{-1} \\
    Z_1
    \end{array}
    \right] X \\
    &= \left[U_{\Ss}\ ,\ U_Q Z_H Z_{H-1} \cdots Z_{i+1}\right] W'_i \left[
    \begin{array}{c}
    U_{\Ss}^T\Sigma_{YX}\Sigma_{XX}^{-1} \\
    0
    \end{array}
    \right] X \\
    &= \left[U_{\Ss}\ ,\ U_Q Z_H Z_{H-1} \cdots Z_{i+1}\right](W'_i)_{.,1:r}U_{\Ss}^T\Sigma_{YX}\Sigma_{XX}^{-1} X \\
    &=  U_{\Ss} (W'_i)_{1:r,1:r} U_{\Ss}^T \Sigma_{YX}\Sigma_{XX}^{-1}X + U_Q Z_H Z_{H-1} \cdots Z_{i+1} (W'_i)_{r+1:d_i,1:r} U_{\Ss}^T \Sigma_{YX}\Sigma_{XX}^{-1}X \;.
\end{align*}
Note that the above calculations are still valid in the case $i=H-1$.
In this case using the convention in Section \ref{Settings},
$W_{H-1} \cdots W_{i+1} = I_{d_{H-1}}$ and $Z_{H-1} \cdots Z_{i+1} = I_{d_{H-1}-r}$.\\
This proves \eqref{T_i J_1 simplified}. \\

 \textbf{Consider now the case $i \in J_2 \cup J_3 = \llbracket 2,p-1 \rrbracket$.} \\
Substituting \eqref{W_H...W_p simplified}, \eqref{W_h simplified} and \eqref{W_1 simplified}, in \eqref{T_i}, we have, for $i \in J_2 \cup J_3$,
\begin{align*}
    T_i &= (W_H \cdots W_{i+1})W'_i(W_{i-1} \cdots W_2)W_1X \\
    &= \biggl[U_{\Ss} \ ,\ 0 \biggr] W'_i \left[
    \begin{array}{c  c}
    I_r & 0 \\
    0 & Z_{i-1}
    \end{array}
    \right] \cdots \left[
    \begin{array}{c  c}
    I_r & 0 \\
    0 & Z_2
    \end{array}
    \right] \left[
    \begin{array}{c}
    U_{\Ss}^T\Sigma_{YX}\Sigma_{XX}^{-1} \\
    Z_1
    \end{array}
    \right] X \\
    &=  U_{\Ss}(W'_i)_{1:r,.} \left[
    \begin{array}{c}
    U_{\Ss}^T\Sigma_{YX}\Sigma_{XX}^{-1} \\
    Z_{i-1} \cdots Z_2 Z_1
    \end{array}
    \right] X \\
    &=  U_{\Ss} (W'_i)_{1:r,1:r} U_{\Ss}^T \Sigma_{YX}\Sigma_{XX}^{-1}X + U_{\Ss} (W'_i)_{1:r,r+1:d_{i-1}} Z_{i-1} \cdots Z_2Z_1 X \;.
\end{align*}
Note that the above calculations are still valid in the case $i=2$.
In this case, using the conventions of Section \ref{Settings}, $W_{i-1} \cdots W_{2} = I_{d_{1}}$ and $Z_{i-1} \cdots Z_2 = I_{d_{2}-r}$.\\
This proves \eqref{T_i J_2 J_3 simplified}.
\\

 \textbf{Consider finally the case $i=1$.} \\
Substituting \eqref{W_H...W_p simplified} in \eqref{T_i}, we have
\begin{align*}
    T_1 &= (W_H \cdots W_2)W'_1X \\
    &= \left[U_{\Ss} \ ,\ 0 \right] W'_1X \\
    &= U_{\Ss} (W'_1)_{1:r,.} X \;.
\end{align*}
This proves \eqref{T_1 simplified}. \\
Note that, using the conventions of Section \ref{Settings}, the proof still holds for $r=0$.
In this case, $T_i = 0,  \forall i$.\\
This concludes the proof.

\subsubsection{Proof of Lemma \ref{simplif all T_ij}\label{proof simplif all T_ij}}
Let $\Wbf = (W_H, \cdots , W_1)$ be a tightened first-order critical point satisfying the hypotheses of Proposition \ref{prop 4}, and $r$, $\Ss$, $Q$, $(Z_h)_{h=1..H}$ defined as in Proposition \ref{prop 4}.
Since $\Wbf$ satisfies the hypotheses of Proposition \ref{prop 4}, we are going to use all the equations \eqref{W_H simplified}, \eqref{W_1 simplified}, \eqref{W_h simplified} and \eqref{W_H...W_2 simplified} defined by these hypotheses and \eqref{W_H...W_p simplified}, \eqref{W_i-1...W_2 simplified}, \eqref{hammou} and \eqref{W_H-1...W_k+1 simplified} of Lemma \ref{E}.\\
Let $\Wbf' = (W'_H, \ldots, W'_1)$ and $(i,j) \in \llbracket 1,H \rrbracket^2$, with $i>j$.
Recall that $T_{i,j}$ is defined in \eqref{T_ij} and $J_1 = \llbracket p,H-1 \rrbracket$, $J_2 = \llbracket q+1 , p-1 \rrbracket$, $J_3 = \llbracket 2,q \rrbracket$, where $p$ and $q$ are defined as in Lemma \ref{E}.\\

\textbf{Consider the case $i \in \{H\} \cup J_1$ and $j \in J_1 \cup J_2$ with $i>j$.} \\
Applying Lemma \ref{simplif cross-product} to \eqref{T_ij} and using \eqref{W_h simplified} and \eqref{W_1 simplified}, we obtain
\begin{align*}
    T_{i,j} &= \left\langle  W_H \cdots W_{i+1} W'_i W_{i-1} \cdots W_{j+1} W'_j W_{j-1} \cdots W_1 X \ , \ W_H \cdots W_1 X - Y \right\rangle \\
    &= \left\langle  W_H \cdots W_{i+1} W'_i W_{i-1} \cdots W_{j+1} W'_j W_{j-1} \cdots W_1 \ , \ -U_{Q} U_{Q}^T \Sigma_{YX} \right\rangle \\
    &= - tr\left(W_H \cdots W_{i+1} W'_i W_{i-1} \cdots W_{j+1} W'_j W_{j-1} \cdots W_1 \Sigma_{XY} U_Q U_Q^T\right) \\
    &= - tr\left((W_H \cdots W_{i+1} W'_i W_{i-1} \cdots W_{j+1} W'_j \left[
    \begin{array}{c}
    U_{\Ss}^T\Sigma U_Q \\
    Z_{j-1} \cdots Z_2 Z_1 \Sigma_{XY} U_Q
    \end{array}
    \right] U_Q^T\right) \;.
\end{align*}
Using Lemma \ref{U_S sigma U_Q = 0} and since $j \geq q+1$, using \eqref{hammou},  we obtain
$$ T_{i,j} = 0.$$
This proves \eqref{T_Hj j in J_1 cup J_2} and \eqref{T_ij i in J_1, j in J_1 cup J_2}. \\

\textbf{Consider now the case $i=H$ and $j \in J_3$.} \\
Applying Lemma \eqref{simplif cross-product} to \eqref{T_ij} and using \eqref{W_H-1...W_k+1 simplified}, \eqref{W_h simplified} and \eqref{W_1 simplified}, we obtain
\begin{align*}
    T_{H,j} &= \left\langle  W'_H W_{H-1} \cdots W_{j+1} W'_j W_{j-1} \cdots W_1 X \ , \ W_H \cdots W_1 X - Y \right\rangle \\
    &= \left\langle  W'_H W_{H-1} \cdots W_{j+1} W'_j W_{j-1} \cdots W_1 \ , \ -U_{Q} U_{Q}^T \Sigma_{YX} \right\rangle \\
    &= \left\langle W'_H \left[
    \begin{array}{c  c}
    I_r & 0 \\
    0 & 0
    \end{array}
    \right] W'_j \left[
    \begin{array}{c}
    U_{\Ss}^T\Sigma_{YX}\Sigma_{XX}^{-1} \\
    Z_{j-1} \cdots Z_2 Z_1 
    \end{array}
    \right] \ , \  U_Q U_Q^T \Sigma_{YX} \right\rangle \\
    &= - tr\left( W'_H \left[
    \begin{array}{c  c}
    I_r & 0 \\
    0 & 0
    \end{array}
    \right] W'_j \left[
    \begin{array}{c}
    U_{\Ss}^T\Sigma U_Q U_Q^T\\
    Z_{j-1} \cdots Z_2 Z_1 \Sigma_{XY} U_Q U_Q^T
    \end{array}
    \right] \right) \;.
\end{align*}
Using Lemma \ref{Sigma_XY U_Q}, Lemma \ref{U_S sigma U_Q = 0} and the cyclic property of the trace, we have
\begin{align*}
    T_{H,j} &= -tr\left((W'_H)_{.,1:r} (W'_j)_{1:r,.} \left[
    \begin{array}{c}
    0 \\
    Z_{j-1} \cdots Z_2 Z_1 X V_Q \Delta^{(Q)} U_Q^T
    \end{array}
    \right] \right) \\
    &= -tr\left((W'_H)_{.,1:r} (W'_j)_{1:r,r+1:d_{j-1}} Z_{j-1} \cdots Z_2 Z_1 X V_Q \Delta^{(Q)} U_Q^T\right) \\
    &= -tr\left(\Delta^{(Q)} U_Q^T (W'_H)_{.,1:r} (W'_j)_{1:r,r+1:d_{j-1}} Z_{j-1} \cdots Z_2 Z_1 X V_Q\right) \\
    &= - \left\langle  \Delta^{(Q)} U_Q^T (W'_H)_{.,1:r}\ ,\ \left((W'_j)_{1:r,r+1:d_{j-1}} Z_{j-1} \cdots Z_2 Z_1 X V_Q\right)^T \right\rangle \;.
\end{align*}
This proves \eqref{T_Hj simplified}. \\

\textbf{Consider now the case $i=H$ and $j=1$.} \\
Applying Lemma \ref{simplif cross-product} to \eqref{T_ij} and using \eqref{W_H-1...W_k+1 simplified} and Lemma \ref{Sigma_XY U_Q}, we obtain
\begin{align*}
    T_{H,1} &= \left\langle  W'_H W_{H-1} \cdots W_2 W'_1 X \ , \ W_H \cdots W_1 X - Y \right\rangle \\
    &= \left\langle  W'_H W_{H-1} \cdots W_2 W'_1 \ , \ -U_{Q} U_{Q}^T \Sigma_{YX} \right\rangle \\
    &= - \left\langle  W'_H \left[
    \begin{array}{c  c}
    I_r & 0 \\
    0 & 0
    \end{array}
    \right]W'_1 \ , \ U_Q(X V_Q \Delta^{(Q)})^T \right\rangle \\
    &= - \left\langle  (W'_H)_{.,1:r} (W'_1)_{1:r,.} \ , \ U_Q \Delta^{(Q)} V_Q^T X^T \right\rangle \\
    &= - \left\langle  \Delta^{(Q)} U_Q^T (W'_H)_{.,1:r} \ , \ V_Q^T X^T \left( (W'_1)_{1:r,.} \right)^T \right\rangle \\
    &= -\left\langle  \Delta^{(Q)} U_Q^T (W'_H)_{.,1:r}  \ , \ \left((W'_1)_{1:r,.} X V_Q \right)^T \right\rangle \;.
\end{align*}
This proves \eqref{T_H1 simplified}. \\

\textbf{Consider now the case $i \in J_1$ and $j \in J_3$.} \\
Applying Lemma \ref{simplif cross-product} to \eqref{T_ij} and using \eqref{W_H simplified}, \eqref{W_h simplified} and \eqref{W_1 simplified} , we obtain
\begin{align*}
    T_{i,j} &= \left\langle  W_H \cdots W_{i+1}W'_{i}W_{i-1} \cdots W_{j+1}W'_{j}W_{j-1} \cdots W_1 X \ , \ W_H \cdots W_1 X - Y \right\rangle \\
    &= \left\langle  W_H \cdots W_{i+1}W'_{i}W_{i-1} \cdots W_{j+1}W'_{j}W_{j-1} \cdots W_1 \ , \ -U_{Q} U_{Q}^T \Sigma_{YX} \right\rangle \\
    & = - \left\langle \left[ U_{\Ss} \ , \ U_Q Z_H Z_{H-1} \cdots Z_{i+1} \right] W'_i W_{i-1} \cdots W_{j+1} W'_j \left[
    \begin{array}{c}
    U_{\Ss}^T \Sigma_{YX} \Sigma_{XX}^{-1} \\
    Z_{j-1} \cdots Z_2 Z_1 
    \end{array}
    \right] \ , \ U_Q U_Q^T \Sigma_{XY}  \right\rangle \\
    & = -tr\left(\left[ U_{\Ss} \ , \ U_Q Z_H Z_{H-1} \cdots Z_{i+1} \right] W'_i W_{i-1} \cdots W_{j+1} W'_j \left[
    \begin{array}{c}
    U_{\Ss}^T\Sigma \\
    Z_{j-1} \cdots Z_2 Z_1 \Sigma_{XY}
    \end{array}
    \right] U_Q U_Q^T\right) \\
    &= -tr\left(\left[ U_Q^T U_{\Ss} \ , \ U_Q^TU_Q Z_H Z_{H-1} \cdots Z_{i+1} \right] W'_i W_{i-1} \cdots W_{j+1} W'_j \left[
    \begin{array}{c}
    U_{\Ss}^T\Sigma U_Q \\
    Z_{j-1} \cdots Z_2 Z_1 \Sigma_{XY}U_Q
    \end{array}
    \right] \right) \;.
\end{align*}
Using Lemma \ref{properties of orhtogonality} and Lemma \ref{U_S sigma U_Q = 0}, we have
\begin{align}
    T_{i,j} &= -tr\left(\left[ 0 \ , \  Z_H Z_{H-1} \cdots Z_{i+1} \right] W'_i W_{i-1} \cdots W_{j+1} W'_j \left[
    \begin{array}{c}
    0 \\
    Z_{j-1} \cdots Z_2 Z_1 \Sigma_{XY}U_Q
    \end{array}
    \right] \right) \nonumber \\
    &= -tr\left(Z_H Z_{H-1} \cdots Z_{i+1} (W'_{i})_{r+1:d_{i},.}W_{i-1} \cdots W_{j+1} (W'_j)_{.,r+1:d_{j-1}} Z_{j-1} \cdots Z_2 Z_1 \Sigma_{XY} U_Q\right) \;. \label{middle T_ij}
\end{align}
Here, since $\Wbf$ is tightened, taking the tightened pivot $(i,j)$ we have two possible cases: either $\rk(W_{i-1} \cdots W_{j+1}) = r$ or $\rk(W_{j-1} \cdots W_1 \Sigma_{XY}W_H \cdots W_{i+1})=r \;.$ We treat the two cases separately. \\
In the first case, using \eqref{W_h simplified} we have
\begin{align*}
    W_{i-1} \cdots W_{j+1} &= \left[
    \begin{array}{c  c}
    I_r & 0 \\
    0 & Z_{i-1}
    \end{array}
    \right] \cdots \left[
    \begin{array}{c  c}
    I_r & 0 \\
    0 & Z_{j+1}
    \end{array}
    \right] \\
    &= \left[
    \begin{array}{c  c}
    I_r & 0 \\
    0 & Z_{i-1} \cdots Z_{j+1}
    \end{array}
    \right] \;.
\end{align*}
Hence, $\rk(W_{i-1} \cdots W_{j+1}) = r$ implies $Z_{i-1} \cdots Z_{j+1} = 0$ and we conclude that $$ W_{i-1} \cdots W_{j+1} = \left[
\begin{array}{c  c}
I_r & 0 \\
0 & 0
\end{array}
\right] \;.$$
Then, using this last equality, \eqref{middle T_ij} becomes
\begin{align}
    T_{i,j} &= -tr\left(Z_H Z_{H-1} \cdots Z_{i+1} (W'_{i})_{r+1:d_{i},.} \left[
    \begin{array}{c  c}
    I_r & 0 \\
    0 & 0
    \end{array}
    \right] (W'_j)_{.,r+1:d_{j-1}} Z_{j-1} \cdots Z_2 Z_1 \Sigma_{XY} U_Q\right) \nonumber \\
    &= -tr\left(Z_H Z_{H-1} \cdots Z_{i+1} (W'_{i})_{r+1:d_{i},1:r} (W'_j)_{1:r,r+1:d_{j-1}} Z_{j-1} \cdots Z_2 Z_1 \Sigma_{XY} U_Q\right) \;. \label{etoile}
\end{align}
In the second case, we have $\rk(W_{j-1} \cdots W_1 \Sigma_{XY}W_H \cdots W_{i+1})=r$.
Let us prove that \eqref{etoile} also holds in this case.
Using \eqref{W_h simplified}, \eqref{W_1 simplified}, \eqref{W_H simplified}, Lemma \ref{U_S sigma U_Q = 0} and $\Ss = \llbracket 1,r \rrbracket$, we have
\begin{align*}
    & W_{j-1} \cdots W_1 \Sigma_{XY}W_H \cdots W_{i+1} \\ &= \left[
    \begin{array}{c}
    U_{\Ss}^T\Sigma \\
    Z_{j-1} \cdots Z_2 Z_1 \Sigma_{XY}
    \end{array}
    \right] \left[ U_{\Ss} \ , \ U_Q Z_H Z_{H-1} \cdots Z_{i+1} \right] \\
    &= \left[
    \begin{array}{c  c}
    U_{\Ss}^T\Sigma U_{\Ss} & U_{\Ss}^T\Sigma U_Q Z_H Z_{H-1} \cdots Z_{i+1}\\
    Z_{j-1} \cdots Z_2 Z_1 \Sigma_{XY} U_{\Ss} & Z_{j-1} \cdots Z_2 Z_1 \Sigma_{XY}U_Q Z_H Z_{H-1} \cdots Z_{i+1}
    \end{array}
    \right] \\
    &= \left[
    \begin{array}{c  c}
    diag(\lambda_1, \cdots ,\lambda_r) & 0 \\
    Z_{j-1} \cdots Z_2 Z_1 \Sigma_{XY} U_{\Ss} & Z_{j-1} \cdots Z_2 Z_1 \Sigma_{XY}U_Q Z_H Z_{H-1} \cdots Z_{i+1}
    \end{array}
    \right] \;.
\end{align*}
Therefore since $\rk(W_{j-1} \cdots W_1 \Sigma_{XY}W_H \cdots W_{i+1})=r$ and for all $i \in \llbracket 1,r \rrbracket$, $\lambda_i \neq 0$, we must have 
\begin{align}
    Z_{j-1} \cdots Z_2 Z_1 \Sigma_{XY}U_Q Z_H Z_{H-1} \cdots Z_{i+1} = 0 \ . \label{2e cas T_ij J_1 J_3}
\end{align}
Using the above equation, and the cyclic property of the trace, \eqref{middle T_ij} becomes
\begin{align*}
    T_{i,j} &= -tr\left(Z_H Z_{H-1} \cdots Z_{i+1} (W'_{i})_{r+1:d_{i},.}W_{i-1} \cdots W_{j+1} (W'_j)_{.,r+1:d_{j-1}} Z_{j-1} \cdots Z_2 Z_1 \Sigma_{XY} U_Q\right) \\
    &= -tr\left(Z_{j-1} \cdots Z_2 Z_1 \Sigma_{XY} U_Q Z_H Z_{H-1} \cdots Z_{i+1} (W'_{i})_{r+1:d_{i},.}W_{i-1} \cdots W_{j+1} (W'_j)_{.,r+1:d_{j-1}}\right) \\
    &= 0 \;.
\end{align*}
We can use \eqref{2e cas T_ij J_1 J_3} again to write the equation $T_{i,j} = 0$ in the format of equation \eqref{etoile}. Indeed, we have
\begin{align*}
    & -tr\left(Z_H Z_{H-1} \cdots Z_{i+1} (W'_{i})_{r+1:d_{i},1:r} (W'_j)_{1:r,r+1:d_{j-1}} Z_{j-1} \cdots Z_2 Z_1 \Sigma_{XY} U_Q\right) \\
    &= -tr\left(Z_{j-1} \cdots Z_2 Z_1 \Sigma_{XY} U_Q Z_H Z_{H-1} \cdots Z_{i+1} (W'_{i})_{r+1:d_{i},1:r} (W'_j)_{1:r,r+1:d_{j-1}}\right) \\
    &=0 \\
    &= T_{i,j} \;.
\end{align*}
Therefore, in both cases we have 
\begin{align*}
    T_{i,j} &= -tr\left(Z_H Z_{H-1} \cdots Z_{i+1} (W'_{i})_{r+1:d_{i},1:r} (W'_j)_{1:r,r+1:d_{j-1}} Z_{j-1} \cdots Z_2 Z_1 \Sigma_{XY} U_Q\right) \;.
\end{align*}
Using Lemma \ref{Sigma_XY U_Q}, it becomes
\begin{align*}
    T_{i,j} &= -tr\left(Z_H Z_{H-1} \cdots Z_{i+1} (W'_{i})_{r+1:d_{i},1:r} (W'_j)_{1:r,r+1:d_{j-1}} Z_{j-1} \cdots Z_2 Z_1 X V_Q \Delta^{(Q)}\right) \\
    &= -tr\left(\Delta^{(Q)} Z_H Z_{H-1} \cdots Z_{i+1} (W'_{i})_{r+1:d_{i},1:r} (W'_j)_{1:r,r+1:d_{j-1}} Z_{j-1} \cdots Z_2 Z_1 X V_Q\right) \\
    &= - \left\langle  \Delta^{(Q)} Z_H Z_{H-1} \cdots Z_{i+1} (W'_{i})_{r+1:d_{i},1:r} \ , \ \left((W'_j)_{1:r,r+1:d_{j-1}} Z_{j-1} \cdots Z_2 Z_1 X V_Q\right)^T \right\rangle \;.
\end{align*}
This proves \eqref{T_ij J_1 x J_3 simplified}. \\

\textbf{Consider now the case $i \in J_1$ and $j=1$.} \\
Using Lemma \ref{simplif cross-product} to simplify \eqref{T_ij}, we have
\begin{align*}
    T_{i,1} &= \left\langle  W_H \cdots W_{i+1}W'_{i}W_{i-1} \cdots W_2 W'_1 X \ , \ W_H \cdots W_1 X - Y \right\rangle \\
    &= \left\langle  W_H \cdots W_{i+1}W'_{i}W_{i-1} \cdots W_2 W'_1 \ , \ -U_{Q} U_{Q}^T \Sigma_{YX} \right\rangle \;.
\end{align*}
Using Lemma \ref{Sigma_XY U_Q} and substituting \eqref{W_H simplified}, \eqref{W_h simplified}, and since $i \geq p$, using \eqref{W_i-1...W_2 simplified} , this becomes
\begin{align*}
    T_{i,1} &= -\left\langle  \left[ U_{\Ss} \ , \ U_Q Z_H Z_{H-1} \cdots Z_{i+1} \right] W'_i \left[
    \begin{array}{c  c}
    I_r & 0 \\
    0 & 0
    \end{array}
    \right] W'_1 \ , \ U_Q (X V_Q \Delta^{(Q)})^T \right\rangle \\
    &= - \left\langle  \Delta^{(Q)} \left[ U_Q^T U_{\Ss} \ , \ U_Q^T U_Q Z_H Z_{H-1} \cdots Z_{i+1} \right] (W'_i)_{.,1:r}(W'_1)_{1:r,.} \ , \ (X V_Q)^T \right\rangle \;.
\end{align*}
Using Lemma \ref{properties of orhtogonality}, it becomes
\begin{align*}
    T_{i,1} &= - \left\langle  \Delta^{(Q)}  \left[ 0 \ , \ Z_H Z_{H-1} \cdots Z_{i+1} \right] (W'_i)_{.,1:r}(W'_1)_{1:r,.} \ , \ (X V_Q)^T \right\rangle \\
    &= - \left\langle \Delta^{(Q)} Z_H Z_{H-1} \cdots Z_{i+1} (W'_{i})_{r+1:d_{i},1:r} \ , \ ((W'_1)_{1:r,.} X V_Q)^T \right\rangle \;.
\end{align*}
This proves \eqref{T_i1 simplified}. \\

 \textbf{Consider now the case $i \in J_2 \cup J_3 = \llbracket2,p-1 \rrbracket$ and $j<i$}. \\
Applying Lemma \ref{simplif cross-product} to \eqref{T_ij} and , since $i<p$, using \eqref{W_H...W_p simplified}, we obtain
\begin{align*}
    T_{i,j} &= \left\langle  W_H \cdots W_{i+1}W'_{i}W_{i-1} \cdots W_{j+1}W'_{j}W_{j-1} \cdots W_1 X \ , \ W_H \cdots W_1 X - Y \right\rangle \\
    &= \left\langle  W_H \cdots W_{i+1}W'_{i}W_{i-1} \cdots W_{j+1}W'_{j}W_{j-1} \cdots W_1 \ , \ -U_{Q} U_{Q}^T \Sigma_{YX} \right\rangle \\
    &= -tr( [ U_{\Ss} \ , \ 0] W'_{i}W_{i-1} \cdots W_{j+1}W'_{j}W_{j-1} \cdots W_1 \Sigma_{XY} U_Q U_Q^T)
\end{align*}
The cyclic property of the trace and Lemma \ref{properties of orhtogonality} lead to
\begin{align*}
    T_{i,j} &= -tr( [ U_Q^T U_{\Ss} \ , \ 0] W'_{i}W_{i-1} \cdots W_{j+1}W'_{j}W_{j-1} \cdots W_1 \Sigma_{XY} U_Q) \\
    &= 0 \;.
\end{align*}
This proves \eqref{T_ij, i in J_2 cup J_3} and concludes the proof. \\
Note that, with the convention of Section \ref{Settings}, the proof still holds for $r=0$.
In this case, $T_{i,j} = 0,  \forall i>j$.

\subsection{Proof of Proposition \ref{main prop non strict saddles}} \label{Proof of main prop non strict saddles}

Let $\Wbf = (W_H, \ldots , W_1)$ be a tightened first-order critical point associated with $\Ss = \llbracket 1,r \rrbracket$ with $r<r_{max}$.
Then, using Proposition \ref{simplif mat} there exist invertible matrices $D_{H-1} \in \mathbb{R}^{d_{H-1} \times d_{H-1}}, \ldots ,D_1 \in \mathbb{R}^{d_1 \times d_1}$ and matrices $Z_H \in \mathbb{R}^{(d_y-r) \times (d_{H-1}-r)}$, $Z_1 \in \mathbb{R}^{(d_1-r) \times d_x}$ and $Z_h \in \mathbb{R}^{(d_h-r) \times (d_{h-1}-r)}$ for $h \in \llbracket 2 , H-1 \rrbracket$ such that if we denote $\widetilde{W}_H = W_H D_{H-1}$ , $\widetilde{W}_1 = D_1^{-1} W_1$ and  $\widetilde{W}_h = D_{h}^{-1} W_h D_{h-1}$ for all $h \in \llbracket 2 , H-1  \rrbracket$, and $\widetilde{\Wbf} = (\widetilde{W}_H, \ldots ,\widetilde{W}_1)$, then
\begin{align*}
        \widetilde{W}_H &= [U_{\Ss} , U_Q Z_H]  \\
        \widetilde{W}_1 &= \begin{bmatrix}
        U_{\Ss}^T\Sigma_{YX}\Sigma_{XX}^{-1} \\ 
        Z_1
        \end{bmatrix}  \\
         \widetilde{W}_h &= \left[
        \begin{array}{c  c}
        I_r & 0 \\
        0 & Z_h
        \end{array}
        \right] \quad \forall h \in \llbracket 2 , H-1 \rrbracket  \\
        \widetilde{W}_H \cdots \widetilde{W}_2 &= \left[U_{\Ss} , 0 \right] \;.
\end{align*}
Then, due to Lemma \ref{same nature}, and since $\Wbf$ is a first-order critical point, we have that $\widetilde{\Wbf}$ is a first-order critical point.
We also have $ \widetilde{W}_H \cdots \widetilde{W}_1 = W_H \cdots W_1$.
Hence, according to Proposition \ref{global map and critical values} $\widetilde{\Wbf}$ is also associated with $\Ss$. \\
Since $\Wbf$ is tightened and multiplication by invertible matrices does not change the rank, $\widetilde{\Wbf}$ is also tightened.
Hence, $\widetilde{\Wbf}$ satisfies the hypotheses of Proposition \ref{prop 4} and therefore is a second-order critical point.
Finally, using Lemma \ref{same nature}, we conclude that $\Wbf$ is a second-order critical point.
Since $r<r_{max}$ and $\Sigma$ is invertible (Lemma \ref{sigma invertible}), using Proposition \ref{global map and critical values}, we have 
$$L(\Wbf) = \tr(\Sigma_{YY}) - \sum_{i=1}^r \lambda_i > \tr(\Sigma_{YY}) - \sum_{i=1}^{r_{max}} \lambda_i \;.$$
Therefore, $\Wbf$ is not a global minimizer, hence $\Wbf$ is a non-strict saddle point.

\section{A Simple Illustrative Experiment}\label{sec experiments}

Next we provide more details on the experiment whose results were plotted in Figures~\ref{fig:eg loss near saddle point} and  \ref{fig:histo escape points}. The goal is to illustrate the behavior of the ADAM optimizer in the vicinity of strict or non-strict saddle points.

\textbf{Experimental setting.} We optimize a linear neural network starting in the vicinity either of a strict saddle point ($10000$ runs in total) or of a non-strict saddle point ($10000$ runs in total). For each run, the setting is the following:
\begin{itemize}
    \item Network architecture: $d_x=10$, $d_y=4$, $H=5$ and $d_4=d_3=d_2=d_1=10$.
    \item Data construction: $m=100$ i.i.d. data points $(x_1,y_1),\ldots,(x_m,y_m) \in \mathbb{R}^{d_x} \times \mathbb{R}^{d_y}$ such that, for all $i=1,\ldots,m$, the points $x_i$ and $y_i$ are drawn independently at random from the Gaussian distributions $\mathcal{N}(0,\,I_{d_x})$ and $\mathcal{N}(0,\,I_{d_y})$ respectively.
    \item Initial iterate: we define it as $(W_1,\ldots, W_H) = \Wbf^{cp} + (V_1,\ldots, V_H) $, for a critical point $\Wbf^{cp}$ (defined later) and a random perturbation $(V_1,\ldots, V_H)$ whose components $(V_h)_{i,j}$ are drawn independently from the distributions $\mathcal{N}(0,\sigma_h^2)$, with $\sigma_h=0.1 \frac{\|W_h^{cp}\|_F}{\sqrt{d_{h-1}d_h}}$. The critical point $\Wbf^{cp}$ is  defined as in \eqref{eqpourex} in Appendix~\ref{proof of existence of tightened and non tightened critical points}, for $r=2$ ($\Ss=\{1,2\}$) and
    \[
    Z_h = \left\{ \begin{array}{ll}
        I_{d_h-2} & \textrm{for all $h \in \llbracket2,4 \rrbracket$, for runs starting at a strict saddle point;} \\
        0_{(d_h-2) \times (d_{h}-2)} & \textrm{for all $h \in \llbracket2,4 \rrbracket$, for runs starting at a non-strict saddle point.}
    \end{array} \right.
    \]
 
    Since $d_4=d_3=d_2=d_1$, note that the sizes of the above matrices $Z_h$ are consistent with \eqref{eqpourex}. As explained in Appendix \ref{proof of existence of tightened and non tightened critical points}, when $Z_h=I_{d_h-2}$ for all $h \in \llbracket2,4 \rrbracket$, the critical point $\Wbf^{cp}$ is non-tightened and therefore Theorem \ref{main theorem} guarantees that it is a strict saddle point. Similarly, when $Z_h=0_{(d_h-2)\times (d_{h}-2)}$ for all $h \in \llbracket2,4 \rrbracket$, the critical point $\Wbf^{cp}$ is tightened and Theorem \ref{main theorem} guarantees that it is a non-strict saddle point.
    \item Optimizer: we use the ADAM optimizer of the Keras library, with the default parameters. 
\end{itemize}

\textbf{Observations.}
Figure \ref{fig:eg loss near saddle point} in Section~\ref{sec:intro-importanceordre2} shows the evolution of the loss along the optimization process for two representative runs (initialization near a strict or a non-strict saddle point). We can see that, when initialized in the vicinity of the strict saddle point, ADAM rapidly decreases below the initial value $L(\Wbf^{cp})$. On the contrary, ADAM needs many epochs to exit the plateau at the critical value of the non-strict saddle point.

In order to assess the importance of this phenomenon, we repeated the above experiment $10000$ times for both strict saddle points and non-strict saddle points. For each run, we define and compute the {\em escape epoch} as the first epoch such that $ L(\Wbf) < L(\Wbf^{cp})-\frac{\lambda_3}{2}$ (the average of the critical values associated with $\Ss=\{1,2\}$ and $\Ss'=\{1,2,3\}$). On Figure~\ref{fig:histo escape points} (Section~\ref{sec:intro-importanceordre2}) the histograms of the escape epoch are displayed separately for runs corresponding to strict saddle points (in red) or non-strict saddle points (in blue). We can see that, while ADAM quickly escapes from the vicinity of the strict saddle points, it takes many more epochs to escape from the vicinity of the non-strict saddle points. In the last case, the plateau can easily be confused with a global minimum.

\bibliography{references.bib}

\end{document}